\documentclass[11pt]{amsart}
\usepackage{amssymb,amsthm,amsmath,amstext}
\usepackage{mathrsfs}  % allows nice script font for sheaves
\usepackage{xcolor}
\usepackage{multirow}
\usepackage{enumitem}
\usepackage{cases}
\usepackage{graphicx}
\usepackage{bm}
\usepackage{mathtools}
\usepackage{tikz}
\usepackage{tikz-cd}
\tikzset{every picture/.style={line width=0.75pt}} %set default line width to 0.75pt  
\usetikzlibrary{patterns}
\usepackage{caption}
\usepackage{subcaption}
\usepackage{mathdots}
\usepackage{yhmath}
\usepackage{cancel}
\usepackage{color}
\usepackage{siunitx}
\usepackage{array}
\usepackage{tabularx}
\usepackage{extarrows}
\usepackage{booktabs}
\usetikzlibrary{fadings}
\usetikzlibrary{patterns}
\usetikzlibrary{shadows.blur}
\usetikzlibrary{shapes}
\usepackage{setspace}

\usepackage{float}
\usepackage[left=1.0in,top=0.95in,right=1.0in,bottom=0.95in]{geometry}

\usepackage[all]{xy}

\newcommand{\Z}{\mathbb{Z}}

\newcommand{\PP}{\mathbb{P}}

\newcommand{\M}{\mathcal{M}}

\def\OO{\mathcal{O}}
\DeclareMathOperator{\inv}{\operatorname{inv}}

\DeclareMathOperator{\Pic}{\operatorname{Pic}}
\DeclareMathOperator{\codim}{codim}

\usepackage{hyperref}
\hypersetup{pdftitle={Hurwitz-BN-Loci},pdfauthor={Andrei Bu and Richard Haburcak}} 
\hypersetup{colorlinks=true,linkcolor=blue,anchorcolor=blue,citecolor=blue}

\usepackage[nameinlink]{cleveref}

\newtheorem{theorem}{Theorem}[section]
\newtheorem{setup}{Setup}[section]
\newtheorem{lemma}[theorem]{Lemma}
\newtheorem{prop}[theorem]{Proposition}
\newtheorem{cor}[theorem]{Corollary}

\newtheorem*{notation}{Notation}

\newtheorem{theoremintro}{Theorem}

\newtheorem{claim}{Claim}

\theoremstyle{definition}
\newtheorem{defn}[theorem]{Definition}

\theoremstyle{definition}
\newtheorem{remark}[theorem]{Remark}
\newtheorem{example}[theorem]{Example}

\newcommand{\cHTot}{\mathcal{H}^{\mathrm{tot}}_{g,k}}
\newcommand{\hhTot}{\overline{\mathcal{H}}^{\mathrm{tot}}_{g,k}}
\def\elhur{[\pi\colon C\rightarrow \mathbb{P}^1, p]}
\DeclareMathOperator{\rank}{rank}

\DeclareMathOperator{\Id}{Id}
\DeclareMathOperator{\pr}{pr}

\newcommand{\duckicon}{%
	\includegraphics[width=0.50cm, trim=200 170 230 180, clip]{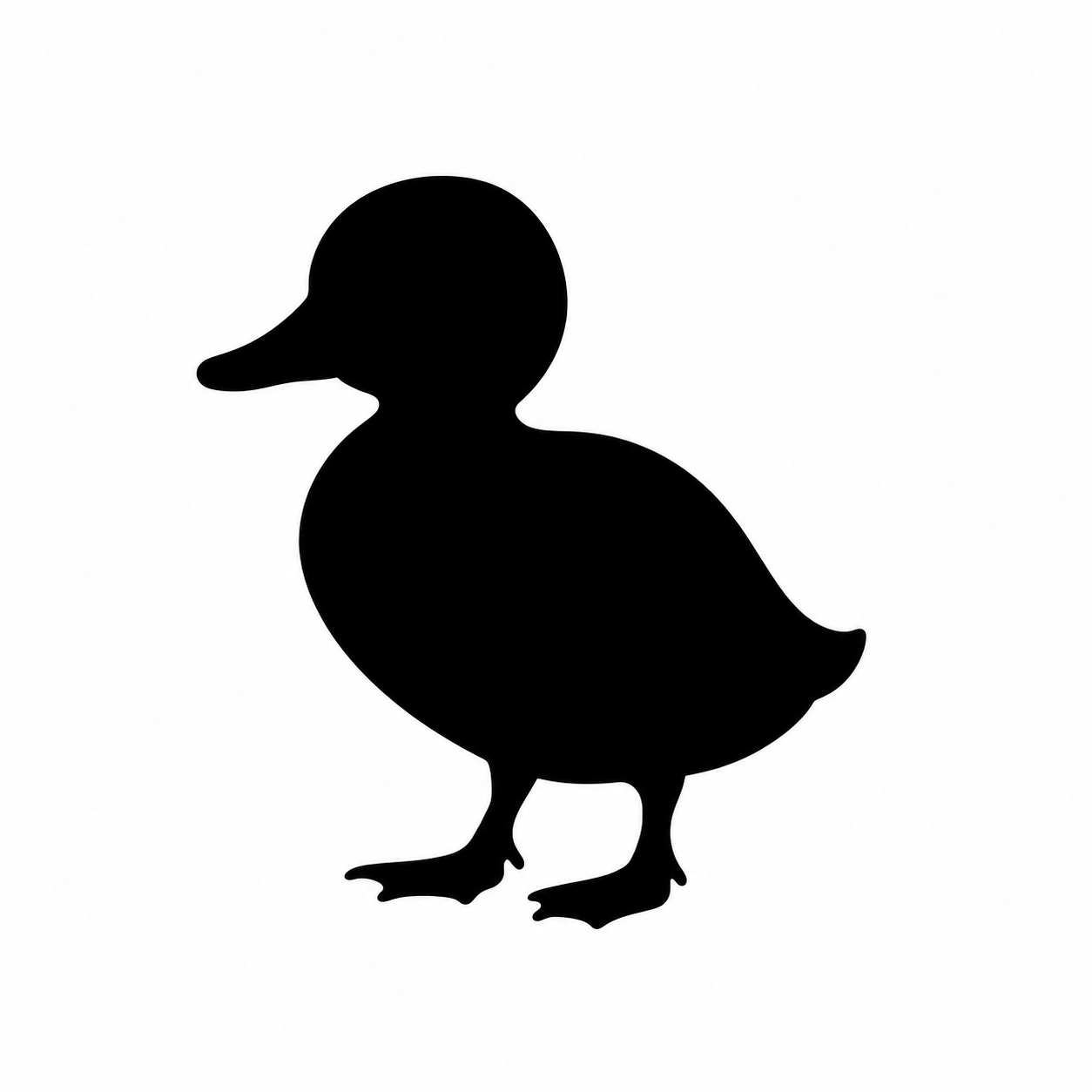}%
}

\newcommand{\duckiconTwo}{%
	\includegraphics[
	width=0.72cm,
	trim=120 100 120 100,
	clip
	]{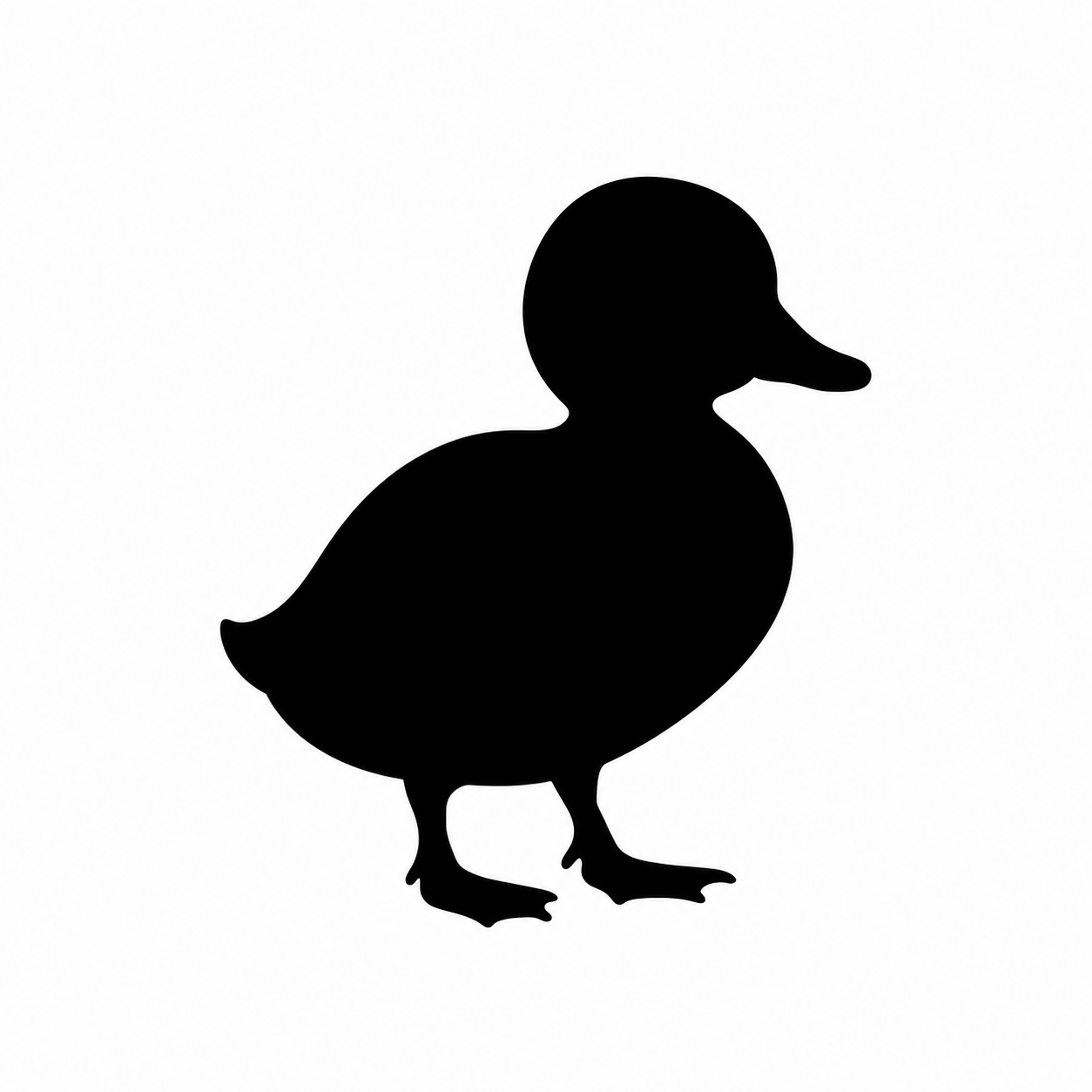}%
}

\title{Pointed Hurwitz--Brill--Noether theory} 

\author{Andrei Bud}
\address{Bud, Institut Mathématique de Toulouse; Université de Toulouse; CNRS\\
	118 route de Narbonne, 31062 Toulouse, France}
\email{andreibud95@protonmail.com}

\author{Richard Haburcak}
\address{Haburcak, Department of Mathematics\\%
	The Ohio State University\\%
	100 Math Tower\\%
	Columbus, OH 43210}
\email{haburcak.1@osu.edu}

\begin{document}
	\thispagestyle{empty}
	\vspace*{-.5cm}
	\begin{abstract} 		
		We study the geometry of Hurwitz--Brill--Noether spaces with vanishing conditions at a totally ramified point. Degeneration to a curve with an elliptic tail simplifies the combinatorics of Larson, Larson--Larson--Vogt, and Aggarwal, giving inductive proofs of Aggarwal's results in the one-point setting, and yields an efficient recursion for the algebraic Euler characteristics. Finally, when the expected dimension is negative, we give a numerical criterion ensuring the universal Hurwitz--Brill--Noether locus admits a component of expected relative dimension, generically finite over its image in the Hurwitz space, analogous to the existing results of Pflueger and Teixidor i Bigas for classical Brill--Noether loci.
	\end{abstract}
	\maketitle
	
	\section*{Introduction}
	
	Brill--Noether theory studies the linear series on algebraic curves. For a general smooth curve $C$ of genus $g$, the Brill--Noether--Petri theorem describes the geometry of the loci \[W^r_d(C) = \{L\in\Pic(C) ~\mid~ h^0(C,L)\geq r+1\}\] in terms of the Brill--Noether number $\rho(g,r,d)=g-(r+1)(g-d+r)$. Namely, when $\rho(g,r,d)<0$, then $W^r_d(C)$ is empty, and when $\rho(g,r,d)\geq 0$, then $W^r_d(C)$ has the expected dimension $\rho(g,r,d)$. Furthermore, when $\rho(g,r,d)>0$ the Brill-Noether locus $W^r_d(C)$ is irreducible and smooth away from $W^{r+1}_d(C)$, see~\cite{GriffithsHarris1980,Gieseker1982,FultonLazarsfeld1981}. A refined Brill--Noether theory seeks an analogous description for curves that are special in moduli, in particular for curves with unexpected maps to projective space, for example for curves with a low degree map to $\PP^1$. For such curves with a degree $k$ map $\pi:C\to \PP^1$, prescribing the splitting type of 
	$$\pi_{\ast}L=\OO_{\PP^1}(\Vec{e})\coloneqq \OO_{\PP^1}(e_0)\oplus\cdots \oplus \OO_{\PP^1}(e_{k-1})$$ 
	gives the Hurwitz--Brill--Noether loci $W^{\vec{e}}(C)$, and has been useful in understanding the Brill--Noether theory of curves of a fixed gonality~\cite{Larson2021,Larson2Vogt,cook-powell_jensen,CookPowellJensen2022Tropical}. When $\pi:C\to \PP^1$ is totally ramified at a single point $p$, prescribing the vanishing of sections at $p$ gives a further refinement. More generally, in~\cite{Pflueger2026Transmission}, Pflueger considered the refinement given by prescribing the order of vanishing at two points, $p,q\in C$, developing the theory of transmission loci, where vanishing is prescribed by an element of the affine symmetric group, which is further developed in~\cite{Agg26} where analogues of classical Brill--Noether results are proven for transmission loci. In this paper, we give an inductive treatment for the case of one point of ramification, giving new proofs of the results in~\cite{Agg26} in the case of pointed Hurwitz--Brill--Noether loci $W^{\vec{e}}(\pi,p)$, which enables us to recursively compute the Euler characteristic of $W^{\vec{e}}(\pi,p)$ and find components of the expected codimension over the Hurwitz space.
	
	Let $\mathcal{H}^{\mathrm{tot}}_{g,k}$ denote the Hurwitz space parameterizing covers $[\pi:C\to \PP^1,p]$ of degree $k$ with $C$ smooth of genus $g$ and $\pi$ totally ramified at $p$. For an ordered vector $\vec{e}=(e_0,\dots,e_{k-1})\in\Z^{k}_{\geq0}$, let $\mathbf{a} (\vec{e})$ be the integer sequence $\{0,k,\dots,ke_0,1+e_1,\dots,1+ke_1,\dots,(k-1)+(k-1)e_{k-1}\}$, with elements in increasing order. The \emph{pointed Hurwitz--Brill--Noether locus} $W^{\vec{e}}(\pi,p)$ parameterizes line bundles of degree $d=g+|\vec{e}|+k-1$  with $r=|\vec{e}|+k-1$ sections with vanishing sequence at $p$ given exactly by $\mathbf{a}(\vec{e})$. Note that unlike for ordinary splitting loci, $\vec{e}$ is an ordered vector, as prescribing $\mathbf{a}(\vec{e})$ is equivalent to prescribing the splitting types of \[\pi_{\ast}L ,~ \pi_{\ast}L(-p),~\dots,~ \pi_{\ast}L(-(k-1)p).\] Our first result recovers the one-point case of~\cite{Agg26} showing that $W^{\vec{e}}(\pi,p)$ has the expected dimension.
	
	\begin{theoremintro}[{Theorem~\ref{trm: dimension}}]\label[theorem]{thm_dimension_intro}
		Let $[\pi\colon C\to\PP^1,p]$ be a general element of $\mathcal{H}^{\mathrm{tot}}_{g,k}$ and let $\vec{e}\in\Z_{\geq0}^k$. Then $W^{\vec{e}}(\pi,p)$ is nonempty if and only if \[g-u(\vec{e})-\inv(\vec{e})\geq 0.\] When nonempty, it has dimension $g-u(\vec {e})-\inv(\vec{e})$. Where $u(\vec{e})=\sum_{i\neq j} \max\{0,e_i-e_j-1\}$, and $\inv\vec{e}=\#\{(i,j) ~\mid~ 0\leq i<j\leq k-1, e_i<e_j\}$.
	\end{theoremintro}
	
	We note that the assumption that the entries of $\vec{e}$ are non-negative is simply a normalization, indeed, as $\pi^{\ast}(\OO_{\PP^1}(1))=\OO_C(kp)$, twisting by sufficiently large multiple of $kp$ will make all entries non-negative. When $\vec{e}$ is weakly decreasing, we have $\inv(\vec{e})=0$ and $W^{\vec{e}}(\pi,p)$ corresponds to the classical splitting locus $W^{\vec{e}}(C)$. Thus we also recover the dimension theorem of~\cite{Larson2Vogt}.
	
	We provide two proofs of this result. For the first proof, we consider the locus of $\textrm{ord}(\Vec{e})$-nested linear series, which has of dimension $g-u(\Vec{e})$, see~\cite{Larson2Vogt}. Then we see that the additional inversions measure the extra codimension imposed by the marked point via flag degeneracy conditions. This proof has the advantage of being elegant, but relies on the the newly introduced concept of nested linear series. The second proof relies on classical flag degeneracy arguments, and on restrictions to open subloci where the ramification orders are exactly as predicted. This proof has the disadvantage of being numerically less elegant, but the advantage of using nothing more than results from~\cite{basiclimitlinear}.

	\subsection*{Inductive approach and Brill--Noether analogues}
	
	Already due to~\cite{Agg26}, there is a pointed Hurwitz--Brill--Noether refinement of the results in~\cite{Larson2Vogt}. For a general element $[\pi\colon C\rightarrow \mathbb{P}^1, p]$ of $\cHTot$, we show that the locus $\overline{W}^{\Vec{e}}(\pi,p)$ is
	
	\begin{enumerate}
		\item nonempty of pure dimension $g-u(\Vec{e})-\inv(\Vec{e})$ when $g-u(\Vec{e})-\inv(\Vec{e})\geq 0$, and is empty otherwise, cf.~\cite{Agg26} and Theorem~\ref{trm: dimension};
		\item normal and Cohen--Macaulay, and is smooth away from a locus of codimension at least two, see~\cite{Agg26} and Theorem~\ref{theorem: reduced and CM};
		\item irreducible when $g-u(\Vec{e}) -\inv(\Vec{e})\geq 1$, see~\cite{Agg26};
		
		\item finite and reduced when
		$g-u(\Vec{e}) - \inv(\Vec{e})=0$, and there is an efficient algorithm for computing the cardinality, see~\cite{Agg26} and Theorem~\ref{thm: hol-Euler-char}.
	\end{enumerate}
	
	Aggarwal's result follow via degeneration to a chain of elliptic curves, while the current paper only consider the degeneration to a curve of genus $g-1$ glued to an elliptic tail.

	\subsection*{Euler characteristics}
	
	The scheme-theoretic description of the elliptic-tail degeneration gives a recursion for the algebraic Euler characteristic. For a general cover of genus $g=n+u(\vec{e})+\inv(\vec{e})$, write
	\[\chi_n(\vec{e})=\chi(\overline W^{\vec{e}}(\pi,p),
	\OO_{\overline{W}^{\vec{e}}(\pi,p)}).
	\] In particular, for our one-pointed Hurwitz--Brill--Noether loci $W^{\vec{e}}(\pi,p)$, the possibilities on the elliptic tail are described by elementary shifts $t_i$ of the vector $\vec{e}$, permuting entries that are not decreasing, see Definition~\ref{def: t_i}, and $I(\vec{e})=\{i ~\mid~ t_i(\vec{e})\neq \vec{e}\}$, giving a neat recursion.
	
	\begin{theoremintro}[Theorem~\ref{thm: hol-Euler-char}]\label[theorem]{thm_hol_Euler_char_intro}
		The algebraic Euler characteristics satisfy the recursion \[\chi_n(\vec{e})=\sum_{i\in I(\vec{e})}\chi_n(t_i(\vec{e}))-I(\vec{e})\cdot \chi_{n-1}(\vec{e}).\]
	\end{theoremintro}
	
	In dimension zero, this specializes to a recursion for cardinalities. Aggarwal already computes these cardinalities more generally by counting reduced words of affine permutations \cite[Theorem 1.1(2)]{Agg26}; our recursion gives an alternative computation for the one-point loci. In positive dimension, it computes the algebraic Euler characteristic, and hence the arithmetic genus when the locus is a curve. This addresses an enumerative direction also raised in~\cite{Pflueger2026Transmission}. In particular, this formula provides a very efficient algorithm for computing algebraic Euler characteristics (and implicitly genera when $g-u(\Vec{e}) - \inv(\Vec{e}) =1$), see related conjectures in~\cite{Pflueger2026Transmission}. As an example, for 
	\[ \Vec{e} = (28,28,18,18,7,2) \ \textrm{and} \ g = 181 \]
	we have $u(\Vec{e}) = 180$ and $\inv(\Vec{e}) = 0$ so the Hurwitz--Brill--Noether locus is a smooth curve, of genus
	\[ 12354645379354833080765059657446062822690008113927811502151177705426049 \approx 1.2 \times 10^{70}.\]
	In Section~\ref{sec: Euler-char}, we give a concrete formula for cardinality and genus for maximal Hurwitz--Brill--Noether loci in gonality $4$, and moreover, reduce the problem of computing cardinalities, and more generally algebraic Euler characteristics to an elementary problem of counting duckling walks, and observing duckling relative positions.
	
	\subsection*{Components of expected dimension}
	
	Lastly, when $k\geq g+1$, the Hurwitz space $\cHTot$ dominates $\M_{g,1}$. The fact that the forgetful map is dominant allows us to realize the pullbacks of Brill--Noether loci 
	\[ \M^r_{g,d} \coloneqq \{C \ \mid \ W^r_d(C) \neq \emptyset\}  \]
	as Hurwitz--Brill--Noether loci in $\cHTot$, see Proposition~\ref{prop: BN and HBN}. The existence of components of expected dimension is established for Brill--Noether loci of codimension 
	\[ 1 \leq -\rho(g,r,d) \leq g-2,\]
	see~\cite{Pfl23} and~\cite{Teixidor-BN}. This allow us to prove a similar result for the universal pointed Hurwitz--Brill--Noether spaces
	\[
	\mathcal W^{\Vec e}_g
	\coloneqq
	\left\{
	[\pi:C\longrightarrow\mathbb P^1,p,L]
	\ \middle|\
	[\pi,p]\in\mathcal H^{\mathrm{tot}}_{g,k},
	\quad L\in W^{\Vec e}(\pi,p)
	\right\}.
	\]
	We obtain the following result, which is also Theorem~\ref{thrm:low-codimension-existence}.
	
	\begin{theoremintro}[Theorem~\ref{thrm:low-codimension-existence}]
		Let $g \geq 1$ and let $\Vec e=(e_0,\ldots,e_{k-1})\in\mathbb Z_{\geq0}^k$ satisfy
		\[
		1\leq c \coloneqq u(\Vec e)+\operatorname{inv}(\Vec e)-g\leq k-3.
		\]
		Write $b_0(\Vec{e})<b_1(\Vec{e})<\cdots<b_{k-1}(\Vec{e})$ for the increasing rearrangement of the integers $ke_j+j$, for $0\leq j \leq k-1$.
		Assume moreover that $b_{k-2}-b_c>k$. Then $\mathcal W_g^{\Vec e}$ has an irreducible component of
		dimension $2g+k-3 +(g-u(\Vec{e})-\inv(\Vec{e}))$, generically finite over its image in
		$\mathcal H^{\mathrm{tot}}_{g,k}$.
	\end{theoremintro}
	
	While technical, the numerical conditions in the theorem are very mild, and only require that $\vec{e}$ should not have too many of the numbers $ke_j+j$ concentrated in a length $k$ interval below the second largest (which is $b_{k-2}(\vec{e})$). For example, this is satisfied for all vectors $\vec{e}$ satisfying $e_1 > e_2+1$.

	\subsection*{Outline}
	
	In Section~\ref{sec: pHBN and combinatorics}, we provide the definition of the pointed Hurwitz-Brill-Noether loci, along with the combinatorial framework. We prove the dimension theorem in Section~\ref{sec: dimension results}, and establish scheme-theoretic regeneration and discuss the geometric consequences, with the results of~\cite{Agg26}, in Section~\ref{sec: scheme-theoretic regeneration}. In Section~\ref{sec: Euler-char}, we derive the Euler-characteristic recursion and its enumerative applications. Finally, Section~\ref{sec: expected codimension} treats components of expected codimension over the Hurwitz space $\cHTot$.
	
	\subsection*{Acknowledgments}
	The authors would like to thank Martin Ulrisch for his insightful comments at the beginning of the project, as well as Daksh Aggarwal, Gavril Farkas, Andres Rojas, and Isabel Vogt for helpful conversations. During part of preparation of this article, A.B. was supported by DFG through the Collaborative Research Centre TRR 326 Geometry and Arithmetic of Uniformized Structures, project number 444845124. A.B. is currently supported by ERC Starting
	Grant SpiCE, no.101164820. This research was partially conducted during the time that R.H. was supported by the National Science Foundation under Grant No. DMS-2231565, and under Grant No. DMS-2424556 while the author was in residence at the Institute for Computational and Experimental Research in Mathematics in Providence, RI, during the Computations on K3 Surfaces and Related Varieties semester program.

	\subsection*{AI disclosure} During the preparation of this paper, ChatGPT was used to improve the exposition, assist with bibliography, and referee statements and proofs. The proof and statement of Proposition~\ref{prop:low-codimension-existence} was simplified with assistance from ChatGPT 5.6 and 6, and subsequently simplified and rewritten by the authors. The authors take full responsibility for the mathematical content, references, and final text.

	\begin{notation}
		For notational convenience, if $\vec{e}\in\Z^k$ is a vector, we write $|\vec{e}|\coloneqq \sum_{i=0}^{k-1}e_i$.
	\end{notation}
	
	\section{Hurwitz--Brill--Noether space and vanishing orders} \label{sec: pHBN and combinatorics}
	In this section we introduce pointed Hurwitz–Brill–Noether loci in terms of the vanishing orders of line bundles at the totally ramified point. We then prove that we have an equivalent definition in terms of the splitting types of 
	\[ \pi_*L, \ \pi_*(L-p), \ \ldots \ \pi_*(L-(k-1)\cdot p). \] 
	The goal of this section will then be to provide the combinatorics necessary in order to compute the dimension of pointed Hurwitz--Brill--Noether loci. 
	
	Let $r\geq 1$, $d = g+r$ and let 
	\[ \textbf{a} = (0\leq a_0 < a_1 < \cdots < a_r\leq d)\]
	be a sequence of vanishing orders. For a pointed curve $[C,p]\in \mathcal{M}_{g,1}$ we are interested in understanding the space 
	\[ W^r_{d,\textbf{a}}(C,p) \coloneqq \left\{ L\in \textrm{Pic}^d(C) \ | \ h^0(C,L(-a_i\cdot p)) \geq r+1-i \ \text{ for } \ 0\leq i \leq r \right\}.\] 
	In particular, we ask what is the dimension of this space when $C$ admits a degree $k$ map $\pi \colon C\rightarrow \mathbb{P}^1$ totally ramified at $p$. 
	
	In order to compute this dimension, we need to understand what happens when we degenerate the curve $C$ to a singular curve with two components: one of genus $g-1$ and one elliptic. The dimension count will then follow from a simple induction argument. 
	
	Let now $\elhur$ be an element of $\cHTot$. Our first goal is to understand what are the possible vanishing orders at $p$ of a line bundle $L$ on $C$. This is encapsulated in the following definition.
	
	\begin{defn}
		A sequence of vanishing orders $\textbf{a} = (a_0, a_1, \ldots, a_r)$ is called $k$-admissible if it satisfies the following property: 
		\begin{itemize}
			\item If $a_i \geq k$ is one of the entries of the vanishing sequence, then $a_i-k$ is also an entry. 
		\end{itemize}
	\end{defn}
	
	%The motivation for this definition is that Hurwitz--Brill--Noether loci with $k$-admissible ramification orders at a point naturally appear as degeneration of Hurwitz--Brill--Noether loci. 
	
	There is an immediate way to obtain $k$-admissible ramification sequences out of the datum of a vector $\Vec{e}$ with non-negative entries. 
	
	\begin{defn} \label{def: sequence and pHBN}
		Let $\Vec{e} = (e_0,e_1, \ldots, e_{k-1})\in\Z^k_{\geq 0}$. We define its associated $k$-admissible vanishing sequence $\textbf{a}(\Vec{e})$ to be the one with entries
		\[ \{0, k, \ldots, e_0\cdot k \} \cup \{1, k+1, \ldots, e_1\cdot k +1\} \cup \cdots \cup \{k-1, 2k-1,\ldots, e_{k-1}\cdot k+k-1\} \]
		in increasing order. 
		
		For $\elhur \in \cHTot$ we define the pointed Hurwitz--Brill--Noether space $W^{\Vec{e}}(\pi,p)$ to be the space $W^{r,\circ}_{g+r, \textbf{a}(\Vec{e})}(C,p)$, parametrizing degree $g+r$ line bundles on $C$ with vanishing orders at $p$ exactly the entries of $\textbf{a}(\Vec{e})$. 
	\end{defn}

	\begin{example}\label{ex:example_54010}
		We give an example of the vanishing orders for $k=5$ and $\Vec{e} = (5,4,0,1,0)$. By definition, we have \[\textbf{a}(5,4,0,1,0)=(0,1,2,3,4,5,6,8,10,11,15,16,20,21,25).\] The vanishing orders in $\textbf{a}(\Vec{e})$ can be more easily recovered by drawing columns of $e_i +1$ boxes, and filling the boxes consecutively, skipping when there is no box. Thus $\textbf{a}(5,4,0,1,0)$ is recovered from the numbers in Figure~\ref{fig:54010_vanishing}, while the dashed squares represent the gaps in $\textbf{a}(\Vec{e})$.
		\begin{figure}[H]
			\centering
			\begin{tikzpicture}
				% Draw vertical lines
				\foreach \x in {0,...,5}
				\draw[gray] (\x,0) -- (\x,6);
				
				% Draw horizontal lines
				\foreach \y in {0,...,6}
				\draw[gray] (0,\y) -- (5,\y);
				
				\draw[very thick] (0,0) -- (1,0);
				\draw[very thick] (1,1) -- (1,0);
				\draw[very thick] (1,1) -- (2,1);
				\draw[very thick] (2,1) -- (2,5);
				\draw[very thick] (2,5) -- (3,5);
				\draw[very thick] (3,5) -- (3,4);
				\draw[very thick] (3,4) -- (4,4);
				\draw[very thick] (4,4) -- (4,5);
				\draw[very thick] (4,5) -- (5,5);
				
				\node at (0.5,5.5) {\Large 0};
				\node at (1.5,5.5) {\Large 1};
				\node at (2.5,5.5) {\Large 2};
				\node at (3.5,5.5) {\Large 3};
				\node at (4.5,5.5) {\Large 4};
				\node at (0.5,4.5) {\Large 5};
				\node at (1.5,4.5) {\Large 6};
				\node at (3.5,4.5) {\Large 8};
				\node at (0.5,3.5) {\Large 10};
				\node at (1.5,3.5) {\Large 11};
				\node at (0.5,2.5) {\Large 15};
				\node at (1.5,2.5) {\Large 16};
				\node at (0.5,1.5) {\Large 20};
				\node at (1.5,1.5) {\Large 21};
				\node at (0.5,0.5) {\Large 25};
				
				\fill[pattern=north east lines, pattern color=gray] (2,4) rectangle (3,5);
				\fill[pattern=north east lines, pattern color=gray] (4,4) rectangle (5,5);
				\fill[pattern=north east lines, pattern color=gray] (2,3) rectangle (3,4);
				\fill[pattern=north east lines, pattern color=gray] (3,3) rectangle (4,4);
				\fill[pattern=north east lines, pattern color=gray] (4,3) rectangle (5,4);
				\fill[pattern=north east lines, pattern color=gray] (2,1) rectangle (3,2);
				\fill[pattern=north east lines, pattern color=gray] (3,1) rectangle (4,2);
				\fill[pattern=north east lines, pattern color=gray] (4,1) rectangle (5,2);
				\fill[pattern=north east lines, pattern color=gray] (2,2) rectangle (3,3);
				\fill[pattern=north east lines, pattern color=gray] (3,2) rectangle (4,3);
				\fill[pattern=north east lines, pattern color=gray] (4,2) rectangle (5,3);

			\end{tikzpicture}
			\caption{Vanishing orders at $p$ for $L\in W^{(5,4,0,1,0)}(\pi,p)$.}
			\label{fig:54010_vanishing}
		\end{figure}
	\end{example}
	
	We aim to show that the splitting types of line bundles in the loci $W^{\vec{e}}(\pi,p)$ are determined by the ordered tuple $\vec{e}$.
	
	\begin{remark} \label{rmk: different from LLV}
		Notice that $W^{\Vec{e}}(\pi,p)$ is dependent of the order of the elements in $\Vec{e}$. When we have 
		\[e_0\geq e_1 \geq \cdots \geq e_{k-1}, \]
		the space $W^{\Vec{e}}(\pi,p)$ is an open subset of the Hurwitz--Brill--Noether space studied in~\cite{Larson2Vogt}, see Corollary~\ref{cor: LLV-dim2}.
	\end{remark}

	We need the following result about morphisms of vector bundles on $\mathbb{P}^1$. 
	
	\begin{lemma} \label{lemma: VB injectivity}
		Let $\Vec{e},\Vec{f}\in\Z^k$ with $|\vec{e}|-1=|\vec{f}|$ and suppose that
		%\[ \sum_{i=0}^{k-1} e_i -1  =  \sum_{i=0}^{k-1} f_i  \ \textrm{and}\]
		there exists an injective morphism $\mathcal{O}(\Vec{f})\hookrightarrow \mathcal{O}(\Vec{e})$ of vector bundles on $\mathbb{P}^1$. 
		Then $\Vec{f}$ is obtained from $\Vec{e}$ by subtracting $1$ from one of its entries.
	\end{lemma}
	\begin{proof} By twisting the morphism $\mathcal{O}(\Vec{f})\rightarrow \mathcal{O}(\Vec{e})$, we can assume that all entries of $\Vec{e}$ and $\Vec{f}$ are greater or equal to $1$. Denote by $N$ the largest entry of $\Vec{e}$. 
		
		Recall that for $n>m$, $\operatorname{Hom}(\OO_{\PP^1}(n),\OO_{\PP^1}(m))=0$. 
		%We will use the following obvious result: If $n > m$, the only morphism $\mathcal{O}(n)\rightarrow \mathcal{O}(m)$ is the zero morphism. 
		In particular, this implies that for every natural number $n$ we have an injective morphism 
		\[ \mathcal{O}(\Vec{f}_{\geq n}) \rightarrow \mathcal{O}(\Vec{e}_{\geq n})  \]
		In particular, we have the inequality of cardinalities 
		\[ \#\{i \ | \ f_i\geq n\} \leq \#\{i \ | \ e_i\geq n\} \ \textrm{for every} \ n\geq 1. \]
		We denote 
		\[ b_n \coloneqq \#\{ i \ | \ e_i = n\} \ \textrm{and} \  c_n \coloneqq \#\{ i \ | \ f_i = n\}. \]
		The condition 
		\[  \sum_{i=0}^{k-1} e_i -1  =  \sum_{i=0}^{k-1} f_i\]
		can be rewritten as 
		\[ \sum_{n=1}^N n\cdot b_n - 1 = \sum_{n=1}^N n\cdot c_n \]
		or equivalently as 
		\[ \sum_{n=1}^N \sum_{j=n}^N b_j -1 = \sum_{n=1}^N \sum_{j=n}^N c_j.  \]
		But for any $1\leq n \leq N$ we have 
		\[ \sum_{j=n}^N c_j \leq \sum_{j=n}^N b_j. \]
		In particular, all the inequalities but one are in fact equalities. This implies that $\Vec{f}$ has the form stated in the hypothesis. 
	\end{proof}
	
	Using Lemma~\ref{lemma: VB injectivity}, we can study pointed Brill--Noether conditions on $[\pi\colon C \rightarrow \mathbb{P}^1, p] \in \cHTot$. An immediate corollary is that for any line bundle $L$ on $C$, the vanishing orders at $p$ are $k$-admissible.
	
	\begin{prop} \label{prop: ramification behaviour}
		Let $[\pi\colon C\rightarrow \mathbb{P}^1,p]$ be an element of $\cHTot$ and let $L$ be a line bundle on $C$. If $a\geq k$ is a vanishing order of $H^0(C,L)$ at $p$, then so is $a-k$.
	\end{prop}
	\begin{proof}
		We denote by $\Vec{e}_0, \Vec{e}_1,\ldots, \Vec{e}_k$ the vectors satisfying 
		\[ \pi_*(L(-ip)) \cong \mathcal{O}(\Vec{e}_i) \ \textrm{for every} \ 0\leq i \leq k.\]
		Notice that $\Vec{e}_i$ is equal to $\Vec{e}_{i+1}$ except for a unique position on which the entry in $\Vec{e}_i$ is bigger by exactly $1$, see Lemma~\ref{lemma: VB injectivity}.
		
		We write $a = ik +r$ for some $0\leq r \leq k-1$. The fact that $a$ is a vanishing order of $H^0(C,L)$ implies that 
		\[h^0(C, L (- ap)) = h^0(C, L(-(a+1)p)) +1\]
		and hence , via pushforward and the push-pull formula,
		\[ h^0(\mathbb{P}^1, \mathcal{O}(\Vec{e}_r)\otimes \mathcal{O}(-i)) = h^0(\mathbb{P}^1, \mathcal{O}(\Vec{e}_{r+1})\otimes \mathcal{O}(-i))+1.\]
		Because $\Vec{e}_r$ and $\Vec{e}_{r+1}$ differ by exactly $1$ on a unique position, the condition above is equivalent to saying that the entry of $\Vec{e}_r$ on which they differ is greater or equal to $i$. 
		
		Completely analogously, $a-k$ is a vanishing orders if and only if the entry of $\Vec{e}_r$ on which it differs from $\Vec{e}_{r+1}$ is greater or equal to $i-1$. The statement reduces to ``if a number is greater or equal to $i$ then it is greater or equal to $i-1$". The conclusion immediately follows. 
	\end{proof}
	
	\begin{example}
		Given $L\in W^{\vec{e}}(\pi,p)$, the vanishing orders and global sections of $L(-ip)$ are easily determined. We give an example, continuing Example~\ref{ex:example_54010}. The vanishing orders of $L(-p)$ are simply obtained from $\textbf{a}(\vec{e})$, subtracting $1$ and eliminating any negative entries. This can be visualized using the fillings of squares as in Figure~\ref{fig:54010_-ip_vanishing}.
		
		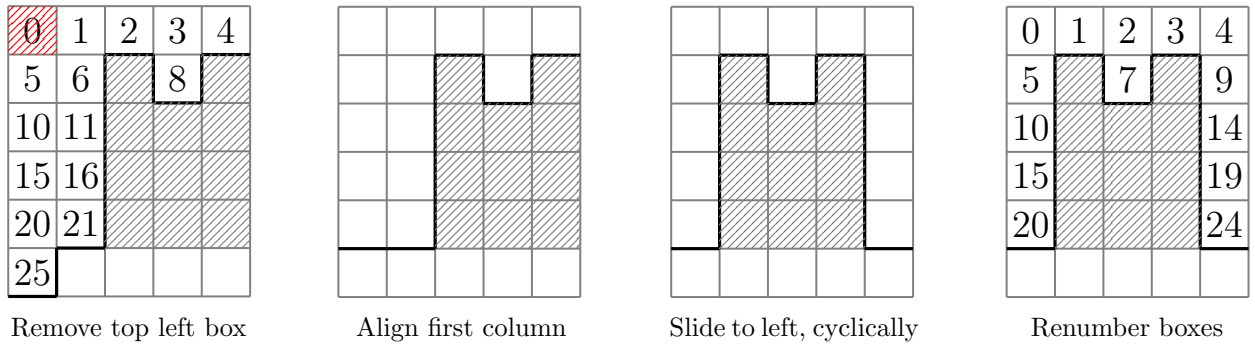
\begin{figure}[H]
			\centering
			\begin{subfigure}{0.2\textwidth}
				\begin{tikzpicture}[scale = 0.64]
					% Draw vertical lines
					\foreach \x in {0,...,5}
					\draw[gray] (\x,0) -- (\x,6);
					
					% Draw horizontal lines
					\foreach \y in {0,...,6}
					\draw[gray] (0,\y) -- (5,\y);
					
					\draw[very thick] (0,0) -- (1,0);
					\draw[very thick] (1,1) -- (1,0);
					\draw[very thick] (1,1) -- (2,1);
					\draw[very thick] (2,1) -- (2,5);
					\draw[very thick] (2,5) -- (3,5);
					\draw[very thick] (3,5) -- (3,4);
					\draw[very thick] (3,4) -- (4,4);
					\draw[very thick] (4,4) -- (4,5);
					\draw[very thick] (4,5) -- (5,5);
					
					\node at (0.5,5.5) {\Large 0};
					\node at (1.5,5.5) {\Large 1};
					\node at (2.5,5.5) {\Large 2};
					\node at (3.5,5.5) {\Large 3};
					\node at (4.5,5.5) {\Large 4};
					\node at (0.5,4.5) {\Large 5};
					\node at (1.5,4.5) {\Large 6};
					\node at (3.5,4.5) {\Large 8};
					\node at (0.5,3.5) {\Large 10};
					\node at (1.5,3.5) {\Large 11};
					\node at (0.5,2.5) {\Large 15};
					\node at (1.5,2.5) {\Large 16};
					\node at (0.5,1.5) {\Large 20};
					\node at (1.5,1.5) {\Large 21};
					\node at (0.5,0.5) {\Large 25};
					
					\fill[pattern=north east lines, pattern color=gray] (2,4) rectangle (3,5);
					\fill[pattern=north east lines, pattern color=gray] (4,4) rectangle (5,5);
					\fill[pattern=north east lines, pattern color=gray] (2,3) rectangle (3,4);
					\fill[pattern=north east lines, pattern color=gray] (3,3) rectangle (4,4);
					\fill[pattern=north east lines, pattern color=gray] (4,3) rectangle (5,4);
					\fill[pattern=north east lines, pattern color=gray] (2,1) rectangle (3,2);
					\fill[pattern=north east lines, pattern color=gray] (3,1) rectangle (4,2);
					\fill[pattern=north east lines, pattern color=gray] (4,1) rectangle (5,2);
					\fill[pattern=north east lines, pattern color=gray] (2,2) rectangle (3,3);
					\fill[pattern=north east lines, pattern color=gray] (3,2) rectangle (4,3);
					\fill[pattern=north east lines, pattern color=gray] (4,2) rectangle (5,3);
					\fill[pattern=north east lines, pattern color=red] (0,5) rectangle (1,6);

				\end{tikzpicture}\caption*{Remove top left box}
			\end{subfigure} 
			\hfill
			\begin{subfigure}{0.2\textwidth}
				\begin{tikzpicture}[scale = 0.64]
					% Draw vertical lines
					\foreach \x in {0,...,5}
					\draw[gray] (\x,0) -- (\x,6);
					
					% Draw horizontal lines
					\foreach \y in {0,...,6}
					\draw[gray] (0,\y) -- (5,\y);
					
					\draw[very thick] (0,1) -- (1,1);
					%\draw[very thick] (1,1) -- (1,0);
					\draw[very thick] (1,1) -- (2,1);
					\draw[very thick] (2,1) -- (2,5);
					\draw[very thick] (2,5) -- (3,5);
					\draw[very thick] (3,5) -- (3,4);
					\draw[very thick] (3,4) -- (4,4);
					\draw[very thick] (4,4) -- (4,5);
					\draw[very thick] (4,5) -- (5,5);
					
					\fill[pattern=north east lines, pattern color=gray] (2,4) rectangle (3,5);
					\fill[pattern=north east lines, pattern color=gray] (4,4) rectangle (5,5);
					\fill[pattern=north east lines, pattern color=gray] (2,3) rectangle (3,4);
					\fill[pattern=north east lines, pattern color=gray] (3,3) rectangle (4,4);
					\fill[pattern=north east lines, pattern color=gray] (4,3) rectangle (5,4);
					\fill[pattern=north east lines, pattern color=gray] (2,1) rectangle (3,2);
					\fill[pattern=north east lines, pattern color=gray] (3,1) rectangle (4,2);
					\fill[pattern=north east lines, pattern color=gray] (4,1) rectangle (5,2);
					\fill[pattern=north east lines, pattern color=gray] (2,2) rectangle (3,3);
					\fill[pattern=north east lines, pattern color=gray] (3,2) rectangle (4,3);
					\fill[pattern=north east lines, pattern color=gray] (4,2) rectangle (5,3);
					
				\end{tikzpicture} \caption*{Align first column}
			\end{subfigure} 
			\hfill 
			\begin{subfigure}{0.2\textwidth}
				\begin{tikzpicture}[scale = 0.64]
					% Draw vertical lines
					\foreach \x in {0,...,5}
					\draw[gray] (\x,0) -- (\x,6);
					
					% Draw horizontal lines
					\foreach \y in {0,...,6}
					\draw[gray] (0,\y) -- (5,\y);
					
					\draw[very thick] (0,1) -- (1,1);
					%\draw[very thick] (1,1) -- (1,0);
					%\draw[very thick] (1,1) -- (2,1);
					\draw[very thick] (1,1) -- (1,5);
					\draw[very thick] (1,5) -- (2,5);
					\draw[very thick] (2,5) -- (2,4);
					\draw[very thick] (2,4) -- (3,4);
					\draw[very thick] (3,4) -- (3,5);
					\draw[very thick] (3,5) -- (4,5);
					\draw[very thick] (4,1) -- (5,1);
					\draw[very thick] (4,1) -- (4,5);
					
					\fill[pattern=north east lines, pattern color=gray] (1,1) rectangle (4,4);
					\fill[pattern=north east lines, pattern color=gray] (1,4) rectangle (2,5);
					\fill[pattern=north east lines, pattern color=gray] (3,4) rectangle (4,5);
					
				\end{tikzpicture} 
				\caption*{Slide to left, cyclically}
				
			\end{subfigure}
			\hfill 
			\begin{subfigure}{0.2\textwidth}
				\begin{tikzpicture}[scale = 0.64]
					% Draw vertical lines
					\foreach \x in {0,...,5}
					\draw[gray] (\x,0) -- (\x,6);
					
					% Draw horizontal lines
					\foreach \y in {0,...,6}
					\draw[gray] (0,\y) -- (5,\y);
					
					\node at (0.5,5.5) {\Large 0};
					\node at (1.5,5.5) {\Large 1};
					\node at (2.5,5.5) {\Large 2};
					\node at (3.5,5.5) {\Large 3};
					\node at (4.5,5.5) {\Large 4};
					\node at (0.5,4.5) {\Large 5};
					\node at (2.5,4.5) {\Large 7};
					\node at (4.5,4.5) {\Large 9};
					\node at (0.5,3.5) {\Large 10};
					\node at (4.5,3.5) {\Large 14};
					\node at (0.5,2.5) {\Large 15};
					\node at (4.5,2.5) {\Large 19};
					\node at (0.5,1.5) {\Large 20};
					\node at (4.5,1.5) {\Large 24};
					
					\draw[very thick] (0,1) -- (1,1);
					%\draw[very thick] (1,1) -- (1,0);
					%\draw[very thick] (1,1) -- (2,1);
					\draw[very thick] (1,1) -- (1,5);
					\draw[very thick] (1,5) -- (2,5);
					\draw[very thick] (2,5) -- (2,4);
					\draw[very thick] (2,4) -- (3,4);
					\draw[very thick] (3,4) -- (3,5);
					\draw[very thick] (3,5) -- (4,5);
					\draw[very thick] (4,1) -- (5,1);
					\draw[very thick] (4,1) -- (4,5);
					
					\fill[pattern=north east lines, pattern color=gray] (1,1) rectangle (4,4);
					\fill[pattern=north east lines, pattern color=gray] (1,4) rectangle (2,5);
					\fill[pattern=north east lines, pattern color=gray] (3,4) rectangle (4,5);
					
				\end{tikzpicture} 
				\caption*{Renumber boxes}
				
			\end{subfigure}
			\caption{Vanishing orders of $L(-p)$ for $L\in W^{(5,4,0,1,0)}(\pi,p)$.}
			\label{fig:54010_-ip_vanishing}
		\end{figure}
		
		Thus, for $L\in W^{(5,4,0,1,0)}(\pi,p)$, the vanishing orders for for $L(-p)$ are \[(0,1,2,3,4,5,7,9,10,14,15,19,20,24),\]
		which is just the vanishing sequence corresponding to the vector $(4,0,1,0,4)$.
	\end{example}
	\
	To keep track of the vanishing orders, we consider the following definition.
	
	\begin{defn} \label{def: sigma}
		Given $\vec{e}=(e_0,\dots,e_{k-1})\in \Z^k$, we define the map \[\sigma: \Z^k \to \Z^k,\ \vec{e}=(e_0,e_1,e_2,\dots,e_{k-1})\mapsto (e_1, e_2, \dots, e_{k-1}, e_{0} -1).\] We define $\sigma^i : \Z^k \to \Z^k$ as the $i^{th}$ composition of $\sigma$. For example, \[\sigma^2(\vec{e})=(e_2,e_3,\dots,e_{k-1},e_0 -1,e_{1}-1), \text{ and } \sigma^{-1}(\vec{e})=(e_{k-1}+1,e_0,e_1,\dots,e_{k-2}).\] 
	\end{defn}
	
	\begin{remark} \label{rmk: Hilbert}
		Let $L\in\Pic(C)$. Since $\OO_C(kp)\cong\pi^*\OO_{\PP^1}(1)$, the projection formula gives
		\[
		h^0\bigl(\PP^1,\pi_*(L(-ip))\otimes\OO_{\PP^1}(n)\bigr)
		=
		h^0\bigl(C,L((nk-i)p)\bigr)
		\]
		for all $i,n\in\Z$. Hence the Hilbert function
		\[
		H_i(n)\coloneqq h^0\bigl(\PP^1,\pi_*(L(-ip))\otimes\OO_{\PP^1}(n)\bigr)
		\]
		determines the splitting type of $\pi_*(L(-ip))$. More precisely, the multiplicity of $\OO_{\PP^1}(j)$ as a summand of $\pi_*(L(-ip))$ is
		\[
		H_i(-j)-2H_i(-j-1)+H_i(-j-2);
		\]
		see \cite[Section~2]{Larson2021} and \cite[Lemma~5.6]{Eisenbud_Schreyer_2008}.
		
		In particular, the vanishing orders at $p$ determine the splitting types of $L(-ip)$ for all $i\in\Z$. Conversely, the splitting types of $L(-ip)$ for all $i\in\Z$ determine the vanishing orders of $L$ at $p$.
	\end{remark}
	
	Consequently for $L\in W^{\vec{e}}(\pi,p)$, the splitting types of $L(-ip)$ are determined completely by the ordered tuple $\vec{e}$, by keeping track of the vanishing orders at $p$.
	
	\begin{prop}\label{prop: intersection description}
		For $\vec{e}\in \Z^k_{\geq 0}$, we have
		\[
		W^{\vec{e}}(\pi,p)
		=
		\left\{
		L\in \Pic(C)
		\ \middle|\
		\pi_{\ast}(L(-ip))\cong \OO_{\PP^1}(\sigma^i(\vec{e}))
		\ \textrm{for every } 0\leq i\leq k-1
		\right\}.
		\]
		Equivalently, the condition on the right-hand side holds for every $i\in\Z$.
	\end{prop}

	\begin{proof} 
		The inclusion $\supseteq$ is clear: If $L \in W^{\vec{e}}(\pi,p)$ it follows that (cf. Example~\ref{ex:example_54010}) 
		\[ \pi_{\ast}(L(-p))\cong \OO_{\PP^1}(\sigma(\vec{e})). \]
		Repeating the argument, we obtain the first inclusion. 
		
		Using Remark~\ref{rmk: Hilbert}, the vanishing orders of $L$ are completely determined by the vector bundles 
		\[ \pi_*L, \ \pi_*(L-p), \ \ldots \ \pi_*(L-(k-1)\cdot p). \]
		In particular, the inclusion above tells us that the conditions  
		\[ \pi_{\ast}(L(-ip))\cong \OO_{\PP^1}(\sigma^i(\vec{e}))
		\ \textrm{for every } 0\leq i\leq k-1\]
		determine the same vanishing orders for $L$ as the condition $L \in W^{\vec{e}}(\pi,p)$. This implies the reverse inclusion. To see that it is enough to consider the splitting types for $0\leq i\leq k-1$ suffices, note that the push-pull formula implies 
		$$\pi_{\ast}(L(nkp))=(\pi_{\ast}(L))\otimes\OO_{\PP^1}(n)$$ 
		for any $n$ and any line bundle $L$ on $C$. 
	\end{proof}
	
	We study degenerations of pointed Hurwitz--Brill--Noether loci
	over the admissible cover compactification $\hhTot$ of $\cHTot$,
	focusing on the following setting.
	
	\begin{setup}[Elliptic-tail degeneration]\label{setup: elliptic tail}
		Fix an admissible cover
		\[
		\pi_0\colon X\cup_q E
		\longrightarrow P_X\cup_{\overline q}P_E,
		\]
		where $X$ and $E$ are smooth curves of genera $g-1$ and $1$,
		respectively, and $P_X\cong P_E\cong\PP^1$.
		The restrictions
		\[
		\pi_X\colon X\to P_X,
		\qquad
		\pi_E\colon E\to P_E
		\]
		have degree $k$ and are totally ramified at $q$.
		We fix a marked point $p\in E\setminus\{q\}$ at which
		$\pi_E$ is also totally ramified, and assume that
		$p-q\in\Pic^0(E)$ has order exactly $k$. We further
		assume that
		\[
		[\pi_X\colon X\longrightarrow P_X,q]
		\in\mathcal{H}^{\mathrm{tot}}_{g-1,k}
		\]
		is general. 
	\end{setup}
	
	To describe the possible vanishing orders at the node $q$
	in Setup~\ref{setup: elliptic tail}, we introduce the following
	transformations.
	
	\begin{defn} \label{def: t_i} For any vector $\Vec{e} = (e_0,\ldots, e_{k-1})$ and any $0\leq i \leq k-1$, we define a vector $t_{j}(\Vec{e})$ using the following rules. 
		\begin{enumerate}
			\item For $1\leq j \leq k-1$
			\[
			t_j(\Vec{e}) \coloneqq \left\{
			\begin{array}{l}
				(e_0, e_1,\ldots, e_{j-2}, e_j, e_{j-1}, e_{j+1}, \ldots, e_{k-1}), \ \textrm{if} \ e_j > e_{j-1}  \\[0.2cm]
				\Vec{e}, \ \textrm{otherwise}.
			\end{array} \right.
			\]
			\item For $j=0$
			\[
			t_0(\Vec{e}) \coloneqq \left\{
			\begin{array}{l}
				(e_{k-1}+1, e_1,\ldots, e_{k-2}, e_0-1), \ \textrm{if} \ e_0 > e_{k-1}+1  \\[0.2cm]
				\Vec{e}, \ \textrm{otherwise}.
			\end{array} \right.
			\]
		\end{enumerate}
		
	\end{defn}

	\begin{lemma} \label{lemma: ramifications} Let $[\pi_0\colon X\cup_q E
		\rightarrow P_X\cup_{\overline q}P_E, p] \in \hhTot$ as described in Setup~\ref{setup: elliptic tail}, $\Vec{e}  \in\Z^k_{\geq0}$, $r = k+ |\vec{e}| - 1$, $d = g+r$, and
		\[\textbf{a}(\Vec{e}) = (0\leq a_0 < a_1 < \cdots < a_r \leq d)  \] 
		the $k$-admissible vanishing sequence associated to $\Vec{e}$. Denote by $\left\{l_X \coloneqq (L_X, V_X), l_E \coloneqq (L_E, V_E)\right\}$ the two aspects of a limit linear $g^r_d$ with ramification orders $\textbf{a}(\Vec{e})$ at $p$. Then, the associated pair of line bundles $\{L_X, L_E\} \in \textrm{Pic}^d(X) \times \textrm{Pic}^d(E)$ is contained in the union of the loci 
		\[ \left( \OO_X(q)\otimes W^r_{r+g-1,\textbf{a}(\Vec{e})}(X,q)\right)  \times \textrm{Pic}^d(E)  \text{ and} \]
		
		\[ \left( \OO_X(q)\otimes W^r_{r+g-1,\textbf{a}(t_j(\Vec{e}))}(X,q)\right) \times\{\OO_E(jp+(d-j)q)\}, \]
		for $0\leq j \leq k-1$ such that $t_j(\Vec{e}) \neq \Vec{e}$.  
		
	\end{lemma}
	\begin{proof}
		We distinguish two possible cases.
		
		\underline{Case~1:} If $L_E$ is not a line bundle of the form $\mathcal{O}_E(jp+ (d-j)q)$ for some $0\leq j \leq k-1$, then the ramification orders of the $E$-aspect at $q$ are smaller or equal to \[ (d-1-a_r, \ldots, d-1-a_0). \] This implies, using the definition of limit linear series, that the vanishing orders of $l_X$ at the point $q$ are greater or equal to \[ 1+\textbf{a}(\Vec{e}) \coloneqq (a_0+1, \ldots, a_r+1).\] It follows that $(L_X(-q), V_X(-q))$ determines a $g^r_{d-1}$ on $X$ with vanishing orders greater or equal to $\textbf{a}(\Vec{e})$ at the point $q$.

		\underline{Case~2:} If $L_E$ is of the form $\mathcal{O}_E(jp+ (d-j)q)$ for some $0\leq j\leq k-1$, then the vanishing orders of the $E$-aspect at $q$ are smaller or equal to 
		\[ (d-1-a_r + \delta^{a_r, j}, \ldots, d-1-a_0+ \delta^{a_0, j}), \]
		where $\delta^{x,j} = 1$ if $x\equiv j \pmod{k}$ and is $0$ otherwise. 
		
		In particular, the vanishing orders of $l_X$ at the point $q$ are greater or equal to
		\[  (a_0+1-\delta^{a_0, j}, \ldots, a_r+1-\delta^{a_r, j}).\] Since $h^0(X,L_X) \geq r+2$, the curve $X$ admits a $g^{r+1}_d$ with vanishing orders at $q$ greater or equal to \[  (0, a_0+1-\delta^{a_0, j}, \ldots, a_r+1-\delta^{a_r, j}).\] Tensoring with $\mathcal{O}_X(-q)$ we obtain that $X$ admits a $g^r_{d-1}$ with vanishing orders at $q$ greater or equal to \[  (a_0-\delta^{a_0, j}, \ldots, a_r-\delta^{a_r, j}).\] But the $r+1$ ramification orders must be distinct non-negative integers. This implies that $X$ admits a $g^r_{d-1}$ with ramification orders at $q$ greater or equal to $\textbf{a}'$, where $\textbf{a}'$ is the smallest vectors with distinct non-negative entries greater or equal to \[  (a_0-\delta^{a_0, j}, \ldots, a_r-\delta^{a_r, j})\] on all positions.
		
		This sequence $\textbf{a}' = (a_0',\ldots, a_r')$ of ramification orders is easily computed from $\textbf{a}(\Vec{e})$ by the following rules 
		\begin{equation*} 
			a_0' = \max\{0, a_0-\delta^{a_0, j} \},
		\end{equation*}
		while for $i>0$ we have 
		\begin{equation*}
			a_i'  =
			\begin{cases} 
				a_i-1, &\ \textrm{if} \ a_{i-1}\neq a_i-1 \ \textrm{and} \ \delta^{a_i, j} = 1 \\ 
				a_i, &\ \textrm{otherwise}.
			\end{cases}
		\end{equation*}
		
		Id est, a vanishing order $ n\cdot k + j$ is changed to $ n\cdot k + j-1$ if and only if  $n\cdot k + j-1$ is not already a vanishing order. But $n\cdot k + j-1$ is not a vanishing order if and only if $n\cdot k + j-1 > e_{j-1}\cdot k + j-1$. Describing $\textbf{a}'$ as the vanishing sequence of a vector, we see this corresponds to interchanging $e_j$ and $e_{j-1}$ if $j>0$ and $e_j > e_{j-1}$, or to changing the first entry to $e_{k-1}+1$ and the last one to $e_0-1$, if $j = 0$. 
		
		In total, $\textbf{a}' = \textbf{a}(t_j(\Vec{e}))$. In particular, the pair $(L_X, L_E)$ associated to a limit $g^r_d$ with vanishing orders $\textbf{a}(\Vec{e})$ must be in one of the loci described in the hypothesis. 
	\end{proof}
	
	Let now $\Vec{e}$ and $\Vec{f}$ be two length $k$ non-negative vectors and let $\elhur \in \cHTot$. We can associate to them the pointed Hurwitz--Brill--Noether spaces $W^{\Vec{e}}(\pi,p)$ and  $W^{\Vec{f}}(\pi,p)$, as appearing in Definition~\ref{def: sequence and pHBN}. We would like these loci to behave well with respect to inclusions: namely if $\textbf{a}(\Vec{e}) \leq \textbf{a}(\Vec{f})$ entry-wise we want the inclusion 
	\[ \overline{W}^{\Vec{f}}(\pi, p) \subseteq  \overline{W}^{\Vec{e}}(\pi, p).  \]
	
	To compare two pointed Hurwitz--Brill--Noether spaces, we first define a partial order $\succeq$ on vectors.
	
	\begin{defn} \label{def: order} Let $\Vec{e}$ and $\Vec{f}$ be two vectors satisfying $|\vec{e}|=|\vec{f}|$. We define a partial order on them via the rule
		\[ \Vec{f} \succeq \Vec{e} \ \iff \ \sum_{i=0}^n f_i \geq \sum_{i=0}^n e_i \ \textrm{for every} \ 0\leq n \leq k-1. \]
		This definition will be useful when working with vectors having ordered entries. As such, we define $\textrm{ord}(\Vec{e})$ to be the vector with same entries as $\Vec{e}$ in decreasing order. 
	\end{defn}
	
	Then, we can relate the entry-wise inequality   $\textbf{a}(\Vec{e}) \leq \textbf{a}(\Vec{f})$ with the ordering defined above.
	
	\begin{lemma} \label{lemma: order of vectors and vanishing}
		Let $\Vec{e},~\vec{f}\in\Z^k_{\geq 0}$ such that $|\vec{e}|=|\vec{f}|$
		%\[ \sum_{i=0}^{k-1} e_i = \sum_{i=0}^{k-1}f_i\]
		and $\textbf{a}(\Vec{e}) \leq \textbf{a}(\Vec{f})$ entry-wise. Then we have 
		\[ \mathrm{ord}(\Vec{f}) \succeq \mathrm{ord}(\Vec{e}).  \]
	\end{lemma}
	\begin{proof} Since $|\vec{e}|=|\vec{f}|$, our conclusion is equivalent to showing that the sum of the last $x$ entries of $\textrm{ord}(\Vec{f})$ is smaller or equal to the sum of the last $x$ entries of $\textrm{ord}(\Vec{e})$, for every $1\leq x \leq k$.
		
		Denote by $A_i$ and $B_i$ the number of entries of $\Vec{e}$ and $\Vec{f}$ equal to $i$. 
		Our first observation is that $A_0 \leq B_0$ since these values count the gaps in $\textbf{a}(\Vec{e})_{< 2k}$ and $\textbf{a}(\Vec{f})_{< 2k}$, respectively, and the entry-wise inequality $\textbf{a}(\Vec{e}) \leq \textbf{a}(\Vec{f})$ implies there are more gaps in $\textbf{a}(\Vec{f})_{< 2k}$. We now consider a number $1\leq x \leq B_0$ and sum the last $x$ elements of $\textrm{ord}(\Vec{f})$ and $\textrm{ord}(\Vec{e})$, respectively. 
		
		Because $x \leq B_0$, the sum of the smallest $x$ entries in $\textrm{ord}(\Vec{f})$ is $0$. For $\textrm{ord}(\Vec{e})$, the sum of its smallest $x$ entries is at least 
		\[ 0\cdot \min\{A_0, x\} + 1\cdot (x-\min\{A_0, x\}). \]
		It is clear that $x-\min\{A_0, x\} \geq 0$. \vspace{2em}
		\\ 
		\noindent
		\begin{minipage}{0.25\textwidth}
			\begin{tikzpicture}
				% Draw vertical lines
				\foreach \x in {0,...,5}
				\draw[gray] (\x,0) -- (\x,7);
				
				% Draw horizontal lines
				\foreach \y in {0,...,7}
				\draw[gray] (0,\y) -- (5,\y);
				
				\draw[very thick] (0,1) -- (1,1);
				\draw[very thick] (1,2) -- (1,1);
				\draw[very thick] (1,2) -- (2,2);
				\draw[very thick] (2,2) -- (2,6);
				\draw[very thick] (2,6) -- (3,6);
				\draw[very thick] (3,6) -- (3,5);
				\draw[very thick] (3,5) -- (4,5);
				\draw[very thick] (4,5) -- (4,6);
				\draw[very thick] (4,6) -- (5,6);
				
				\node at (0.5,6.5) {\Large 0};
				\node at (1.5,6.5) {\Large 1};
				\node at (2.5,6.5) {\Large 2};
				\node at (3.5,6.5) {\Large 3};
				\node at (4.5,6.5) {\Large 4};
				\node at (0.5,5.5) {\Large 5};
				\node at (1.5,5.5) {\Large 6};
				\node at (3.5,5.5) {\Large 8};
				\node at (0.5,4.5) {\Large 10};
				\node at (1.5,4.5) {\Large 11};
				\node at (0.5,3.5) {\Large 15};
				\node at (1.5,3.5) {\Large 16};
				\node at (0.5,2.5) {\Large 20};
				\node at (1.5,2.5) {\Large 21};
				\node at (0.5,1.5) {\Large 25};
				
				\fill[pattern=north east lines, pattern color=gray] (2,5) rectangle (3,6);
				\fill[pattern=north east lines, pattern color=gray] (4,5) rectangle (5,6);
				\fill[pattern=north east lines, pattern color=gray] (2,4) rectangle (3,5);
				\fill[pattern=north east lines, pattern color=gray] (3,4) rectangle (4,5);
				\fill[pattern=north east lines, pattern color=gray] (4,4) rectangle (5,5);
				
			\end{tikzpicture}
		\end{minipage}
		\hfill
		\begin{minipage}{0.6\textwidth} 
			This is an example for $k=5$ and $\Vec{e} = (5,4,0,1,0)$. The vanishing orders in $\textbf{a}(\Vec{e})$ are the numbers appearing in the table, while the dashed squares represent the gaps in $\textbf{a}(\Vec{e})_{< 3k}$, i.e. the numbers smaller than $3k$ not appearing in $\textbf{a}(\Vec{e})$. \\ \ \\
			The number of gaps is computed as $2A_0 + A_1$ = $2\cdot 2 +1$. 
		\end{minipage} 
		\ \\
		
		Now we prove this more generally. By looking at the number of gaps in $\textbf{a}(\Vec{e})_{< (n+2)k}$ and $\textbf{a}(\Vec{f})_{< (n+2)k}$, and comparing the two numbers, we obtain the inequality 
		\[ (n+1)A_0 + nA_1 + \cdots + A_n \leq (n+1)B_0 + nB_1 + \cdots + B_n.   \]
		
		Let $1 \leq x \leq B_n$ and look at the sum of the last $B_0 + B_1 + \cdots + B_{n-1} + x$ entries of $\textrm{ord}(\Vec{f})$ and $\textrm{ord}(\Vec{e})$, respectively. The sum of the last $B_0 + B_1 + \cdots + B_{n-1} + x$ entries in $\textrm{ord}(\Vec{f})$ is equal to 
		\[  B_1 + \cdots + (n-1)B_{n-1} + nx.\]
		On the other  hand, for $\textrm{ord}(\Vec{e})$, the sum is greater or equal to the value 
		\[ m_1 + 2 m_2 + \cdots + (n-1)m_{n-1} + n(B_0 + \cdots +B_{n-1}+x - A_0 - m_1-\cdots - m_{n-1}), \]
		where $m_1  = \max\{0, \min\{A_1, B_0 + \cdots + B_{n-1}+ x - A_0\} \}$ and $m_j$ is defined inductively by 
		\[ m_j = \max\{0,\min\{A_j, B_0 + \cdots + B_{n-1}+ x - A_0 - \sum_{i=1}^{j-1}m_i\}\}. \]
		In order to conclude $\textrm{ord}(\Vec{f}) \succeq \textrm{ord}(\Vec{e})$, it suffices to prove that $B_1 + \cdots + (n-1)B_{n-1} + nx$ is the smallest of the two sums. After simplifications, this is equivalent to proving
		\[ \sum_{i=1}^{n-1} (n-i+1)B_{i-1} \geq nA_0 + \sum_{i=1}^{n-1} (n-i) m_i. \]
		Using the inequality $ m_i \leq A_i$, as well as 
		\[nB_0 + (n-1)B_1 + \cdots + B_{n-1} \geq n A_0 + (n-1)A_1 + \cdots + A_{n-1},\]
		we deduce our conclusion. Hence, the sum of the last $x$ entries of $\textrm{ord}(\Vec{f})$ is smaller or equal to the sum of the last $x$ entries of $\textrm{ord}(\Vec{e})$, for every \[ x \in [1, B_0]\cup [B_0+1, B_0+B_1]\cup [B_0+B_1+1, B_0+B_1+B_2]\cup\cdots = \{1,2,\ldots, k\}.\]
		This is exactly what we wanted to prove.
	\end{proof}
	
	Let $\Vec{e} = (e_0, \ldots, e_{k-1})$ be a vector of non-negative entries. In order to compute $\dim W^{\Vec{e}}(\pi,p)$, we first need to introduce some notation.
	
	\begin{defn} \label{def: functions} 
		We define the functions
		\[ u(\Vec{e}) \coloneqq \sum_{e_i<e_j} (e_j-e_i-1) \ \textrm{and}\]
		\[ \textrm{inv}(\Vec{e}) \coloneqq \#\{(i,j) \ \mid\ 0\leq i < j\leq k-1 \ \textrm{and} \ e_i <e_j\}. \] 
	\end{defn}
	Notice that our definition differs from the standard way the number of inversions is defined. This is because $\dim W^{\Vec{e}}(\pi,p)$ depends on how far the vector is from being decreasing.
	
	By degenerating to the boundary, computing the upper-bound for the dimension will reduce to high-school level combinatorics. We first present the result in the form in which it was proposed to the final competition of the International Mathematical Summer Camp $2025$, and then translate it back into the language of Brill--Noether Theory.
	\begin{prop} \label{prop: ramification game}
		Let $\Vec{e} = (e_0,e_1,\ldots, e_{k-1})$ be a vector of non-negative integers and let $N$ be a positive integer satisfying $N > \max\left\{e_i \ | \ 0\leq i \leq k-1 \right\}$. We consider a $N\times k$ board with its cells numbered left to right, top to bottom with the numbers $0, 1,\ldots, kN-1$ in increasing order. For each $0\leq i \leq k-1$, we add a coin in each of the uppermost $e_i+1$ squares of column $i$. 
		
		A move consists in choosing a column, and trying to move all the coins in that column to the squares numbered one less, i.e. 
		\begin{itemize}
			\item a coin in that column is moved to the position numbered one less if that position is free; 
			\item a coin in that column is left in place if the position numbered one less is already occupied by a coin. 
		\end{itemize}
		Then we always need to perform exactly 
		\[  u(\Vec{e}) + \mathrm{inv}(\Vec{e}) \]
		nontrivial moves in order to arrange the coins on the squares numbered $0,1, \ldots, k-1 + \sum_{i=0}^{k-1}e_i$.
	\end{prop}
	\begin{proof}
		The proof is a simple induction on the value $\ u(\Vec{e}) + \mathrm{inv}(\Vec{e})$. This value is equal to $0$ if and only if the coins are on the squares numbered $0,1, \ldots, k-1 + \sum_{i=0}^{k-1}e_i$. This proves the case when the value is $0$. We will show that $\ u(\Vec{e}) + \mathrm{inv}(\Vec{e})$ is a semi-invariant decreasing by exactly one after each move.
		
		When $u(\Vec{e}) + \mathrm{inv}(\Vec{e}) > 0$, there are two possible cases: 
		
		\underline{Case~1:} If we choose the column $i$ for $1\leq i \leq k-1$, then the new configuration of coins is the one generated by the vector 
		\[ t_i(\Vec{e}) \coloneqq (e_0, e_1, \ldots, e_{i-2}, e_i, e_{i-1}, \ldots, e_{k-1}).\]
		We have $u(\Vec{e}) = u(t_i(\Vec{e}))$ while the number of inversions decreases by exactly one. The induction hypothesis concludes the proposition in this case. 
		
		\underline{Case~2:} We choose column $0$ and we try to move its coins to column $k-1$. In this case, the new configuration is the one generated by the vector 
		\[ t_0(\Vec{e}) \coloneqq (e_{k-1}+1, e_1, e_2, \ldots, e_{k-2}, e_0-1) \]
		and we immediately check the equality
		\begin{align*}
			u(\Vec{e}) - u(t_0(\Vec{e})) = &\#\left\{ 1\leq j \leq k-2 \ \mid \ e_j+1 < e_0 \right\}  + \#\left\{ 1\leq j \leq k-2 \ \mid \ e_j-1 > e_{k-1} \right\} - \cdots \\
			&-\#\left\{ 1\leq j \leq k-2 \ \mid \ e_j > e_0 \right\} - \#\left\{ 1\leq j \leq k-2 \ \mid \ e_j < e_{k-1}\right\} + \cdots \\
			& + e_0-e_{k-1} - 1- \mathrm{max}\left\{0, e_0-e_{k-1} -3\right\}.
		\end{align*}
		To see this, we look at the contribution to $u(\Vec{e})$ of pairs $(e_0, e_j)$ and $(e_{k-1},e_j)$, and describe how it changes when we both decrease $e_0$ and increase $e_{k-1}$ by $1$. 
		
		On the other hand, a similar argument implies 
		\begin{align*}
			\mathrm{inv}(\Vec{e}) - \mathrm{inv}(t_0(\Vec{e})) =
			&\#\left\{ 1\leq j \leq k-2 \ \mid \ e_j > e_0 \right\} + \#\left\{ 1\leq j \leq k-2 \ \mid \ e_j < e_{k-1}\right\} - \cdots \\
			&-\#\left\{ 1\leq j \leq k-2 \ \mid \ e_j < e_0-1 \right\}  - \#\left\{ 1\leq j \leq k-2 \ \mid \ e_j > e_{k-1} +1\right\} - \cdots \\
			& - \begin{cases}
				1 \ \mathrm{if} \ e_0-e_{k-1} \geq 3 \\
				0 \ \mathrm{if} \ e_0-e_{k-1} = 2.
			\end{cases}
		\end{align*}
		
		Adding the two equalities together, we obtain 
		\[ u(t_0(\Vec{e})) + \mathrm{inv}(t_0(\Vec{e})) = u(\Vec{e}) + \mathrm{inv}(\Vec{e}) - 1.\]
		This means that each possible move reduces this semi-invariant by exactly $1$, and hence we need exactly $u(\Vec{e}) +\mathrm{inv}(\Vec{e})$ moves until the coins are on the squares numbered 
		\[0,1, \ldots, k-1 + |\vec{e}|\]%\sum_{i=0}^{k-1}e_i,\]
		i.e. until the coins are in the unique position where no further moves are possible.
	\end{proof}

	There is another way in which we can interpret the semi-invariant $u(\Vec{e})+ \textrm{inv}(\Vec{e})$. For this purpose, we define the displacement number of a coin as the number written in the cell it occupies minus its position in the ordered list of coins.

	\begin{remark} \label{remark: southernomst description}
		Let $\Vec{e} = (e_0,e_1,\ldots, e_{k-1})$ be a vector of positive integers and consider the same squares occupied with coins as in Proposition~\ref{prop: ramification game}. Then the value $u(\Vec{e})+ \textrm{inv}(\Vec{e})$ can be alternatively obtained by summing up the displacement number for the southernmost coin in each column. We clarify what we mean by this with the following example:
		\ \\ 
		\noindent
		\begin{figure}[H]
			\begin{subfigure}{0.33\textwidth}
				\begin{tikzpicture}[scale = 0.8]
					% Draw vertical lines
					\foreach \x in {0,...,5}
					\draw[gray] (\x,0) -- (\x,7);
					
					% Draw horizontal lines
					\foreach \y in {0,...,7}
					\draw[gray] (0,\y) -- (5,\y);
					
					\draw[very thick] (0,1) -- (1,1);
					\draw[very thick] (1,2) -- (1,1);
					\draw[very thick] (1,2) -- (2,2);
					\draw[very thick] (2,2) -- (2,6);
					\draw[very thick] (2,6) -- (3,6);
					\draw[very thick] (3,6) -- (3,5);
					\draw[very thick] (3,5) -- (4,5);
					\draw[very thick] (4,5) -- (4,6);
					\draw[very thick] (4,6) -- (5,6);
					
					\node at (0.5,6.5) {\Large 0};
					\node at (1.5,6.5) {\Large 1};
					\node at (2.5,6.5) {\Large 2};
					\node at (3.5,6.5) {\Large 3};
					\node at (4.5,6.5) {\Large 4};
					\node at (0.5,5.5) {\Large 5};
					\node at (1.5,5.5) {\Large 6};
					\node at (3.5,5.5) {\Large 8};
					\node at (0.5,4.5) {\Large 10};
					\node at (1.5,4.5) {\Large 11};
					\node at (0.5,3.5) {\Large 15};
					\node at (1.5,3.5) {\Large 16};
					\node at (0.5,2.5) {\Large 20};
					\node at (1.5,2.5) {\Large 21};
					\node at (0.5,1.5) {\Large 25};
					
				\end{tikzpicture} \caption*{Initial position}
			\end{subfigure} 
			\centering
			\begin{subfigure}{0.33\textwidth}
				\begin{tikzpicture}[scale = 0.8]
					% Draw vertical lines
					\foreach \x in {0,...,5}
					\draw[gray] (\x,0) -- (\x,7);
					
					% Draw horizontal lines
					\foreach \y in {0,...,7}
					\draw[gray] (0,\y) -- (5,\y);
					
					\draw[very thick] (0,1) -- (1,1);
					\draw[very thick] (1,2) -- (1,1);
					\draw[very thick] (1,2) -- (2,2);
					\draw[very thick] (2,2) -- (2,6);
					\draw[very thick] (2,6) -- (3,6);
					\draw[very thick] (3,6) -- (3,5);
					\draw[very thick] (3,5) -- (4,5);
					\draw[very thick] (4,5) -- (4,6);
					\draw[very thick] (4,6) -- (5,6);
					
					\node at (0.5,6.5) {\Large 0};
					\node at (1.5,6.5) {\Large 1};
					\node at (2.5,6.5) {\Large 2};
					\node at (3.5,6.5) {\Large 3};
					\node at (4.5,6.5) {\Large 4};
					\node at (0.5,5.5) {\Large 5};
					\node at (1.5,5.5) {\Large 6};
					\node at (3.5,5.5) {\Large 7};
					\node at (0.5,4.5) {\Large 8};
					\node at (1.5,4.5) {\Large 9};
					\node at (0.5,3.5) {\Large 10};
					\node at (1.5,3.5) {\Large 11};
					\node at (0.5,2.5) {\Large 12};
					\node at (1.5,2.5) {\Large 13};
					\node at (0.5,1.5) {\Large 14};
					
				\end{tikzpicture} \caption*{Final position}
			\end{subfigure} 
			\hfill 
			\begin{subfigure}{0.30\textwidth}
				\begin{tikzpicture}[scale = 0.8]
					% Draw vertical lines
					\foreach \x in {0,...,5}
					\draw[gray] (\x,0) -- (\x,7);
					
					% Draw horizontal lines
					\foreach \y in {0,...,7}
					\draw[gray] (0,\y) -- (5,\y);
					
					\draw[very thick] (0,1) -- (1,1);
					\draw[very thick] (1,2) -- (1,1);
					\draw[very thick] (1,2) -- (2,2);
					\draw[very thick] (2,2) -- (2,6);
					\draw[very thick] (2,6) -- (3,6);
					\draw[very thick] (3,6) -- (3,5);
					\draw[very thick] (3,5) -- (4,5);
					\draw[very thick] (4,5) -- (4,6);
					\draw[very thick] (4,6) -- (5,6);
					
					\node at (0.5,6.5) {\Large 0};
					\node at (1.5,6.5) {\Large 0};
					\node at (2.5,6.5) {\Large 0};
					\node at (3.5,6.5) {\Large 0};
					\node at (4.5,6.5) {\Large 0};
					\node at (0.5,5.5) {\Large 0};
					\node at (1.5,5.5) {\Large 0};
					\node at (3.5,5.5) {\Large 1};
					\node at (0.5,4.5) {\Large 2};
					\node at (1.5,4.5) {\Large 2};
					\node at (0.5,3.5) {\Large 5};
					\node at (1.5,3.5) {\Large 5};
					\node at (0.5,2.5) {\Large 8};
					\node at (1.5,2.5) {\Large 8};
					\node at (0.5,1.5) {\Large 11};
					
				\end{tikzpicture} \caption*{Displacement number}
			\end{subfigure}
		\end{figure}
		\ \\
		For $\Vec{e} = (5,4,0,1,0)$ we have $u(\Vec{e}) = 19$, $\textrm{inv}(\Vec{e}) = 1$ and 
		\[ u(\Vec{e}) + \textrm{inv}(\Vec{e}) = 11 + 8 + 0 +1 +0. \]
		The proof of this statement is identical to the proof of Proposition~\ref{prop: ramification game}: this semi-invariant decreases by exactly one after each move. 
	\end{remark}
	
	One of our goals is to prove that, if we have the entry-wise inequality $\textbf{a}(\Vec{e}) \leq \textbf{a}(\Vec{f})$, we have the inclusion of pointed Hurwitz--Brill--Noether loci $W^{\Vec{f}}(\pi,p) \subseteq \overline{W}^{\Vec{e}}(\pi,p) $. For our dimension result in Theorem~\ref{trm: dimension}, we will need just an inequality of expected dimensions, derived from the next proposition. 
	
	\begin{prop} \label{prop: vanishing and loci comparison}
		Let $\Vec{e},\Vec{f}\in\Z^k_{\geq 0}$ with $|\vec{e}|=|\vec{f}|$ 
		%\[ \sum_{i=0}^{k-1} e_i = \sum_{i=0}^{k-1}f_i\]
		and $\textbf{a}(\Vec{e}) \leq \textbf{a}(\Vec{f})$ entry-wise. Then we have the inequality
		\[ u(\Vec{e}) + \textrm{inv}(\Vec{e}) \leq  u(\Vec{f}) + \textrm{inv}(\Vec{f}).  \]
	\end{prop}
	
	\begin{proof}
		We will induct on the value $s = |\vec{e}|=|\vec{f}|$. When $s=0$, we must have $\Vec{e} = \Vec{f} = (0,\ldots, 0)$ and there is nothing to be proven. 
		
		Assume $s$ is the first value of the sum for which the proposition is false, and let $\Vec{e},\Vec{f}\in\Z^k_{\geq 0}$ be some vectors satisfying 
		\[ s=  |\vec{e}|=|\vec{f}|, \ \textbf{a}(\Vec{e}) \leq \textbf{a}(\Vec{f}) \ \textrm{entry-wise}, \ \textrm{and} \ u(\Vec{e})+ \textrm{inv}(\Vec{e}) > u(\Vec{f})+ \textrm{inv}(\Vec{f}). \]
		Moreover, we choose such vectors $\Vec{e}, \Vec{f}$ with $u(\Vec{f})+ \textrm{inv}(\Vec{f})$ minimal. 
		
		\begin{claim}\label{claim_1_vec_inequal}
			$f_0 = 0$.
		\end{claim} 
		\begin{proof}[Proof of Claim~\ref{claim_1_vec_inequal}]
			\renewcommand{\qedsymbol}{$\lozenge$}
			
			If $f_0\neq 0$, it follows that the first $k+1$ entries of $\textbf{a}(\Vec{f})$ are $(0,1,\ldots, k)$. The entry-wise inequality $\textbf{a}(\Vec{e}) \leq \textbf{a}(\Vec{f})$ implies that the first $k+1$ entries of $\textbf{a}(\Vec{e})$ are $(0,1,\ldots, k)$ and hence $e_0 \neq 0$. 
			
			Taking $\sigma(\Vec{e}) = (e_1, \ldots, e_{k-1}, e_0-1)$ and $\sigma(\Vec{f}) = (f_1, \ldots, f_{k-1}, f_0-1)$ we obtain $\textbf{a}(\sigma(\Vec{e})) \leq \textbf{a}(\sigma(\Vec{f}))$. Indeed, the sequence $\textbf{a}(\sigma(\Vec{e}))$ is obtained from $\textbf{a}(\Vec{e})$ by removing $0$ and reducing all other entries by $1$, so the entry-wise inequality is trivially preserved. 
			
			The minimality of $s$ now implies that 
			\[ u(\sigma(\Vec{e})) + \textrm{inv}(\sigma(\Vec{e})) \leq  u(\sigma(\Vec{f})) + \textrm{inv}(\sigma(\Vec{f})).  \]
			
			But the invariant $u + \textrm{inv}$ amounts for the sum of displacements for the southernmost entries in each column, see Remark~\ref{remark: southernomst description}. The process of removing $0$ and decreasing all entries by $1$ does not change any displacement number.
			\[ \Rightarrow u(\Vec{e}) + \textrm{inv}(\Vec{e}) = u(\sigma(\Vec{e})) + \textrm{inv}(\sigma(\Vec{e})) \leq  u(\sigma(\Vec{f})) + \textrm{inv}(\sigma(\Vec{f})) = u(\Vec{f}) + \textrm{inv}(\Vec{f}),  \]
			contradicting the choice of $\Vec{e}$ and $\Vec{f}$.
		\end{proof}

		Hence we must have $f_0 = 0$ and let $f_i$ be the first non-zero entry of $\Vec{f}$. Let $m \coloneqq \min\{e_i, f_i\}$ and consider the numbers $i, k+i, \ldots, m\cdot k + i$. By the choice of $m$, these numbers appear in both $\textbf{a}(\Vec{e})$ and $\textbf{a}(\Vec{f})$.
		
		For each $0 \leq x \leq m$, let $N(x,\Vec{e})$ be the position of $i + x\cdot k$ in $\textbf{a}(\Vec{e})$ and similarly $N(x,\Vec{f})$ the position of $i + x\cdot k$ in $\textbf{a}(\Vec{f})$.

		\begin{claim}\label{claim_2_vec_inequal}
			For each $0 \leq x \leq m$, we have $N(x,\Vec{e}) \geq N(x,\Vec{f})$. 
		\end{claim}
		\begin{proof}[Proof of Claim~\ref{claim_2_vec_inequal}]
			\renewcommand{\qedsymbol}{$\lozenge$}
			
			If $N(x,\Vec{e}) < N(x,\Vec{f})$, then the entry of $\textbf{a}(\Vec{e})$ in position $N(x,\Vec{f})$ is greater than or equal to
			\[
			i + x\cdot k + N(x,\Vec{f}) - N(x,\Vec{e}).
			\]
			In this case, the entry-wise inequality $\textbf{a}(\Vec{e}) \leq \textbf{a}(\Vec{f})$ fails at entry $N(x,\Vec{f})$. This contradiction proves the claim.
		\end{proof}

		Let $\Delta(x) \coloneqq N(x, \Vec{e}) - N(x,\Vec{f})$.
		
		\begin{claim}\label{claim_3_vec_inequal}
			There exists $1\leq x \leq m$ such that $\Delta(x) = 0$. 
		\end{claim}
		\begin{proof}
			\renewcommand{\qedsymbol}{$\lozenge$}
			
			If $\Delta(x) \neq 0$ for each $1\leq x \leq m$, then each of the numbers $ i + x\cdot k$ in $\textbf{a}(\Vec{f})$ is compared with a number in $\textbf{a}(\Vec{e})$ which is strictly lower. 
			
			Moreover, since for $m < x \leq f_i$ the number $i+x\cdot k$ does not appear in $\textbf{a}(\Vec{e})$, it follows that the number $i+x\cdot  k$ in $\textbf{a}(\Vec{f})$ is compared with a number in $\textbf{a}(\Vec{e})$ that is strictly lower. 
			
			Hence, reducing all entries $i+x\cdot k$ for $1\leq x \leq f_i$ in $\textbf{a}(\Vec{f})$ by $1$, the entry-wise inequality remains valid. In the notation of Definition~\ref{def: t_i}, this means $\textbf{a}(\Vec{e}) \leq \textbf{a}(t_i(\Vec{f}))$ entry-wise. 
			
			But we know 
			\[ u(t_i(\Vec{f})) + \textrm{inv}(t_i(\Vec{f})) = u(\Vec{f}) + \textrm{inv}(\Vec{f}) - 1\]
			and the assumption on minimality of $u(\Vec{f}) + \textrm{inv}(\Vec{f})$ implies 
			\[ u(\Vec{e})+ \textrm{inv}(\Vec{e}) \leq u(\Vec{f})+ \textrm{inv}(\Vec{f}) - 1. \]
			This contradicts the assumption, and hence the claim follows.
		\end{proof}

		Hence there exists values of $x$ between $1$ and $m$ such that $\Delta(x) = 0$. For such an $x$, both sequences $\textbf{a}(\Vec{e})$ and $\textbf{a}(\Vec{f})$ have the number $i + x\cdot k$ written on entry $N(x, \Vec{e}) = N(x,\Vec{f})$. 
		
		We look at the entry $N(x,\Vec{f}) - 1$ in each of the two sequences. Because $f_{i-1} = 0$, the number of $\textbf{a}(\Vec{f})$ on position $N(x,\Vec{f})-1$ is strictly smaller than $i-1 +x\cdot k$. This implies that the entry of $\textbf{a}(\Vec{e})$ on position $N(x,\Vec{f})-1$ is strictly smaller than $i-1+x\cdot k$. 
		
		This immediately implies the entry-wise inequality 
		\[ \textbf{a}(t_i(\Vec{e})) \leq \textbf{a}(t_i(\Vec{f})): \]
		When $\Delta(x) \neq 0$ or $x > m$, the same argument as for Claim~\ref{claim_3_vec_inequal} works. When $\Delta(x) = 0$, switching the entries $e_{i-1}$ and $e_i$ (as well as $f_{i-1}$ and $f_i$) just reduces the number  on position $N(x,\Vec{e}) = N(x,\Vec{f})$ by $1$ in both sequences. 
		
		The minimality of $u(\Vec{f}) + \textrm{inv}(\Vec{f})$ implies 
		\[u(t_i(\Vec{e})) + \textrm{inv}(t_i(\Vec{e})) = u(\Vec{e})+\textrm{inv}(\Vec{e}) - 1 \leq u(t_i(\Vec{f})) + \textrm{inv}(t_i(\Vec{f})) = u(\Vec{f})+\textrm{inv}(\Vec{f}) - 1. \]
		This contradicts the assumption on $\Vec{e}$ and $\Vec{f}$. In particular, we obtain that there do not exist non-negative vectors $\Vec{e}$ and $\Vec{f}$ satisfying 
		\[ |\vec{e}|=|\vec{f}|, \ \textbf{a}(\Vec{e}) \leq \textbf{a}(\Vec{f}) \ \textrm{entry-wise} \ \textrm{and} \ u(\Vec{e}) + \textrm{inv}(\Vec{e}) > u(\Vec{f}) + \textrm{inv}(\Vec{f}). \]
		
		Moreover, the same inductive method implies that if $u(\Vec{e}) + \textrm{inv}(\Vec{e}) = u(\Vec{f}) + \textrm{inv}(\Vec{f})$ then $\Vec{e} = \Vec{f}$.
	\end{proof}

	\section{Dimension results} \label{sec: dimension results}
	
	The goal of this section is to compute the dimension of pointed Hurwitz--Brill--Noether loci $W^{\vec{e}}(\pi,p)$ for a generic element $\elhur \in \cHTot$. The main result of this section will be that 
	\[ \dim W^{\vec{e}}(\pi,p) = g -u(\Vec{e})-\inv(\Vec{e})\]
	whenever this value is non-negative, and the locus is empty otherwise. 
	
	The upper bound $\dim W^{\vec{e}}(\pi,p)  \leq g -u(\Vec{e})-\inv(\Vec{e})$ is obtained via a degeneration of the pointed Hurwitz--Brill--Noether locus to the boundary of $\hhTot$ and description of the limit linear series appearing in this way, given by Lemma~\ref{lemma: ramifications}. Proposition~\ref{prop: vanishing and loci comparison} together with an induction argument will then imply our inequality. 
	
	To obtain a lower-bound on the dimension, we realize $W^{\Vec{e}}(\pi, p)$ as an open set of a flag degeneracy locus. We will realize $W^{\Vec{e}}(\pi, p)$ as such a locus in two different ways. The first, and simpler proof, follows techniques developed in~\cite{Larson2Vogt} and uses $\vec{e}$-nested linear series. The second proof uses Proposition~\ref{prop: ramification behaviour} to control the ramification orders and realizes the pointed Hurwitz--Brill--Noether locus as an intersection of flag degeneracy loci. This second proof is independent of~\cite{Larson2Vogt}.

	\begin{prop} \label{prop: upperbound}
		Let $[\pi\colon C \rightarrow \mathbb{P}^1,p]$ be a generic element of $\cHTot$, $\Vec{e} = (e_0, \ldots, e_{k-1})\in\Z^k_{\ge 0}$, $r = k+ |\vec{e}| - 1$, and $d = g+r$. 
		Then the locus of $g^r_d$'s on $C$ having ramification orders exactly $\textbf{a}(\Vec{e})$ at $p$ has dimension at most $g- u(\Vec{e})- \mathrm{inv}(\Vec{e})$, that is
		\[ \dim W^{\Vec{e}}(\pi,p) \leq g- u(\Vec{e}) - \textrm{inv}(\Vec{e}). \]
		In particular, when the expected dimension is negative, the pointed Hurwitz--Brill--Noether locus is the empty set. 
		%Moreover, for a generic such $g^r_d$, the vanishing orders at $p$ are exactly given by $\textbf{a}(\Vec{e})$. 
	\end{prop}
	\begin{proof}
		We induct on the genus. When $C$ has genus $1$, this is immediately verified. 
		
		For genus $g\geq 2$, we look at an element $[\pi_0\colon X\cup_q E
		\rightarrow P_X\cup_{\overline q}P_E, p]$ as in Setup~\ref{setup: elliptic tail}. This is an element appearing in the admissible cover compactification of the totally ramified Hurwitz space $\hhTot$, see~\cite{AbramovichCV}. 
		
		We describe the limit $g^r_d$'s with ramification orders $\textbf{a}(\Vec{e})$ on the the pointed nodal curve $[X\cup_qE,p]$. This is the content of Lemma~\ref{lemma: ramifications}. Let $(L_X, L_E) \in \mathrm{Pic}^d(X)\times \mathrm{Pic}^d(E)$ be the line bundles associated to such a limit linear series $(l_X, l_E)$. Then  
		\begin{enumerate}
			\item If $L_E$ is not a line bundle of the form $\mathcal{O}_E(i\cdot p + (d-i)\cdot q)$, then the vanishing orders of $l_X$ at the point $q$ are greater than or equal to $1+\textbf{a}(\Vec{e})$. As in Lemma~\ref{lemma: ramifications}, $l_X(-q)$ determines a $g^r_{d-1}$ on $X$ with ramification orders at least $\textbf{a}(\Vec{e})$ at $q$. Using Proposition~\ref{prop: vanishing and loci comparison}, we obtain that $(L_X, L_E)$ is contained in the locus 
			\[\left(\mathcal{O}_X(q)\otimes W_{d-1, \textbf{a}(\Vec{e})}^r(X,q)\right) \times \mathrm{Pic}^d(E). \]
			Using induction and Proposition~\ref{prop: vanishing and loci comparison}, we find that all components of $W_{d-1, \textbf{a}(\Vec{e})}^r(X,q)$ have dimension at most $g-1-u(\Vec{e})- \textrm{inv}(\Vec{e})$. Hence, each component of 
			\[\left(\mathcal{O}_X(q)\otimes W_{d-1, \textbf{a}(\Vec{e})}^r(X,q)\right) \times \mathrm{Pic}^d(E) \]
			has dimension at most $g-1-u(\Vec{e})-\textrm{inv}(\Vec{e}) +1$, as required. 
			
			\item If $L_E \cong \mathcal{O}_E(i\cdot p + (d-i)\cdot q)$, it follows from Lemma~\ref{lemma: ramifications} that $L_X(-q)$ is contained in the locus $W^r_{d-1, \mathbf{a}(t_i(\Vec{e}))}(X,q)$. 
			
			In particular, the locus of such pairs $(L_X, L_E)$ is contained in 
			\[ \text{\Large (}\OO_X(q) \otimes W^r_{d-1,\mathbf{a}(t_i(\Vec{e}))}(X,q) \text{\Large)} \times \{ \mathcal{O}_E(i\cdot p + (d-i)\cdot q)\}. \]
		\end{enumerate}
		If we order the ramification orders in columns modulo $k$ as in Proposition~\ref{prop: ramification game}, we see that $\mathbf{a}(t_i(\Vec{e}))$ is obtained from $\textbf{a}(\Vec{e})$ by moving all ramifications in column $i \pmod{k}$ one place to the left when this is possible. This implies $u(t_i(\Vec{e})) + \mathrm{inv}(t_i(\Vec{e})) = u(\Vec{e}) + \mathrm{inv}(\Vec{e}) -1$, cf. Proposition~\ref{prop: ramification game}. In particular, using Proposition~\ref{prop: vanishing and loci comparison} and the induction hypothesis, the locus of limit $g^r_d$'s with ramification orders at least $\textbf{a}(\Vec{e})$ at $p$ has dimension at most $g-u(\Vec{e})-\textrm{inv}(\Vec{e})$, whence
		\[\dim W^r_{d,\textbf{a}(\Vec{e})}(C,p) \leq g-u(\Vec{e})-\textrm{inv}(\Vec{e}). \qedhere \]
	\end{proof}
	
	Our next goal is to prove the converse inequality.
	
	\begin{prop} \label{prop: lowerbound}
		Let $[\pi\colon C \rightarrow \mathbb{P}^1,p]$ be an element of $\cHTot$, $\Vec{e} = (e_0, \ldots, e_{k-1})\in\Z^k_{\geq 0}$, $r = k + |\vec{e}| - 1$, and $d = g+r$. 
		Then the locus of $g^r_d$'s on $C$ having ramification orders exactly $\textbf{a}(\Vec{e})$ at $p$ is either empty or has dimension at least
		\[ g- u(\Vec{e})- \mathrm{inv}(\Vec{e}),\]
		that is $\dim W^{\Vec{e}}(\pi,p) \geq g- u(\Vec{e})- \mathrm{inv}(\Vec{e}) $.
		%Moreover, for a generic such $g^r_d$, the vanishing orders at $p$ are exactly given by $\textbf{a}(\Vec{e})$. 
	\end{prop}
	
	We will provide two proofs of this result. Prior to the first proof, we recall some useful definitions of nested linear series from \cite[Section~6]{Larson2Vogt} and set some notation we will use. Our definitions differ very slightly as we retain the ordering of $\vec{e}$.

	Given $\vec{e}\in\Z^k$, let $d_1>\cdots>d_s\in \Z$ be the distinct entries of $\vec{e}$, and let \[M_j\coloneqq \{i\in [0,k-1] ~\mid~ e_i=d_j\}\] record the positions of $\vec{e}$ where $d_j$ appears, let $k_j=|M_j|$ be the multiplicity of $d_j$ in $\vec{e}$, and denote the entries of $M_j$ by \[M_j=\{m_{j,1}<m_{j,2}<\cdots<m_{j,k_{j}}\}.\] 
	
	Let $\pi \colon C\to\PP^1$ be a degree-$k$ cover and, for a line bundle $L$ on
	$C$, write
	\[
	L(t)\coloneqq L\otimes \pi^*\OO_{\PP^1}(t).
	\]
	As observed in \cite[Section~6]{Larson2Vogt}, the
	condition that $\pi_*L$ have splitting type $\vec{e}$ or a specialization
	thereof can be expressed through a collection of subspaces of
	$H^0(C,L(-d_j))$ that are compatible under multiplication by sections on
	$\PP^1$.
	
	\begin{defn}
		An \emph{$\vec{e}$-nested linear series} on $C$ is a tuple $(L,V_1,\ldots,V_s)$, where $L$ is a line bundle on $C$ and
		\[
		V_j\subseteq H^0(C,L(-d_j)),
		\qquad
		\dim V_j
		=
		h^0\bigl(\PP^1,\OO_{\PP^1}(\vec{e})(-d_j)\bigr)
		\]
		for every $1\leq j\leq s$, such that, for every $2\leq j\leq s$, the
		image of the natural multiplication map
		\[
		V_{j-1}\otimes
		H^0\bigl(\PP^1,\OO_{\PP^1}(d_{j-1}-d_j)\bigr)
		\longrightarrow
		H^0(C,L(-d_j))
		\]
		is contained in $V_j$. For any pair $1\leq i<j\leq s$, we denote by $\operatorname{Im}_j(V_i)\subseteq V_j$ the image of the multiplication map
		\[
		V_i\otimes
		H^0\bigl(\PP^1,\OO_{\PP^1}(d_i-d_j)\bigr)
		\longrightarrow
		H^0(C,L(-d_j)).
		\]
		By convention, we set $V_0=\{0\}$.
	\end{defn}
	
	This definition motivates Proposition~\ref{prop: ramification behaviour}. If $\elhur \in \cHTot$, $i\leq j$ and $V_i$ has a vanishing order $a$ at $p$, then there are sections vanishing with order $a,a+k,\dots,a+k(d_{i}-d_{j})$ at $p$ in $\mathrm{Im}_j(V_{i})\subset V_j$.
	
	\begin{defn}[{\cite[Section~6]{Larson2Vogt}}]
		Let $\elhur \in \cHTot$. The \emph{new vanishing orders at layer $j$}
		are the vanishing orders of $V_j$ at $p$ that do not occur among
		the vanishing orders of $\operatorname{Im}_j(V_{j-1})$ at $p$.
		By convention, all vanishing orders at the first layer are new. We say an $\vec{e}$-nested linear series is \emph{non-colliding at $p$}
		if, for every $1\leq j\leq s$, the new vanishing orders at layer
		$j$ are pairwise incongruent modulo $k$ and are incongruent modulo
		$k$ to every vanishing order occurring at any lower layer $V_i$,
		where $1\leq i<j$.
		
		We will denote by $G^{\Vec{e}, \textrm{nc}}(\pi, p)$ to be the space of non-colliding $\Vec{e}$-nested linear series.
	\end{defn}
	
	\begin{proof}[{Proof~1 of Proposition~\ref{prop: lowerbound}}]\label{proof_1_of_dim}
		
		Let $(L,V)$ be an element of $G^{r, \circ}_{d, \textbf{a}(\Vec{e})}(C,p)$. Our first observation is that $V = H^0(C,L)$. Indeed, if $V\subsetneq H^0(C,L)$ then the sections of $L$ have at least $r+2$ distinct vanishing orders, with $0,1,\ldots, k-1$ being $k$ of them. Then, for any positive integer $i$ we have 
		\[ h^0(C, L(ik\cdot p)) \geq ik + r+2, \]
		since its sections have at least that many distinct vanishing orders at $p$, see Proposition~\ref{prop: ramification behaviour}. On the other hand, for $i$ big enough, the Riemann-Roch Theorem implies 
		\[ h^0(C, L(ik\cdot p)) = ik+r+1,\] 
		a contradiction.
		
		%Next, we denote by $d_1 > d_2 >\cdots > d_s$ the distinct entries of $\Vec{e}$ and denote by $M_1, M_2, \ldots, M_s$ the positions in the vector $\Vec{e}$ where each $d_i$ appears. For each such set, we denote its elements as 
		%\[ M_j = \left\{ a_{j,1} < a_{j,2} < \cdots < a_{j,k_j}\right\}.\]
		
		Using the notation above and that $(L,V)$ is complete, we construct an injective map 
		\[ G^{r, \circ}_{d, \textbf{a}(\Vec{e})}(C,p) \rightarrow G^{\textrm{ord}(\Vec{e}), \textrm{nc}}(\pi\colon C\rightarrow \mathbb{P}^1,p) \]
		into the space of non-colliding $\textrm{ord}(\Vec{e})$-nested linear series, sending an element $(L,V)$ to the non-colliding $\textrm{ord}(\Vec{e})$-nested linear series $(L,V_1, V_2, \ldots, V_s)$ where $V_j = H^0(C, L(-d_j))$ for each $1\leq j \leq s$. Indeed, the natural maps 
		\[ V_{j-1} \otimes H^0(\mathbb{P}^1, \mathcal{O}_{\PP^1}(d_{j-1}-d_j)) \rightarrow V_j, \] show that $(L,V_1,\dots,V_s)$ is an $\mathrm{ord}(\vec{e})$-nested linear series, and by the definition of $V_i$ the new vanishing orders of $V_j$ not appearing as vanishing orders of $\mathrm{Im}_j(V_{j'})$ for $j'<j$ are exactly $M_j$, which are distinct modulo $k$, whereby $(L,V_1,\dots, V_s)$ is non-colliding.

		We consider the open subspace of $G^{\textrm{ord}(\Vec{e}), \textrm{nc}}(\pi,p)$ where the line bundles have splitting type $\Vec{e}$, so that we have $V_j = H^0(C, L(-d_j))$ for every $1 \leq j \leq s$. We will show the linear series with vanishing specified by $\textbf{a}(\vec{e})$ form a sublocus of codimension at most $\textrm{inv}(\Vec{e})$. We construct this locus in steps, for each $d_j$, and begin with $j=1$.
		
		The condition that $\textbf{a}(\Vec{e})_{\geq d_1k}$ are vanishing orders of $(L,V)$ at $p$ is equivalent to saying that the entries of $M_1$ are the vanishing orders of $V_1 = H^0(C, L(-d_1))$ at $p$. This is equivalent  to the inequalities 
		\[ h^0(C, L(-d_1k\cdot p-m_{1,i}\cdot p)) \geq 1+k_1-i \ \textrm{for each} \ 1\leq i \leq k_1 \]
		\[ \Leftrightarrow \textrm{rank}(V_1 \rightarrow H^0(C, L(-d_1k\cdot p)\vert_{m_{1,i}\cdot p})) \leq i-1 \ \textrm{for each} \ 1\leq i\leq k_1.\]
		
		The locus of $\textrm{ord}(\Vec{e})$-nested linear series satisfying these conditions can be interpreted as a flag degeneracy locus as follows: Let $\mathcal{L}$ be a Poincar\'e line bundle on the space $C\times G^{\textrm{ord}(\Vec{e}),\textrm{nc}}(\pi,p)$ and let $P = \{ p\}\times G^{\textrm{ord}(\Vec{e}),\textrm{nc}}(\pi,p)$. For each $1\leq j \leq s$ and $1\leq i \leq k_j$ we consider the short exact sequence of line bundles 
		\[ 0\rightarrow \mathcal{L}(-(d_jk+m_{j,i})\cdot P) \rightarrow \mathcal{L}(-d_jk\cdot P)  \rightarrow \mathcal{L}(-d_jk\cdot P) / \mathcal{L}(-(d_jk+m_{j,i})\cdot P) \rightarrow 0 \]
		and take the rank $m_{j,i}$ vector bundle
		\[ \mathcal{F}_{j,i} \coloneqq \pi_{2*}(\mathcal{L}(-d_jk\cdot P) / \mathcal{L}(-(d_jk+m_{j,i})\cdot P)), \]
		defined as a pushforward to $G^{\textrm{ord}(\Vec{e}),\textrm{nc}}(\pi,p)$.

		Let $\mathcal{V}_1$ denote the universal vector bundle on $G^{\mathrm{ord}(\vec e),\mathrm{nc}}(\pi,p))$, whose fiber over $(L,V_1,\ldots,V_s)$ is $V_1$. We have a flag of vector bundle morphisms 
		\[ \mathcal{V}_1 \rightarrow \mathcal{F}_{1,k_1}\twoheadrightarrow\mathcal{F}_{1,k_1-1}\twoheadrightarrow \cdots \twoheadrightarrow \mathcal{F}_{1,1}\]
		and the condition that the entries of $\textbf{a}(\Vec{e})_{\geq d_1k}$ are vanishing orders is equivalent to 
		\[ \textrm{rank}(\mathcal{V}_1 \rightarrow \mathcal{F}_{1,i}) \leq i-1 \ \textrm{for each} \ 1\leq i \leq k_1.  \]
		
		The sublocus of $G^{\textrm{ord}(\Vec{e}),\textrm{nc}}(\pi,p)$ satisfying these conditions is a flag degeneracy locus and its codimension is at most 
		\[ \sum_{i=1}^{k_1}(m_{1,i}-i+1) = \#\{ (m,n) \in M_1 \times (M_2\cup\cdots \cup M_s) \ | \  n < m\}. \]
		
		We take $G_1^{\textrm{ord}(\Vec{e}), \textrm{nc}}(\pi,p)$ to be the locus where $H^0(C,L)$ has the entries $\textbf{a}(\Vec{e})_{\geq d_1k}$ as vanishing orders. This is an open subspace of this flag degeneracy locus. We inductively define similar loci at each level $j$. We consider the locus $G^{\textrm{ord}(\Vec{e}), \textrm{nc}}_{j-1}(\pi,p)$ where the vanishing orders of $H^0(C,L)$ at $p$ that are at least $d_{j-1}k$ are exactly $\textbf{a}(\Vec{e})_{\geq d_{j-1}k}$. Inside this locus, we consider the sublocus where the vanishing orders that are at least $d_jk$ are given by $\textbf{a}(\Vec{e})_{\geq d_jk}$.
		
		In particular, for all $\textrm{ord}(\Vec{e})$-nested linear series $(L,V_1, V_2,\ldots, V_s)$ in this sublocus, the new vanishing orders of $V_j$ not appearing in $\textrm{Im}_j(V_{j-1})$ are the elements of $M_j$. Id est, the locus $G_j^{\textrm{ord}(\Vec{e}), \textrm{nc}}(\pi,p)$ is defined by the conditions 
		\[ \dim V_j(-m_{j,i}p) / \textrm{Im}_j(V_{j-1})(-m_{j,i}p) \geq k_j + 1 -i \ \textrm{for all} \ 1\leq i \leq k_j. \]
		We look at the morphisms 
		
		\[
		\begin{tikzcd}
			0\arrow{r}{} & \textrm{Im}_j(V_{j-1})(-m_{j,i}p) \arrow{r}{}   \arrow[d, hookrightarrow]  & \textrm{Im}_j(V_{j-1})  \arrow[d, hookrightarrow]  \arrow{r}{} & H^0(C, L(-d_jk\cdot p)\vert_{m_{j,i}\cdot p}) \arrow[d, equal]\\
			0\arrow{r}{} &H^0(C, L(-(d_jk+m_{j,i})\cdot p)) \arrow{r}{}& V_j \arrow{r}{\varphi_{j,i}} & H^0(C, L(-d_jk\cdot p)\vert_{m_{j,i}\cdot p}) 
		\end{tikzcd}
		\] 
		
		The conditions 
		\[ \dim V_j(-m_{j,i}p) / \textrm{Im}_j(V_{j-1})(-m_{j,i}p) \geq k_j + 1 -i \ \textrm{for all} \ 1\leq i \leq k_j \]
		above are equivalent, using Snake Lemma, to the maps 
		\[ V_j/\textrm{Im}_j(V_{j-1}) \rightarrow H^0(C, L(-d_jk\cdot p)\vert_{m_{j,i}\cdot p}) / \varphi_{j,i}(\textrm{Im}_j(V_{j-1}))\]
		having rank less or equal to $i-1$. 
		
		Letting $\mathcal{V}_j$ be the universal vector bundle on $G^{\mathrm{ord}(\vec e),\mathrm{nc}}(\pi,p))$, whose fiber over $(L,V_1,\ldots,V_s)$ is $V_j$, the morphism of vector spaces above comes from the morphism \[ \mathcal{V}_j/\mathrm{Im}_{j}(\mathcal{V}_{j-1}) \xrightarrow{\overline{\varphi}_{j,i}} \mathcal{F}_{j,i}/\varphi_{j,i}(\mathrm{Im}_{j}(\mathcal{V}_{j-1})) \] of vector bundles on $G^{\textrm{ord}(\Vec{e}), \textrm{nc}}_{j-1}(\pi,p)$.
		
		The two vector bundles have ranks $k_j$ and $m_{j,i} - \#\{ m\in M_1\cup M_2\cup \cdots \cup M_{j-1} \mid m < m_{j,i}\}$, respectively.
		
		These morphisms fit into a flag 
		\[ \mathcal{V}_j/\mathrm{Im}_{j}(\mathcal{V}_{j-1}) \xrightarrow{\overline{\varphi}_{j,{k_j}}} \mathcal{F}_{j,k_j}/\varphi_{j,k_j}(\mathrm{Im}_{j}(\mathcal{V}_{j-1})) \twoheadrightarrow \cdots \twoheadrightarrow \mathcal{F}_{j,1}/\varphi_{j,1}(\mathrm{Im}_{j}(\mathcal{V}_{j-1})). \]
		
		In particular, the sublocus of $G^{\textrm{ord}(\Vec{e}),\textrm{nc}}_{j-1}(\pi,p)$ with vanishing orders of $H^0(C,L)$ at $p$ given by $\textbf{a}(\Vec{e})_{\geq d_jk}$ is a flag degeneracy locus of codimension at most
		\[\sum_{i=1}^{k_j} (m_{j,i} - \#\{ m\in M_1\cup M_2\cup \cdots \cup M_{j-1} \mid m < m_{j,i}\} - i +1). \] 
		This value is equal to 
		\[ \textrm{inv}(M_j) = \#\{(m,n) \in M_j\times(M_{j+1}\cup\cdots\cup M_s) \mid n<m\}.  \]
		
		Consequently, the sublocus $G^{\textrm{ord}(\Vec{e}),\textrm{nc}}_{j}(\pi,p)$ of $G^{\textrm{ord}(\Vec{e}),\textrm{nc}}_{j-1}(\pi,p)$ is open in this flag degeneracy locus. 
		\[ \Rightarrow \textrm{codim}(G^{\textrm{ord}(\Vec{e}),\textrm{nc}}_{j}(\pi,p),  G^{\textrm{ord}(\Vec{e}),\textrm{nc}}_{j-1}(\pi,p)) \leq \textrm{inv}(M_j).\] 
		As a consequence 
		\[ \textrm{codim}(G^{\textrm{ord}(\Vec{e}),\textrm{nc}}_{s-1}(\pi,p),  G^{\textrm{ord}(\Vec{e}),\textrm{nc}}(\pi,p)) \leq \sum_{j=1}^{s-1}\textrm{inv}(M_j) = \textrm{inv}(\Vec{e}).\]
		
		But $G^{\textrm{ord}(\Vec{e}),\textrm{nc}}_{s-1}(\pi,p)$ is the image of $G^{r, \circ}_{d, \textbf{a}(\Vec{e})}(C,p)$ into $G^{\textrm{ord}(\Vec{e}),\textrm{nc}}(\pi,p)$ and hence the conclusion follows from $\textrm{dim}(G^{\textrm{ord}(\Vec{e}),\textrm{nc}}(\pi,p)) \geq g - u(\Vec{e})$, see \cite[Theorem~6.5]{Larson2Vogt}. 
	\end{proof}
	
	In fact, Proposition~\ref{prop: lowerbound} can be proven independently of the results in~\cite{Larson2Vogt}. As a consequence, our methods give a new proof of the main result of~\cite{Larson2Vogt} about the dimension of Hurwitz--Brill--Noether loci. 
	
	\begin{proof}[{Proof~2 of Proposition~\ref{prop: lowerbound}}]\label{proof_2_of_dim}
		Let $B(\Vec{e})$ be the set of essential vanishing orders, defined as 
		\[ B(\Vec{e}) \coloneqq \{ ke_0, ke_1 + 1, \ldots, ke_{k-1} + k-1\}, \]
		i.e. the southernmost entry in each column, in view of Remark~\ref{remark: southernomst description}. We order its elements in increasing order
		\[ B(\Vec{e}) = \{ b_0 < b_1 < \cdots < b_{k-1}\} = \{ke_{i_{0}} + i_{0} < ke_{i_{1}} + i_{1} < \cdots < ke_{i_{k-1}} + i_{k-1} \}.\] 
		and also denote by $\textbf{a}(\Vec{e})_{>b_i}$ the subset of $\textbf{a}(\Vec{e})$ containing just the elements strictly bigger than $b_i$. Moreover, we denote by $c_i = \min \textbf{a}(\Vec{e})_{>b_i}$ the smallest element of the set, for each $0 \leq i \leq k-2$. 
		
		Our goal is to realize the pointed Hurwitz--Brill--Noether locus $W^{\Vec{e}}(\pi,p) = W^{r,\circ}_{g+r, \textbf{a}(\Vec{e})}(C,p)$ as a degeneracy locus. For this, we will consider the conditions for the essential vanishing orders
		\[ h^0(C,L-(e_i\cdot k+i) \cdot p) \geq 1 + \sum_{e_j \geq e_i} (e_j-e_i) + \#\{j \ | \ j > i \ \textrm{and} \ e_j \geq e_i\} \ \textrm{for each} \ 0\leq i \leq k-1\]
		in an inductive manner, and emphasize that all other vanishing conditions are automatically satisfied because of Proposition~\ref{prop: ramification behaviour}.
		
		\ \\ 
		\noindent
		\begin{minipage}{0.25\textwidth}
			\begin{tikzpicture}
				% Draw vertical lines
				\foreach \x in {0,...,4}
				\draw[gray] (\x,0) -- (\x,6);
				
				% Draw horizontal lines
				\foreach \y in {0,...,6}
				\draw[gray] (0,\y) -- (4,\y);
				
				\draw[very thick] (0,3) -- (1,3);
				\draw[very thick] (1,3) -- (1,1);
				\draw[very thick] (1,1) -- (2,1);
				\draw[very thick] (2,1) -- (2,4);
				\draw[very thick] (2,4) -- (3,4);
				\draw[very thick] (3,4) -- (3,5);
				\draw[very thick] (3,5) -- (4,5);
				
				\node at (0.5,5.5) {\Large 0};
				\node at (1.5,5.5) {\Large 1};
				\node at (2.5,5.5) {\Large 2};
				\node at (2.5,4.5) {\Large 6};
				\node at (3.5,5.5) {\Large 3};
				\node at (0.5,4.5) {\Large 4};
				\node at (1.5,4.5) {\Large 5};
				\node at (0.5,3.5) {\Large 8};
				\node at (1.5,3.5) {\Large 9};
				\node at (1.5,2.5) {\Large 13};
				\node at (1.5,1.5) {\Large 17};
				
			\end{tikzpicture}
		\end{minipage}
		\hfill
		\begin{minipage}{0.6\textwidth} 
			This is an example for $k=4$ and $\Vec{e} = (2,4,1,0)$. We have $B(\Vec{e}) = \{3,6,8,17\}$. For $i= 0$, the term $\sum_{e_j \geq e_i} (e_j-e_i) = 2$ counts the vanishing orders on higher rows than $8 = 4\cdot e_0 + 0$. The term 
			\[\sum_{e_j \geq e_i} (e_j-e_i) + \#\{j \ | \ j > i \ \textrm{and} \ e_j \geq e_i\} = 3 \]
			counts the vanishing orders strictly greater than $8$. In particular, the vanishing condition described in this example is \[ h^0(C, L-8p) \geq 1 + 2 +1 = 4. \]
		\end{minipage} 
		\ \\
		
		We start with the condition 
		\[ h^0(C,L(-b_{k-1}\cdot p)) \geq 1 \]
		which we want to realise as a degeneracy locus on $\mathrm{Pic}^{g+r}(C)$. 
		
		Let $m > 0$ big enough so that $L\otimes \mathcal{O}_C(m \cdot p)$ is non-special for every $L \in \mathrm{Pic}^{g+r}(C)$, and consider the short exact sequence 
		\[ 0 \rightarrow L\otimes \mathcal{O}_C(-b_{k-1}\cdot p) \rightarrow L\otimes  \mathcal{O}_C(m\cdot p) \rightarrow L\otimes  \mathcal{O}_C(m\cdot p)\vert_{(b_{k-1} +m)\cdot p}\rightarrow 0.\]
		Looking at the exact sequence obtained by taking global sections, the condition $h^0(C,L-b_{k-1}\cdot p) \geq 1$ is equivalent to the map 
		\[ H^0(L\otimes \mathcal{O}_C(m\cdot p)) \rightarrow H^0(C, L\otimes  \mathcal{O}_C(m\cdot p)\vert_{(b_{k-1} +m)\cdot p})\]
		having rank at most $m+r$. 
		
		Consider the product space $C\times \mathrm{Pic}^{g+r}(C)$, and its projection maps $\pi_1$ and $\pi_2$ to the first and second component. Let $\mathcal{L}$ be a Poincar\'e line bundle on $C\times \mathrm{Pic}^{g+r}(C)$, $D \coloneqq m\cdot\{p\}\times \mathrm{Pic}^{g+r}(C)$ and let 
		\[P_{i} \coloneqq (b_i+m)\cdot\{p\}\times \mathrm{Pic}^{g+r}(C)\] 
		for each $0\leq i \leq k-1$. 
		
		The locus in $\mathrm{Pic}^{g+r}(C)$ of line bundles satisfying $h^0(C,L-b_{k-1}\cdot p) \geq 1$ is the locus where the morphism of vector bundles 
		\[ \mathcal{E} \coloneqq \pi_{2*}\mathcal{L}(D) \rightarrow \pi_{2*}(\mathcal{L}(D)\otimes \mathcal{O}_{P_{k-1}}) \]
		has rank at most $m+r$. The two vector bundles have rank $r+1+m$ and $b_{k-1}+m$ respectively, and hence this degeneracy locus is either empty or has codimension at most $(b_{k-1}-r)$ in $\mathrm{Pic}^{g+r}(C)$. 
		
		We denote by $X_{k-1}$ the locus of such line bundles $L$ for which we have $h^0(C,L-b_{k-1}\cdot p) = 1$ and $h^0(C,L-(b_{k-1}+1)\cdot p) = 0$. Because this is an open condition, it follows that $X_{k-1}$ is an open subset of the degeneracy locus above, and hence (if non-empty) has codimension at most $(b_{k-1}-r)$ in $\mathrm{Pic}^{g+r}(C)$.
		
		Using Proposition~\ref{prop: ramification behaviour}, it follows that all elements of $\textbf{a}(\Vec{e})_{>b_{k-2}}$ are vanishing orders of the line bundles $L$ contained in $X_{k-1}$. Let $i_{k-2}$ be the index such that $b_{k-2} = ke_{i_{k-2}} + i_{k-2}$. We have that 
		
		\[ h^0(C, L-c_{k-2} \cdot p) \geq  \sum_{j \geq k-2} (e_{i_j}-e_{i_{k-2}}) + \#\{j \ | \ j > i_{k-2} \ \textrm{and} \ e_j \geq e_{i_{k-2}}\} \]
		and we denote by $Y_{k-1}$ the open sublocus of $X_{k-1}$ where we have equality. This is the sublocus of line bundles $L$ whose ramification orders greater or equal to $c_{k-2}$ are exactly those appearing in $\textbf{a}(\Vec{e})_{>b_{k-2}}$. 
		
		We look at the sublocus of $Y_{k-1}$ of line bundles respecting the condition 
		\[ h^0(C, L-b_{k-2} \cdot p) \geq  1 + \sum_{j \geq k-2} (e_{i_j}-e_{i_{k-2}}) + \#\{j \ | \ j > i_{k-2} \ \textrm{and} \ e_j \geq e_{i_{k-2}}\} \]
		and we want to estimate its codimension. 
		
		Let
		\[ Q_i \coloneqq  (c_i+m)\cdot\{p\}\times \mathrm{Pic}^{g+r}(C) \]
		for each $0\leq i \leq k-2$ and consider the morphisms of vector bundles 
		\[ \mathcal{E} \rightarrow \pi_{2*}(\mathcal{L}(D)\otimes \mathcal{O}_{Q_{k-2}}) \rightarrow  \pi_{2*}(\mathcal{L}(D)\otimes \mathcal{O}_{P_{k-2}}) \]
		which we restrict to $Y_{k-1}$. The sublocus of $Y_{k-1}$ satisfying the required condition is the locus of line bundles $L$ where the composition of the morphisms above has rank at most 
		\[r+m-\sum_{j \geq k-2} (e_{i_j}-e_{i_{k-2}})- \#\{j \ | \ j > i_{k-2} \ \textrm{and} \ e_j \geq e_{i_{k-2}}\}.\]
		
		Because we have 
		\[ h^0(C, L-c_{k-2} \cdot p) =  \sum_{j \geq k-2} (e_{i_j}-e_{i_{k-2}}) + \#\{j \ | \ j > i_{k-2} \ \textrm{and} \ e_j \geq e_{i_{k-2}}\} \]
		on $Y_{k-1}$, it follows that the morphism
		\[ \varphi_{k-2}\colon \mathcal{E} \rightarrow \pi_{2*}(\mathcal{L}(D)\otimes \mathcal{O}_{Q_{k-2}}) \]
		has rank $r+1+m -\sum_{j \geq k-2} (e_{i_j}-e_{i_{k-2}}) - \#\{j \ | \ j > i_{k-2} \ \textrm{and} \ e_j \geq e_{i_{k-2}}\}$ at each point $L$ of $Y_{k-1}$. In particular, this implies that $\mathrm{Im}(\varphi_{k-2})$ is a vector bundle on $Y_{k-1}$ of rank 
		
		\[r+1+m -\sum_{j \geq k-2} (e_{i_j}-e_{i_{k-2}}) - \#\{j \ | \ j > i_{k-2} \ \textrm{and} \ e_j \geq e_{i_{k-2}}\}.\]
		
		We can view our required sublocus of $Y_{k-1}$ as the locus where the morphism 
		\[ \mathrm{Im}(\varphi_{k-2}) \rightarrow \pi_{2*}(\mathcal{L}(D)\otimes \mathcal{O}_{P_{k-2}}) \]
		has rank at most 
		\[r+m -\sum_{j \geq k-2} (e_{i_j}-e_{i_{k-2}}) - \#\{j \ | \ j > i_{k-2} \ \textrm{and} \ e_j \geq e_{i_{k-2}}\}.\]
		
		The vector bundles $\mathrm{Im}(\varphi_{k-2})$ and $\pi_{2*}(\mathcal{L}(D)\otimes \mathcal{O}_{P_{k-2}})$ have ranks
		\[r+m +1 -\sum_{j \geq k-2} (e_{i_j}-e_{i_{k-2}}) - \#\{j \ | \ j > i_{k-2} \ \textrm{and} \ e_j \geq e_{i_{k-2}}\}\]
		and $b_{k-2} + m$, respectively. Hence this degeneracy locus is either empty or has codimension in $Y_{k-1}$ at most 
		\[ b_{k-2} -r + \sum_{j \geq k-2} (e_{i_j}-e_{i_{k-2}}) + \#\{j \ | \ j > i_{k-2} \ \textrm{and} \ e_j \geq e_{i_{k-2}}\}. \]
		
		We denote by $X_{k-2}$ the open sublocus where we have the equalities
		\[ h^0(L-b_{k-2} \cdot p) =  1 + \sum_{j \geq k-2} (e_{i_j}-e_{i_{k-2}}) + \#\{j \ | \ j > i_{k-2} \ \textrm{and} \ e_j \geq e_{i_{k-2}}\} \]
		and 
		\[ h^0(L-(b_{k-2}+1) \cdot p) =  \sum_{j \geq k-2} (e_{i_j}-e_{i_{k-2}}) + \#\{j \ | \ j > i_{k-2} \ \textrm{and} \ e_j \geq e_{i_{k-2}}\} \]
		
		We now define $Y_{k-2}$ to be the open sublocus of $X_{k-2}$ parametrizing line bundles whose ramification order greater or equal to $c_{k-3}$ are exactly those in $\textbf{a}(\Vec{e})_{>b_{k-3}}$.
		
		Inductively, we obtain a series of inclusions 
		\[ Y_{0} \subseteq Y_{1} \subseteq \cdots \subseteq Y_{k-1} \subseteq \mathrm{Pic}^{g+r}(C),\]
		with $Y_t$ the locus of line bundles whose ramification orders at $p$ greater or equal to $c_{t-1}$ are exactly those in $\textbf{a}(\Vec{e})_{>b_{t-1}}$. The space $Y_0$ is the sublocus of $\mathrm{Pic}^{g+r}(C)$ consisting of $g^r_{g+r}$ having ramification orders exactly $\mathbf{a}(\Vec{e})$ at $p$, that is $Y_0 = W^{\Vec{e}}(\pi,p)$.
		
		The codimension estimate that we get using our inductive argument is 
		\[ \mathrm{codim}(Y_{t-1}; Y_t) \leq b_{t-1}-r+\sum_{j \geq t-1} (e_{i_j}-e_{i_{t-1}}) +  \#\{j \ | \ j > i_{t-1} \ \textrm{and} \ e_j \geq e_{i_{t-1}}\}.\]
		Adding up all the codimension we get that 
		\begin{align*}
			\mathrm{codim}(Y_0; \mathrm{Pic}^{g+r}(C)) &\leq \sum_{i=0}^{k-1}b_i - kr + \sum_{e_i\geq e_j}(e_i-e_j) + \#\{(i,j) \ | \ j > i \ \textrm{and} \ e_j \geq e_{i}\} \\
			&\leq k \sum_{i=0}^{k-1}e_i +  \sum_{i=0}^{k-1} i - kr + \sum_{e_i\geq e_j}(e_i-e_j) + \#\{(i,j) \ | \ j > i \ \textrm{and} \ e_j \geq e_{i}\} \\
			&\leq k(r+1-k) + \sum_{i=0}^{k-1} i - kr + \sum_{e_i\geq e_j}(e_i-e_j) + \#\{(i,j) \ | \ j > i \ \textrm{and} \ e_j \geq e_{i}\} \\
			&\leq -\frac{k(k-1)}{2} + \sum_{e_i\geq e_j}(e_i-e_j) + \#\{(i,j) \ | \ j > i \ \textrm{and} \ e_j \geq e_{i}\}.
		\end{align*}
		
		But we immediately check the equalities 
		\[  \sum_{e_i\geq e_j}(e_i-e_j) -\frac{k(k-1)}{2} = u(\Vec{e}) - \#\{ (i,j) \ | \ i < j \ \textrm{and} \ e_i = e_j\} \]
		and 
		\[ \#\{(i,j) \ | \ j > i \ \textrm{and} \ e_j \geq e_{i}\} = \mathrm{inv}(\Vec{e}) +  \#\{ (i,j) \ | \ i < j \ \textrm{and} \ e_i = e_j\}. \]
		
		Putting the two equalities together we obtain that $Y_0$ is either empty or has dimension greater or equal to $g - u(\Vec{e}) - \mathrm{inv}(\Vec{e})$.
	\end{proof}
	
	Putting together Proposition~\ref{prop: upperbound} and Proposition~\ref{prop: lowerbound} we obtain a pointed version of the main result in~\cite{Larson2Vogt}.
	
	\begin{theorem} \label{trm: dimension}
		Let $[\pi\colon C \rightarrow \mathbb{P}^1,p]$ be a generic element of $\cHTot$ and let $\Vec{e}\in\Z^k_{\geq0}$. 
		Then the pointed Hurwitz--Brill--Noether locus $W^{\Vec{e}}(\pi,p)$, if non-empty, has dimension 
		\[ \dim W^{\Vec{e}}(\pi,p) = g- u(\Vec{e})- \mathrm{inv}(\Vec{e})\]
		%Moreover, for a generic such $g^r_d$, the vanishing orders at $p$ are exactly given by $\textbf{a}(\Vec{e})$. 
	\end{theorem}
	
	In fact, whenever $g - u(\Vec{e})- \inv(\Vec{e}) \geq 0$, the pointed Hurwitz--Brill--Noether locus is non-empty. This will be a direct consequence of the statement in genus $1$, together with the regeneration argument in Theorem~\ref{theorem: HBN regeneration}.
	
	\begin{remark}
		Let $[\pi,p]$ be a general element of $\cHTot$. We consider a maximal Hurwitz--Brill--Noether locus as in~\cite{Larson2021} whose associated vector has the form 
		\[ \Vec{e} = (\underbrace{a+1, \ldots, a+1}_{k_1 \ \textrm{times}}, \underbrace{a,\ldots a}_{k_2 \ \textrm{times}}, \underbrace{1, \ldots, 1}_{k_3 \ \textrm{times}}, \underbrace{0,\ldots, 0}_{k_4 \ \textrm{times}}), \]
		where $a\geq 2$, the integers $k_i$ are nonnegative, and $k_1+k_2+k_3+k_4=k$. Then the second proof of Proposition~\ref{prop: lowerbound} realizes the maximal Hurwitz--Brill--Noether locus as a pointed Brill--Noether locus of expected dimension. 
	\end{remark}
	
	\begin{proof}
		Consider the rank degeneracy conditions 
		\begin{gather*}
			h^0(C, L - (ak + k +i)p) \geq k_1 - i \ \textrm{for all} \ 0 \leq i \leq k_1-1 \ \textrm{and}\\
			h^0(C, L - (ak + k_1 +i)p) \geq k_1 + k_2 - i \ \textrm{for all} \ 0 \leq i \leq k_2-1.
		\end{gather*}
		According to the second proof of Proposition~\ref{prop: lowerbound}, these are all the non-trivial rank conditions we need to cut out $W^{\Vec{e}}(\pi,p)$. The expected dimension is the same, since the rank conditions are the same. 
		
		Hence, if we denote 
		$$\textbf{a} \coloneqq (0,1,\ldots, r-k_1-k_2, ak+k_1, \ldots, ak +k_1 + k_2-1, ak+k, \ldots, ak+k+k_1-1),$$
		we obtain that $W^{\Vec{e}}(\pi,p)$ is a component of expected dimension of $W^r_{g+r, \textbf{a}}(C,p)$. 
		
		\begin{figure}[H]
			\begin{tikzpicture}[x=0.72cm,y=-0.72cm,
				every node/.style={inner sep=0pt},
				order/.style={font=\normalsize}]
				
				% Common grids; shading marks the three nonzero southernmost displacements.
				\foreach \offset in {0,8} {
					\begin{scope}[shift={(\offset,0)}]
						\fill[black!12] (1,4) rectangle (3,5);
						\fill[black!12] (0,5) rectangle (1,6);
						\foreach \x in {0,...,5}
						\draw[gray!65,thin] (\x,0) -- (\x,6);
						\foreach \y in {0,...,6}
						\draw[gray!65,thin] (0,\y) -- (5,\y);
					\end{scope}
				}
				
				% Hurwitz--Brill--Noether vanishing orders.
				\begin{scope}
					\node at (2.5,-1.05) {Hurwitz--Brill--Noether};
					\node at (2.5,-0.45) {$\Vec{e}=(5,4,4,0,0)$};
					\foreach \n in {0,1,2,3,4,5,6,7,10,11,12,15,16,17,20,21,22,25} {
						\pgfmathtruncatemacro{\col}{mod(\n,5)}
						\pgfmathtruncatemacro{\row}{floor(\n/5)}
						\node[order] at ({\col+0.5},{\row+0.5}) {$\n$};
					}
				\end{scope}
				
				% Pointed Brill--Noether: only a_15=21, a_16=22, a_17=25 are displaced.
				% The gaps below must remain empty; this is not a column-prefix board.
				\begin{scope}[shift={(8,0)}]
					\node at (2.5,-1.05) {Pointed Brill--Noether};
					\node at (2.5,-0.45) {$\mathbf{a}=(0,\ldots,14,21,22,25)$};
					\foreach \n in {0,1,2,3,4,5,6,7,8,9,10,11,12,13,14,21,22,25} {
						\pgfmathtruncatemacro{\col}{mod(\n,5)}
						\pgfmathtruncatemacro{\row}{floor(\n/5)}
						\node[order] at ({\col+0.5},{\row+0.5}) {$\n$};
					}
				\end{scope}
			\end{tikzpicture} \caption{Expected dimension: $g-20$}
		\end{figure}
		In both cases, the expected codimension is the same. In one case, it is computed as the sum of displacements of southernmost squares, while in the second case it is computed as the sum of all displacement. The terms are the same (in our example: $6 + 6 + 8$), and the expected dimension of the two loci is the same. 
	\end{proof}
	Consider the vector 
	\[
	\Vec{f}
	\coloneqq
	(e_0,\ldots,e_{i-1},e_{i+1},e_i,\ldots,e_{k-1})
	\]
	so that $t_{i+1}(\Vec{f}) = \Vec{e}$. The argument in the first proof of Proposition~\ref{prop: lowerbound} can be adapted to show that the irreducible components of
	\[ W^{\Vec{e}}(\pi,p) \cup W^{\Vec{f}}(\pi,p)\]
	have dimension at least $g-u(\Vec{e})- \inv(\Vec{e})$. However, Theorem~\ref{trm: dimension} implies that 
	\[  \dim W^{\Vec{f}}(\pi,p) = g - 1 - u(\Vec{e})- \inv(\Vec{e}), \]
	and hence $ W^{\Vec{f}}(\pi,p) \subseteq  \overline{W}^{\Vec{e}}(\pi,p)$.
	
	\begin{prop} \label{prop: containments}
		Let $\elhur$ be a generic element of $\cHTot$. We take $0\leq i \leq k-2$ be such that $e_i>e_{i+1}$, and consider
		\[
		\Vec{f}
		\coloneqq
		(e_0,\ldots,e_{i-1},e_{i+1},e_i,\ldots,e_{k-1}).
		\]
		Then we have the inclusion
		\[
		W^{\Vec{f}}(\pi,p)
		\subseteq
		\overline{W}^{\Vec{e}}(\pi,p).
		\]
		
		The same holds if $e_{k-1}+1>e_0$ and 
		\[\vec{f}=(e_{k-1}+1,e_1,e_2\dots,e_{k-2},e_0-1).\]
	\end{prop}
	
	\begin{proof}
		There is nothing to prove if $W^{\Vec{f}}(\pi,p)$ is empty. Otherwise,
		Theorem~\ref{trm: dimension} gives
		\[
		\dim W^{\Vec{f}}(\pi,p)
		=
		g-u(\Vec{e})-\inv(\Vec{e})-1.
		\]
		It therefore suffices to show that every irreducible component of
		$W^{\Vec{e}}(\pi,p)\cup W^{\Vec{f}}(\pi,p)$ has dimension at least
		$g-u(\Vec{e})-\inv(\Vec{e})$. It will then follow that $W^{\Vec{f}}(\pi,p)$ is not a component, but contained in the closure of $W^{\Vec{e}}(\pi,p)$.
		
		We use the construction in the first proof of
		Proposition~\ref{prop: lowerbound}. We will realize $W^{\Vec{e}}(\pi,p)\cup W^{\Vec{f}}(\pi,p)$ as an intersection of flag degeneracy loci in the open locus
		of $\mathrm{ord}(\Vec{e})$-nested linear series with exact splitting
		type $\mathrm{ord}(\Vec{e})$.
		
		Let $(j_1,i_1)$ and $(j_2,i_2)$ be the pairs satisfying
		\[
		m_{j_1,i_1}=i
		\qquad\text{and}\qquad
		m_{j_2,i_2}=i+1,
		\]
		in particular  $d_{j_1}=e_i>e_{i+1}=d_{j_2}$ and $j_1<j_2$. 
		
		Recall that in the proof of Proposition~\ref{prop: lowerbound}, we imposed the rank conditions
		\[ \operatorname{rk}\left(V_j/\textrm{Im}_j(V_{j-1}) \rightarrow H^0(C, L(-d_jk\cdot p)\vert_{m_{j,i}\cdot p}) / \varphi_{j,i}(\textrm{Im}_j(V_{j-1}))\right) \leq i-1\]
		for $1\leq j\leq s$ and $1\leq i\leq k_j$, organized in such a way that for each $j$ the locus respecting the conditions is a flag degeneracy locus. After imposing each such
		condition, we restricted to the open subset on which the rank is maximal, and the vanishing orders are exactly the $m_{j,i}$, i.e.
		\[
		h^0\left(C,L\bigl(-(d_jk+m_{j,i})\cdot p\bigr)\right)
		=
		1+
		h^0\left(C,L\bigl(-(d_jk+m_{j,i}+1)\cdot p\bigr)\right).
		\]
		
		The imposed vanishing conditions ensure that $  \varphi_{j,i}\left(\mathrm{Im}_j(\mathcal V_{j-1})\right)$ does not jump in dimension and hence
		\[
		\mathcal F_{j,i}/
		\varphi_{j,i}\left(\mathrm{Im}_j(\mathcal V_{j-1})\right)
		\]
		is a vector bundle of rank
		\[
		m_{j,i}
		-
		\#\left\{
		m\in M_1\cup\cdots\cup M_{j-1}
		\ \middle|\
		m<m_{j,i}
		\right\}.
		\]
		Finally, the flag degeneracy conditions, together with the vanishing order conditions cut out exactly the locus $W^{\Vec{e}}(\pi, p)$. 
		
		For our flag degeneracy conditions to cut out  $W^{\Vec{e}}(\pi,p)\cup W^{\Vec{f}}(\pi,p)$, we make the following two modifications.
		
		\begin{enumerate}
			\item The flag degeneracy condition on layer $j_2$ is the same as in the proof of Proposition~\ref{prop: lowerbound}, but we relax the vanishing order condition: we restrict to the open locus on which either
			\[
			d_{j_1}k+m_{j_1,i_1}
			\qquad\text{or}\qquad
			d_{j_1}k+m_{j_1,i_1}+1
			\]
			is a vanishing order; or equivalently
			\[
			h^0\left(
			C,L\bigl(-(d_{j_1}k+m_{j_1,i_1})\cdot p\bigr)
			\right)
			=
			1+
			h^0\left(
			C,L\bigl(-(d_{j_1}k+m_{j_1,i_1}+2)\cdot p\bigr)
			\right).
			\]
			\textbf{Remark:} In the Proof of Proposition~\ref{prop: lowerbound}, we imposed the condition 
			\[
			h^0\left(
			C,L\bigl(-(d_{j_1}k+m_{j_1,i_1})\cdot p\bigr)
			\right)
			=
			1+
			h^0\left(
			C,L\bigl(-(d_{j_1}k+m_{j_1,i_1}+1)\cdot p\bigr)
			\right),
			\]
			so that the vanishing order was exactly $d_{j_1}k+m_{j_1,i_1}$.  
			\item Instead of the condition corresponding to $(j_2,i_2)$ in the proof of Proposition~\ref{prop: lowerbound}
			\[
			\operatorname{rk}\!\left[
			\left(
			V_{j_2}
			\big/
			\mathrm{Im}_{j_2}(V_{j_2-1})
			\right)
			\longrightarrow H^0\left(
			C,
			L(-d_{j_2}k\cdot p)
			\big|_{m_{j_2,i_2}\cdot p}
			\right)
			\big/
			\varphi_{j_2,i_2}\!\left(
			\mathrm{Im}_{j_2}(V_{j_2-1})
			\right)
			\right]
			\leq i_2-1.
			\]
			we impose the condition
			\[
			\operatorname{rk}\!\left[
			\left(
			V_{j_2}
			\big/
			\mathrm{Im}_{j_2}(V_{j_2-1})
			\right)
			\longrightarrow H^0\left(
			C,
			L(-d_{j_2}k\cdot p)
			\big|_{m_{j_1,i_1}\cdot p}
			\right)
			\big/
			\varphi_{j_2,i_2}\!\left(
			\mathrm{Im}_{j_2}(V_{j_2-1})
			\right)\right] \leq i_2-1.
			\]
			We then restrict again to the locus of
			$\operatorname{ord}(\Vec{e})$-nested linear series for which this map
			has maximal rank and both
			\[
			d_{j_2}k+m_{j_1,i_1}
			\qquad\text{and}\qquad
			d_{j_2}k+m_{j_1,i_1}+1
			\]
			are vanishing orders.
		\end{enumerate}
		We look at the vector bundle whose fiber over each $(L,V_1,\ldots, V_s)$ is
		\[ H^0\left(C, L(-d_{j_2}k\cdot p) \big|_{m_{j_1,i_1}\cdot p} \right) \big/ \varphi_{j_2,i_2}\!\left(\mathrm{Im}_{j_2}(V_{j_2-1}) \right).\]
		This has rank $	i-\#\left\{m\in M_1\cup M_2\cup\cdots\cup M_{j_2-1} \ \middle|\ m<i\right\}$, independent on whether $d_{j_1}k + m_{j_1,i_1}$ or $d_{j_1}k + m_{j_1,i_1} +1$ is a vanishing order of $L$. This vector bundle replaces $$\mathcal{F}_{j_2,i_2}/\varphi_{j_2,i_2}(\textrm{Im}_{j_2}(\mathcal{V}_{j_2-1}))$$ 
		from the first proof of Proposition~\ref{prop: lowerbound}, and we know that this has rank $$	i+1 -\#\left\{m\in M_1\cup M_2\cup\cdots\cup M_{j_2-1} \ \middle|\ m<i+1\right\}.$$ 
		Since $i \in M_{j_1}$, these two ranks are the same and hence the modified target has the same rank as the original target in the proof of Proposition~\ref{prop: lowerbound}. 
		
		All flag codimension bounds are therefore unchanged. As in the first proof of Proposition~\ref{prop: lowerbound}, successive flag degeneracy conditions followed by open restrictions show that every irreducible component of the resulting union has dimension at least $g-u(\Vec{e})-\inv(\Vec{e})$. Furthermore, there conditions cut out exactly $W^{\Vec{e}}(\pi,p)\cup W^{\Vec{f}}(\pi,p)$: 
		\begin{itemize}
			\item Except for the two vanishing and rank conditions outlined above, all the other condition stay the same. In particular, all vanishing orders in $\textbf{a}(\Vec{e})$ that are not $i$ or $i+1$ modulo $k$ are vanishing orders of the locus cut out by the successive flag degeneracy conditions. 
			\item The rest of the vanishing orders depend on whether $d_{j_1}k + m_{j_1,i_1}$ or $d_{j_1}k + m_{j_1,i_1} +1$ is a vanishing order of $L$. Depending on these cases, we obtain that the flag degeneracy conditions cut out the loci $W^{\Vec{e}}(\pi,p)$ and $W^{\Vec{f}}(\pi,p)$.
		\end{itemize} 
		Since $W^{\Vec{f}}(\pi,p)$ has dimension $g-1-u(\Vec{e})-\inv(\Vec{e})$, it must be contained in the closure of $W^{\Vec{e}}(\pi,p)$.

		To get the statement when $e_{k-1}+1>e_0$ and \[\vec{f}=(e_{k-1}+1,e_1,\dots,e_{k-2},e_0-1),\] we first note that since \[W^{\vec{e}}(\pi,p)=\{L \mid \pi_{\ast}(L(-ip))=\OO(\sigma^i(\vec{e})) \text{ for $0\leq i \leq k-1$ (eq. for all $i$)}\},\] we have \[W^{\vec{e}}(\pi,p)\otimes\OO(-ip)=W^{\sigma^i(\vec{e})}(\pi,p) \text{ and } \overline{W}^{\vec{e}}(\pi,p)\otimes\OO(-ip)=\overline{W}^{\sigma^i(\vec{e})}(\pi,p),\] as twisting by $\OO(-ip)$ is an isomorphism $\Pic^d\to\Pic^{d-i}$ and commutes with closures. Twisting $L\in W^{\vec{e}}(\pi,p)$ by $\OO(-p)$, the first part of the Proposition now gives \[W^{(e_2\dots,e_{k-2},e_{0}-1,e_{k-1})}\subseteq \overline{W}^{(e_2,\dots,e_{k-2},e_{k-1},e_{0}-1)}\] and twisting up by $\OO(p)$ gives the result.
	\end{proof}

	\subsection{Hurwitz--Brill--Noether loci}
	In fact, Theorem~\ref{trm: dimension} generalizes the dimension result in~\cite{Larson2Vogt}. As a consequence of Theorem~\ref{trm: dimension} and Proposition~\ref{prop: containments} we obtain the following.
	
	\begin{cor} \label{cor: LLV-dim}
		Let $[\pi\colon C \rightarrow \mathbb{P}^1,p]$ be a generic element of $\cHTot$ and let $\Vec{e}\in\Z^k_{\geq0}$. Then the spaces
		\[ W^{\mathrm{ord}(\Vec{e})}(\pi,p) \ \textrm{and} \ \{ L \in \mathrm{Pic}(C) \ \mid \ \pi_*L \cong \OO(\Vec{e})\}, \]
		are birationally equivalent of dimension $g-u(\Vec{e})$.
	\end{cor}
	\begin{proof}
		Using Proposition~\ref{prop: containments} we obtain 
		\[ W^{\Vec{e}}(\pi,p) \subseteq \overline{W}^{\mathrm{ord}(\Vec{e})}(\pi,p)  \]
		for every vector $\Vec{e}$ with non-negative entries. Moreover, we have the set-theoretical equality 
		\[ \{ L \in \mathrm{Pic}(C) \ \mid \ \pi_*L \cong \OO(\Vec{e})\} = \bigsqcup_{\Vec{f} \ \textrm{a reordering of} \ \Vec{e}} W^{\Vec{f}}(\pi,p).\]
		In particular, the spaces
		\[ \{ L \in \mathrm{Pic}(C) \ \mid \ \pi_*L \cong \OO(\Vec{e})\} \ \mathrm{and} \ W^{\mathrm{ord}(\Vec{e})}(\pi,p) \]
		have the same closure in $\textrm{Pic}(C)$ and the conclusion follows. 
	\end{proof}
	
	\begin{remark}
		Notice that our Corollary~\ref{cor: LLV-dim} proves the Hurwitz--Brill--Noether dimension result for a generic map $[\pi\colon C\rightarrow \mathbb{P}^1]$ having a total ramification point $p$, and not for a generic $k\colon 1$ map as in~\cite{Larson2Vogt}. To solve this issue, we can work with a Hurwitz moduli space parametrizing tuples 
		\[ [\pi\colon C\rightarrow \mathbb{P}^1, p_1, p_2,\ldots, p_k] \] where $p_1,\dots,p_k$ are distinct points such that $\pi(p_1) = \pi(p_2) = \cdots = \pi(p_k)$.
		Instead of considering vanishing orders with respect to a point $p$, we consider vanishing orders with respect to the chain of effective divisors
		\[ p_1 < p_1 + p_2 < \cdots <p_1 + p_2 +\cdots +p_k < 2p_1 + p_2 +p_3 + \cdots + p_k < \cdots \]
		as defined in~\cite{Osserman1} and~\cite{Osserman2}. The entire theory works \textit{mutatis mutandis} and we obtain the desired results.
	\end{remark}
	
	\begin{cor} \label{cor: LLV-dim2}
		Let $[\pi\colon C \rightarrow \mathbb{P}^1,p_1,p_2,\cdots, p_k]$ be a generic $k\colon 1$ map together with $k$ distinct points in the same generic fiber, and let $\Vec{e} \in\Z^k_{\geq0}$. Then the locus
		\[ \ W^{\mathrm{ord}(\Vec{e})}(\pi,p_1, p_2,\ldots, p_k) \]
		has dimension $g- u(\Vec{e})$.  
	\end{cor}
	
	Consequently, we recover the dimension result of~\cite{Larson2Vogt}.
	
	It is a very natural question to ask, what loci are contained in the closure of $W^{\Vec{e}}(\pi,p)$. We will prove that $W^{\Vec{f}}(\pi,p) \subseteq  \overline{W}^{\Vec{e}}(\pi,p)$ if and only if $\textbf{a}(\Vec{e}) \leq \textbf{a}(\Vec{f})$ entry-wise. We first prove the following intermediary result. 
	
	\begin{prop} \label{prop: closure intersection} Let $[\pi\colon C\rightarrow \mathbb{P}^1]$ a generic element of $\cHTot$ and let $\Vec{e}$ be a vector with $k$ nonnegative entries. Then we have the intersection 
		\[\bigcap_{i=0}^{k-1} \overline{\{L \in \textrm{Pic}(C) \ \mid \ \pi_*(L-(i\cdot p))\cong \OO(\sigma^i\Vec{e}) \} } = \overline{W}^{\Vec{e}}(\pi,p) \cup \Delta, \]
		where $\Delta$ consists of components of dimension strictly smaller than $g-u(\Vec{e})-\textrm{inv}(\Vec{e})$.  
	\end{prop}
	
	\begin{proof}
		We know from Proposition~\ref{prop: intersection description} that
		\[ W^{\Vec{e}}(\pi,p) = \bigcap_{i=0}^{k-1} \{L \in \textrm{Pic}(C) \ \mid \ \pi_*(L-(i\cdot p))\cong \OO(\sigma^i\Vec{e}) \}, \]
		and hence $\overline{W}^{\Vec{e}}(\pi,p)$ is contained in the intersection of the closures. Let now $\Delta'$ be another component in the intersection, and let $\textbf{a}(\Vec{f})$ be the vanishing orders at $p$ of a generic element $L$ in $\Delta'$. 
		
		We have that $\textbf{a}(\Vec{e}) \leq \textbf{a}(\Vec{f})$ entry-wise, and since $\Delta' \neq \overline{W}^{\Vec{e}}(\pi,p)$, this inequality is strict for at least one position. Therefore, using Proposition~\ref{prop: vanishing and loci comparison} and Theorem~\ref{trm: dimension}, we obtain 
		\[ \dim(\Delta') \leq \dim(W^{\Vec{f}}(\pi,p)) = g-u(\Vec{f})-\textrm{inv}(\Vec{f}) < g-u(\Vec{e}) - \textrm{inv}(\Vec{e}).\qedhere  \]
	\end{proof}
	
	In fact, we will see in Proposition~\ref{prop: reduced-intersection-right-dim} that $\Delta = \emptyset$. 
	
	\section{Scheme-theoretic regeneration} \label{sec: scheme-theoretic regeneration}
	
	The goal of this section is to prove set-theoretic, and scheme-theoretic regeneration, in the spirit of \cite[Theorem 3.4]{basiclimitlinear} and \cite[Section 6 and Section 7]{Larson2Vogt}. This will require several new definitions along with a description of our set-up: We start by looking at the limit linear series in Lemma~\ref{lemma: ramifications} and describe what $\textrm{ord}(\Vec{e})$-nested limit linear series are compatible with the data, and then proceed to construct $\textrm{ord}(\Vec{e})$-nested linear series in families. 
	
	\subsection{Nested limit linear series}
	To describe $\textrm{ord}(\Vec{e})$-nested limit linear series, we need to describe them componentwise. As we want to treat the case when the curve degenerates to Setup~\ref{setup: elliptic tail}, our goal is to understand $\textrm{ord}(\Vec{e})$-nested linear series appearing as limits in the context of Lemma~\ref{lemma: ramifications}.
	
	\begin{lemma} \label{lemma: nested description}
		Let $\pi\colon X \rightarrow \mathbb{P}^1$ be a $k:1$ cover, totally ramified at $q \in X$. Let $\Vec{e}=(e_0,\ldots,e_{k-1})$ be a length $k$ vector with non-negative entries, satisfying $e_0 > e_{k-1}+1$. Let $L$ be a line bundle in
		\[
		W^{t_0(\Vec{e})}(\pi,q)\otimes \OO_X(q).
		\]
		Then $L$ admits a unique $\mathrm{ord}(\Vec{e})$-nested linear series whose full space of sections $V$ satisfies
		\[
		\dim V(-2q)=\dim V-1.
		\]
	\end{lemma}
	
	\begin{proof}
		We consider the vector \[
		\Vec{f}
		=
		(e_0,e_{k-1}+1,e_1,\ldots,e_{k-2}),
		\] 
		and we immediately see
		\[
		W^{t_0(\Vec{e})}(\pi,q)\otimes \OO_X(q)
		=
		W^{\Vec{f}}(\pi,q).
		\]
		
		To determine the vanishing orders of $V$ at $q$ and we first make a few observations.
		
		All vanishing orders of $\textbf{a}(\Vec{e})$ at least $2$ are vanishing orders of $V$: We have
		\[ h^0(X, L) = \sum_{i=0}^{k-1}(f_i+1) = \sum_{i=0}^{k-1}(e_i+1)+1, \]
		hence $\dim V=h^0(C, L)-1$.
		As
		\[
		\dim V(-2q)=\dim V-1,
		\]
		it follows that either $0$ or $1$ is not a vanishing order, and $V(-2q) = H^0(C, L-2q)$, hence all elements of $\textbf{a}(\Vec{f})_{\geq 2}$ are vanishing orders of $V$. 
		
		In particular, the integer $e_0k$ is a vanishing order of $V$. Using this, we obtain that $0$ is a vanishing order of $V$ at $q$: We consider the well-defined map       
		\[
		V(-e_0k\cdot q)\otimes
		H^0\bigl(\mathbb{P}^1,\OO_{\mathbb{P}^1}(e_0)\bigr)
		\rightarrow V,
		\]
		and we will construct a section in the image with vanishing order $0$ at $q$. We choose a section $\sigma\in V$ such that $\operatorname{ord}_q(\sigma)=e_0k$, and view $\sigma$ as a section of
		$L(-e_0k\cdot q)$. Moreover, we choose $s\in H^0\left(\PP^1,\OO_{\PP^1}(1)\right)$, having a pole at the image of $q$ in $\PP^1$. Then  $\sigma\cdot s^{ e_0}$ is an element in the image satisfying $\operatorname{ord}_q\left(\sigma\cdot s^{e_0}\right)=0$.
		
		It follows that $0$ is a vanishing order of $V$ and hence
		\[
		V=
		\left\langle
		H^0\left(C,L(-2q)\right),
		\sigma\cdot s^{ e_0}
		\right\rangle.
		\]
		
		It remains to show that this space is uniquely determined. Indeed, suppose that $\sigma'$ is
		another section satisfying
		\[
		\operatorname{ord}_q(\sigma')=e_0k.
		\]
		Since $k e_0+1$ is not a vanishing order of $V$, there exists $t\in\mathbb C$ such that
		\[
		\operatorname{ord}_q(\sigma-t\sigma')
		\geq e_0k+2.
		\]
		Therefore,
		\[
		(\sigma-t\sigma')\cdot s^{ e_0}
		\in H^0\left(X,L(-2q)\right),
		\]
		so replacing $\sigma$ by $\sigma'$ does not change $V$. 
		
		Finally, the nesting is then given by the conditions 
		\[ V_i = V(-d_ik \cdot q) \coloneqq H^0(X, L-d_ik\cdot q) \cap V \]
		for all $1\leq i \leq s$ satisfying $d_i > e_{k-1}+1$ and 
		\[ V_i = \langle V(-(d_ik+2)\cdot q), \ \sigma \cdot s^{e_0 - d_i}\rangle\]
		for $d_i \leq e_{k-1}+1$.
	\end{proof}
	
	This result extends for all other loci appearing in Lemma~\ref{lemma: ramifications}. In these cases, describing the nested linear series is straight-forward, as the condition $\dim V(-q) = \dim V$ immediately implies $V = H^0(X, L-q)$. 
	
	\begin{lemma}\label{lem: nested description}
		Let $\pi\colon X\to\PP^1$ be a degree $k$ cover totally
		ramified at $q$, and let $\Vec e\in\mathbb Z_{\geq 0}^k$.
		Suppose that
		\[
		L\in W^{\Vec e}(\pi,q)\otimes\OO_X(q) \ \  \textrm{or} \  \ L \in W^{t_i(\Vec e)}(\pi,q) \otimes\OO_X(q), \qquad 1\leq i\leq k-1.
		\]
		Then $L$ admits a unique $\mathrm{ord}(\Vec e)$-nested
		linear series whose space of sections satisfies $V(-q)=V$.
		It is given by
		\[
		V=H^0(X,L(-q)),
		\qquad
		V_j=H^0(X,L(-(d_jk+1)q)).
		\]
		
	\end{lemma}
	
	The results above will be used to describe the $\textrm{ord}(\Vec{e})$-nested limit linear series associated to the limit linear series in Lemma~\ref{lemma: ramifications}. We will prove a regeneration formula for the nested linear series in Lemma~\ref{lemma: nested description} and Lemma~\ref{lem: nested description}.
	
	\begin{defn} \label{def: nested description}
		Let $[\pi_0,p]$ as defined in Setup~\ref{setup: elliptic tail}. We denote by $G^{\mathrm{ord}(\Vec e),\mathrm{nc}}(\pi_0,p)$ the locus of refined $\mathrm{ord}(\Vec e)$-nested limit linear series that are non-colliding at $q$ on $X$ and at $p,q$ on $E$,
		and whose aspects
		\[
		l_E=(L_E,V_E),\qquad l_X=(L_X,V_X)
		\]
		satisfy $\mathbf a(V_E,p)=\mathbf a(\Vec e)$ and one of the
		following conditions:
		\begin{enumerate}
			\item $L_X(-q)\in W^{\Vec e}(\pi_X,q)$ and $ \mathbf a(V_X,q)=1+\mathbf a(\Vec e).$

			\item For some $1\leq j\leq k-1$ with
			$t_j(\Vec e)\neq\Vec e$,
			\[
			L_E\simeq\OO_E(jp+(d-j)q),
			\qquad
			L_X(-q)\in W^{t_j(\Vec e)}(\pi_X,q),
			\]
			and $\mathbf a(V_X,q)=1+\mathbf a(t_j(\Vec e))$.

			\item If $t_0(\Vec e)\neq\Vec e$,
			\[
			L_E\simeq\OO_E(dq),
			\qquad
			L_X(-q)\in W^{t_0(\Vec e)}(\pi_X,q),
			\]
			and
			\[
			\mathbf a(V_X,q)=\mathbf a(\Vec f)\setminus\{1\}, \ \ \textrm{where} \ \
			\Vec f=(e_0,e_{k-1}+1,e_1,\ldots,e_{k-2}).
			\]
		\end{enumerate}
		Here $1+\mathbf a$ denotes entrywise addition.
	\end{defn}
	The layers of every series in this locus are uniquely
	determined by its untwisted aspects. On $E$, they are 
	$$V_{E,i}=V_E(-d_i k p),$$ 
	since nesting gives containment and the prescribed vanishing
	sequence at $p$ gives equality of dimensions.
	In cases~\textup{(1)} and~\textup{(2)},
	Lemma~\ref{lem: nested description} gives
	\[
	V_X=H^0(X,L_X(-q)),
	\qquad
	V_{X,i}=H^0(X,L_X(-(d_i k+1)q)).
	\]
	In case~\textup{(3)}, they are the unique spaces constructed
	in Lemma~\ref{lemma: nested description}.
	
	\subsection{Set-theoretic regeneration}
	\begin{setup}\label{setup: smoothing}
		Fix the admissible cover
		\[
		\pi_0\colon X\cup_q E
		\longrightarrow P_X\cup_{\overline q}P_E, \qquad p  \in E
		\]
		described in Setup~\ref{setup: elliptic tail}. By the deformation theory of Hurwitz spaces and admissible
		covers, see~\cite{harris_mumford,AbramovichCV}, there exists a smoothing family
		\[
		\begin{tikzcd}[column sep=large, row sep=large]
			\mathcal C
			\arrow{rr}{\mathfrak \pi}
			\arrow{dr}
			&&
			\mathcal P
			\arrow{dl}
			\\
			&
			B
			\arrow[ul, bend left=18, "p"]
			&
		\end{tikzcd}
		\]
		over $B=\operatorname{Spec}\mathbb C[[t]]$, whose special
		fiber is $(\pi_0,p)$ and whose generic fiber is a degree $k$
		cover of smooth curves, totally ramified at $p$. We choose the smoothing so that $\mathcal C$ is regular
		and $\mathcal P\to B$ is obtained by base change from a
		family with regular total space.
	\end{setup}

	Let $\Vec{e}$ be a length $k$ vector with non-negative entries in decreasing order. In this setting, we can adapt the proof of~\cite[Theorem 6.5]{Larson2Vogt} to obtain a quasi-projective scheme $\mathcal{G}^{\Vec{e}, \textrm{nc}}(\pi,p)$ whose general fiber over $t \neq 0$ is the space of non-colliding $\Vec{e}$-nested linear series and whose special fiber is the space of $\Vec{e}$-nested limit linear series associated to the loci appearing in Definition~\ref{def: nested description}.
	
	\begin{theorem}[Regeneration] \label{theorem: nested regeneration} Let $\Vec{e}$ be a length $k$ non-negative vector and let $\mathcal{C} \rightarrow \mathcal{P} \rightarrow B$ as in Setup~\ref{setup: smoothing}. Then, there exists a quasi-projective $B$-scheme $\mathcal{G}^{\Vec{e},\operatorname{nc}}
		(\mathfrak{\pi},p)\rightarrow B$
		whose general fiber is contained in the space of non-colliding
		$\Vec{e}$-nested linear series with the prescribed
		ramification $\textbf{a}(\Vec{e})$ at $p$, and whose special fiber is $G^{\Vec{e},\operatorname{nc}}(\pi_0,p)$, as described in Definition~\ref{def: nested description}. Moreover, every
		irreducible component $Z$ of $\mathcal{G}^{\Vec{e},\operatorname{nc}}
		(\mathfrak{\pi},p)$ satisfies
		\[
		\dim Z
		\geq
		\dim \operatorname{Pic}(\mathcal{C}/B)-u(\Vec{e})-\textrm{inv}(\Vec{e}).
		\]
	\end{theorem}
	
	\begin{proof}
		The proof of~\cite[Theorem 6.5]{Larson2Vogt} applies almost verbatim to construct a quasi-projective scheme $\mathcal{G}^{\textrm{ord}(\Vec{e}),\operatorname{nc}}
		(\mathfrak{\pi},p)$ parametrizing refined $\textrm{ord}(\Vec{e})$-nested linear series, with ramification orders at the node of the central fiber as described in Definition~\ref{def: nested description}. Each component of this scheme has dimension at least $\operatorname{Pic}(\mathcal{C}/B)-u(\Vec{e})$.
		
		The quasi-projective scheme $\mathcal{G}^{\textrm{ord}(\Vec{e}),\operatorname{nc}}(\mathfrak{\pi},p)$ is obtained as a subscheme of a product of towers of Grassmannians, and it admits universal vector bundles for each of the vector spaces in the $\textrm{ord}(\Vec{e})$-nesting. As such, the same method as in the first proof of Proposition~\ref{prop: lowerbound} can be applied here, and we obtain a subspace  
		\[ \mathcal{G}^{\Vec{e},\operatorname{nc}}
		(\mathfrak{\pi},p) \subseteq \mathcal{G}^{\textrm{ord}(\Vec{e}),\operatorname{nc}}
		(\mathfrak{\pi},p)\]
		for which the ramification orders at $p$ are exactly given by $\textbf{a}(\Vec{e})$. In the central fiber, the vanishing conditions at the node, together with the vanishing conditions at point $p$, cut out exactly the locus $G^{\Vec{e},\operatorname{nc}}(\pi_0,p)$.  
		
		Moreover, the proof of Proposition~\ref{prop: lowerbound} implies that 
		\[ \codim(\mathcal{G}^{\Vec{e},\operatorname{nc}}
		(\mathfrak{\pi},p), \mathcal{G}^{\textrm{ord}(\Vec{e}),\operatorname{nc}}
		(\mathfrak{\pi},p)) \leq \inv(\Vec{e}).\]
		This implies that every irreducible component $Z$ of $\mathcal{G}^{\Vec{e},\operatorname{nc}}
		(\mathfrak{\pi},p)$ satisfies
		\[
		\dim Z
		\geq
		\dim \operatorname{Pic}(\mathcal{C}/B)-u(\Vec{e})-\textrm{inv}(\Vec{e}),
		\]
		which is what we wanted to prove.
	\end{proof}
	
	Let $d = g + |\Vec{e}|+ k -1$ and consider the forgetful map from $\mathcal{G}^{\Vec{e},\operatorname{nc}}(\mathfrak{\pi},p)$  to $\textrm{Pic}^d(\mathcal{C}/B)$. We then obtain a counterpart of Theorem~\ref{theorem: nested regeneration}.
	
	\begin{theorem} \label{theorem: HBN regeneration} Let $\Vec{e}$ be a length $k$ non-negative vector and let $\mathcal{C} \rightarrow \mathcal{P} \rightarrow B$ as in Setup~\ref{setup: smoothing}. Then, there exists a quasi-projective $B$-scheme $\mathcal{W}^{\Vec{e}}
		(\mathfrak{\pi},p)\rightarrow B$
		whose general fiber is the pointed Hurwitz--Brill--Noether locus with prescribed
		ramification exactly $\textbf{a}(\Vec{e})$ at $p$, and whose special fiber is the union of the components in $\textrm{Pic}^d(X) \times \textrm{Pic}^0(E)$: 
		\begin{itemize}
			\item 	\[\left(\mathcal{O}_X(q)\otimes \overline{W}^{\Vec{e}}(\pi_{|X},q)\right)\times \mathrm{Pic}^0(E), \ \textrm{and} \]
			\item \[ \left(\OO_X(q) \otimes \overline{W}^{t_i(\Vec{e})}(\pi_{|X},q) \right) \times \{ \mathcal{O}_E(i\cdot p -i\cdot q)\} \]
			for $0\leq i \leq k-1$ satisfying $t_i(\Vec{e}) \neq \Vec{e}$.
		\end{itemize}
		Moreover, every
		irreducible component $Z$ of $\mathcal{W}^{\Vec{e}}
		(\mathfrak{\pi},p)$ satisfies
		\[
		\dim Z
		\geq
		\dim \operatorname{Pic}(\mathcal{C}/B)-u(\Vec{e})-\textrm{inv}(\Vec{e}).
		\]
	\end{theorem}
	
	\begin{proof}
		Consider the forgetful map 
		\[ \mathcal{G}^{\Vec{e},\operatorname{nc}}
		(\mathfrak{\pi},p) \rightarrow \operatorname{Pic}^d(\mathcal{C}/B),\]
		whose image is an open subspace of $\mathcal{W}^{\Vec{e}}
		(\mathfrak{\pi},p)$. This forgetful map is injective. Indeed, on the smooth locus, a line bundle satisfying 
		\[ h^0(C, L) = \sum_{i=0}^{k-1}(e_i+1)\]
		and having vanishing orders exactly $\textbf{a}(\Vec{e})$ admits a unique $\textrm{ord}(\Vec{e})$-nested linear series, given by 
		\[ V_j = H^0(C, L - d_jk\cdot p ).  \]
		Above the boundary, the forgetful map is injective because of Lemma~\ref{lemma: nested description} and Lemma~\ref{lem: nested description}, and the conclusion follows.
	\end{proof}
	
	\subsection{Scheme theoretic regeneration, Reduceness, and Cohen--Macaulayness} We know from Proposition~\ref{prop: closure intersection} that we have the set-theoretic intersection 
	\[\bigcap_{i=0}^{k-1} \overline{\{L \in \textrm{Pic}(C) \ \mid \ \pi_*(L-(i\cdot p))\cong \OO(\sigma^i\Vec{e}) \} } = \overline{W}^{\Vec{e}}(\pi,p) \cup \Delta, \]
	where $\Delta$ is some scheme of dimension $\dim \Delta < g- u(\Vec{e})- \textrm{inv}(\Vec{e})$. Each set in the intersection in the LHS can be given a scheme structure using the determinantal conditions
	\[\OO_C(i\cdot p) \otimes \overline{W}^{\sigma^i\Vec{e}}(\pi,p) = \{L\in \textrm{Pic}^d(C) \ \mid \ h^0(C, L-(mk+i)p) \geq h^0(\mathbb{P}^1, \OO(\sigma^i\Vec{e})(-m)) \ \text{ for } m\geq 0\}.  \]
	As such, using all rank conditions 
	$$\{L\in \textrm{Pic}^d(C) \ \mid \ h^0(C, L-(mk+i)p) \geq h^0(\mathbb{P}^1, \OO(\sigma^i\Vec{e})(-m)) \ \text{ for } m\geq 0, \ \textrm{and} \ 0\leq i \leq k-1\},$$
	the space $\overline{W}^{\Vec{e}}(\pi,p) \cup \Delta$ inherits a scheme structure. We will prove inductively that this scheme is reduced and of pure dimension $g-u(\Vec e)-\inv(\Vec e)$. In view of Proposition~\ref{prop: closure intersection}, this implies that $\Delta=\emptyset$ and that $\overline W^{\Vec e}(\pi,p)$, endowed with the induced determinantal scheme structure, is reduced. 
	
	The same methods as in \cite[Section 7]{Larson2Vogt} can be applied for the elliptic tail degeneration. Let $C_0=X\cup_q E$, where $X$ and $E$ are smooth curves of genus $g-1$ and $1$, respectively. Fix a degree distribution
	\[
	\Vec{d}=(d_X,d_E),
	\qquad d_X+d_E=d,
	\]
	and consider the determinantal locus
	\[
	W_{\Vec{d}}^r(C_0)
	=
	\left\{
	L\in\Pic^{\Vec{d}}(C_0)
	\ \middle|\
	h^0(C_0,L)\geq r+1
	\right\}.
	\]
	Since $C_0$ is of compact type, we have
	\[
	\Pic^{\Vec{d}}(C_0)
	\simeq
	\Pic^{d_X}(X)\times\Pic^{d_E}(E),
	\]
	and our goal is to describe this rank degeneracy locus on this product. Let \[\pi\colon C_0\times\Pic^{\Vec{d}}(C_0)
	\longrightarrow
	\Pic^{\Vec{d}}(C_0)\] be the projection, and let $\mathcal{L}:=\mathcal{L}_{\Vec{d}}$ be a universal line bundle. Choose a degree $N$ divisor $D$ on $C_0$, contained in the smooth locus, such that
	\[
	D_X:=D\cap X
	\qquad\text{and}\qquad
	D_E:=D\cap E
	\]
	are sufficiently positive, with degrees denoted $N_X$ and $N_E$. In terms of the natural restriction morphism
	\[
	\phi\colon
	\pi_*\mathcal{L}(D)
	\longrightarrow
	\pi_*\bigl(\mathcal{L}(D)|_D\bigr),
	\]
	the locus of interest is
	\[
	W_{\Vec{d}}^r(C_0)
	=
	\left\{
	L\in\Pic^{\Vec{d}}(C_0)
	\ \middle|\
	\dim\ker(\phi|_L)\geq r+1
	\right\}.
	\]
	Its scheme structure is given by the $(n+1)\times(n+1)$ minors of
	$\phi$, where
	\[
	n
	=
	\rank\pi_*\mathcal{L}(D)-(r+1)
	=
	d-g-r+\deg(D).
	\]
	
	We describe the rank loci of $\phi$ in terms of evaluation maps on
	the normalization
	\[
	\nu\colon X\sqcup E\longrightarrow C_0.
	\]
	Let
	\[
	\pi_X\colon X\times\Pic^{d_X}(X)\longrightarrow\Pic^{d_X}(X),
	\qquad
	\pi_E\colon E\times\Pic^{d_E}(E)\longrightarrow\Pic^{d_E}(E)
	\]
	be the projections, and let $\mathcal{L}_X$ and $\mathcal{L}_E$ be
	universal line bundles on the corresponding products, compatible with the universal line bundle $\mathcal{L}$, i.e. for the maps 
	\[ \textrm{pr}_X \colon \Pic^{\Vec{d}}(C_0) \rightarrow \Pic^{d_X}(X) \ \textrm{and} \ \textrm{pr}_E \colon \Pic^{\Vec{d}}(C_0) \rightarrow \Pic^{d_E}(E) \]
	we have 
	\[ \mathcal{L}\vert_{X\times \Pic^{\Vec{d}}(C_0) } = (\textrm{Id}_X\times\textrm{pr}_X)^*\mathcal{L}_X \ \textrm{and} \ \mathcal{L}\vert_{E\times \Pic^{\Vec{d}}(C_0) } = (\textrm{Id}_E\times \textrm{pr}_E)^*\mathcal{L}_E.\]
	
	The normalization sequence gives a commutative diagram
	\begin{center}
		\small
		\begin{tikzcd}[column sep=small]
			&&
			\displaystyle
			(\pr_X)^*(\pi_X)_*\mathcal{L}_X(D_X) \oplus (\pr_E)^*(\pi_E)_*\mathcal{L}_E(D_E)
			
			\arrow[equal]{d}
			\\
			0
			\arrow{r}
			&
			\pi_*\mathcal{L}(D)
			\arrow{r}
			\arrow{d}[swap]{\phi}
			&
			\pi_*(\nu\times\Id)_*(\nu\times\Id)^*\mathcal{L}(D)
			\arrow{r}{\eta}
			\arrow{d}{\psi}
			&
			\pi_*\bigl(\mathcal{L}(D)|_q\bigr)
			\arrow{r}
			&
			0
			\\
			&
			\pi_*\bigl(\mathcal{L}(D)|_D\bigr)
			\arrow{r}{\sim}
			&
			\pi_*(\nu\times\Id)_*
			(\nu\times\Id)^*\mathcal{L}(D)|_{\nu^{-1}(D)}
			\arrow[equal]{d}
			\\
			&&
			\displaystyle
			(\pr_X)^*(\pi_X)_*
			\bigl(\mathcal{L}_X(D_X)|_{D_X}\bigr) \oplus (\pr_E)^*(\pi_E)_*
			\bigl(\mathcal{L}_E(D_E)|_{D_E}\bigr).
		\end{tikzcd}
	\end{center}
	
	Here $\eta(s_X,s_E)=s_X(q)-s_E(q)$,
	and $\psi$ is restriction to $D_X\sqcup D_E$. The top row is exact
	because $D_X$ and $D_E$ are sufficiently positive and thus $h^1(C_0, L(D)) = 0$ for any $L \in \Pic^{\Vec{d}}(C_0)$. As in \cite[Section 7.1]{Larson2Vogt}, since exact sequences of vector bundles split locally, the rank
	loci of $\phi$ are the corresponding rank loci of $\psi\oplus\eta$.
	More precisely,
	\[
	\left\{
	L\in\Pic^{\Vec{d}}(C_0)
	\ \middle|\
	\rank\phi|_L\leq n
	\right\}
	=
	\left\{
	L\in\Pic^{\Vec{d}}(C_0)
	\ \middle|\
	\rank(\psi\oplus\eta)|_L\leq n+1
	\right\}
	\]
	scheme-theoretically. Then we can apply \cite[Lemma 7.1]{Larson2Vogt} and obtain the following result.
	
	\begin{lemma}\label{lemma: nodal rank decomposition}
		Fix a multidegree $\Vec d=(d_X,d_E)$. Under the identification
		\[
		\textrm{Pic}^{\Vec d}(C_0)
		\cong \textrm{Pic}^{d_X}(X)\times\textrm{Pic}^{d_E}(E),
		\]
		the determinantal locus
		\[
		W_{\Vec d}^r(C_0)
		=\{L:h^0(C_0,L)\geq r+1\}
		\]
		satisfies
		\[
		\begin{aligned}
			W_{\Vec d}^r(C_0)
			={}&
			\bigcup_{\substack{x,y\geq0\\x+y=r+1}}
			\left(\OO_X(q)\otimes W_{d_X-1}^{x-1}(X)\right)
			\times
			\left(\OO_E(q)\otimes W_{d_E-1}^{y-1}(E)\right)
			\\
			&\quad\cup
			\bigcup_{\substack{x,y\geq0\\x+y=r+2}}
			W_{d_X}^{x-1}(X)\times W_{d_E}^{y-1}(E)
		\end{aligned}
		\]
		scheme-theoretically, where
		$W_a^{-1}(Z)=\textrm{Pic}^a(Z)$.
	\end{lemma}
	
	The preceding lemma expresses a rank condition on the nodal curve as a union of products. The following identities allow us to intersect such loci scheme-theoretically.
	
	\begin{lemma}\label{lemma:schematic-intersection}
		Let $S$ and $T$ be schemes over a field and let
		\[
		X_1\subseteq Z_1\subseteq S \ \textrm{and} \ X_2\subseteq Z_2\subseteq S 
		\]
		be closed subschemes.
		
		\begin{enumerate}
			\item If $Y\subseteq T$ is a closed subscheme, then we have the scheme-theoretic equality  
			\[
			\bigl[(X_1\times T)\cup(Z_1\times Y)\bigr]
			\cap
			\bigl[(X_2\times T)\cup(Z_2\times Y)\bigr] =
			\bigl[(X_1\cap X_2)\times T\bigr]
			\cup
			\bigl[(Z_1\cap Z_2)\times Y\bigr].
			\]

			\item If $Y_1,Y_2\subseteq T$ are disjoint closed subschemes, then we have the scheme-theoretical equality 
			\[
			\bigl[(X_1\times T)\cup(Z_1\times Y_1)\bigr]
			\cap
			\bigl[(X_2\times T)\cup(Z_2\times Y_2)\bigr] =
			\bigl[(X_1\cap X_2)\times T\bigr]
			\cup
			\bigl[(Z_1\cap X_2)\times Y_1\bigr]
			\cup
			\bigl[(X_1\cap Z_2)\times Y_2\bigr].
			\]
		\end{enumerate}
	\end{lemma}
	
	\begin{proof}
		The assertions are local, and hence can be checked on products of affine open subsets. We may therefore assume that $S=\operatorname{Spec}A$ and $T=\operatorname{Spec}B$. Let $I_j,K_j\subseteq A$ be the ideals of $X_j,Z_j$,
		respectively, so that $X_j\subseteq Z_j$ implies $K_j\subseteq I_j$.

		Let $J\subseteq B$ be the ideal of $Y\subseteq T$.
		The ideals of the relevant products are
		\[
		I(X_j\times T)=I_j\otimes B \qquad \textrm{and} \qquad I(Z_j\times Y)
		=
		(K_j\otimes B)+(A\otimes J).
		\]
		Hence we have
		\begin{align*}
			I\bigl((X_j\times T)\cup(Z_j\times Y)\bigr) &= (I_j\otimes B) \cap \bigl((K_j\otimes B)+(A\otimes J)\bigr) \\
			&=(K_j\otimes B) + \bigl((I_j\otimes B)\cap(A\otimes J)\bigr).
		\end{align*}
		
		Using $(I_j\otimes B)\cap(A\otimes J) = I_j\otimes J$, we obtain
		\[
		I\bigl((X_j\times T)\cup(Z_j\times Y)\bigr)
		=
		(K_j\otimes B)+(I_j\otimes J).
		\]
		Consequently, the ideal of the left-hand side is
		\[
		\bigl((K_1\otimes B)+(I_1\otimes J)\bigr)
		+
		\bigl((K_2\otimes B)+(I_2\otimes J)\bigr)
		=
		\bigl((K_1+K_2)\otimes B\bigr)
		+
		\bigl((I_1+I_2)\otimes J\bigr).
		\]
		
		On the other hand, the ideal of $(X_1\cap X_2)\times T$ is $(I_1+I_2)\otimes B$, whereas the ideal of $(Z_1\cap Z_2)\times Y$ is $\bigl((K_1+K_2)\otimes B\bigr)
		+
		(A\otimes J)$. Hence the ideal of the right-hand side is
		\[
		\begin{aligned}
			&\bigl((I_1+I_2)\otimes B\bigr)
			\cap
			\left(
			\bigl((K_1+K_2)\otimes B\bigr)
			+
			(A\otimes J)
			\right).
		\end{aligned}
		\]
		Since $K_1+K_2\subseteq I_1+I_2$, we have the equality of ideals
		\[
		\begin{aligned}
			\bigl((I_1+I_2)\otimes B\bigr)
			\cap
			\left(
			\bigl((K_1+K_2)\otimes B\bigr)
			+
			(A\otimes J)
			\right)
			&=
			\bigl((K_1+K_2)\otimes B\bigr)
			+
			\left(
			\bigl((I_1+I_2)\otimes B\bigr)
			\cap(A\otimes J)
			\right)\\
			&=
			\bigl((K_1+K_2)\otimes B\bigr)
			+
			\bigl((I_1+I_2)\otimes J\bigr).
		\end{aligned}
		\]
		This is the ideal of the left-hand side, proving the first assertion.
		
		For the second assertion, the open subsets $T\setminus |Y_1|$ and $T\setminus |Y_2|$ cover $T$. It is therefore enough to verify the equality over these two
		open subsets.
		
		Let $V=\textrm{Spec}(C)\subseteq T\setminus |Y_2|$ be an affine open subset, and let $J_{1,V}\subseteq C$ be the ideal of
		$Y_1\cap V$. Since $Y_2\cap V=\emptyset$, over $S\times V$ the
		left-hand side becomes
		\[
		\bigl[(X_1\times V)\cup(Z_1\times(Y_1\cap V))\bigr]
		\cap(X_2\times V).
		\]
		Its ideal in $A\otimes C$ is
		\[
		\left(
		(K_1\otimes C)
		+
		(I_1\otimes J_{1,V})
		\right)
		+
		(I_2\otimes C) =
		\bigl((K_1+I_2)\otimes C\bigr)
		+
		(I_1\otimes J_{1,V}).
		\]
		
		Over $S\times V$, the right-hand side becomes
		\[
		\bigl[(X_1\cap X_2)\times V\bigr]
		\cup
		\bigl[(Z_1\cap X_2)\times(Y_1\cap V)\bigr].
		\]
		Its ideal is
		\[
		\bigl((I_1+I_2)\otimes C\bigr)
		\cap
		\left(
		\bigl((K_1+I_2)\otimes C\bigr)
		+
		(A\otimes J_{1,V})
		\right).
		\]
		Since $K_1+I_2\subseteq I_1+I_2$, the same ideal identity we used before gives
		\[
		\begin{aligned}
			\bigl((I_1+I_2)\otimes C\bigr)
			\cap
			\left(
			\bigl((K_1+I_2)\otimes C\bigr)
			+
			(A\otimes J_{1,V})
			\right)
			& =
			\bigl((K_1+I_2)\otimes C\bigr)
			+
			\bigl((I_1+I_2)\otimes J_{1,V}\bigr)\\
			& =
			\bigl((K_1+I_2)\otimes C\bigr)
			+
			(I_1\otimes J_{1,V}),
		\end{aligned}
		\]
		where the last equality follows from $I_2\otimes J_{1,V}
		\subseteq I_2\otimes C$. Thus the two sides agree over $T\setminus |Y_2|$. An analogue argument proves that the two sides agree over $T\setminus |Y_1|$. Since these open subsets cover $T$, the second assertion follows.
	\end{proof}
	
	The previous two lemmas will be used extensively in the next proposition, which proves reducedness of the pointed Hurwitz--Brill--Noether loci via reducedness of the central fiber (and hence scheme-theoretic regeneration). 
	
	\begin{prop} \label{prop: reduced-intersection-right-dim}
		Let $[\pi\colon C \rightarrow \PP^1, p]$ be a generic element of $\mathcal{H}^{\textrm{tot}}_{g,k}$ and let $\Vec{e}$ be a length $k$ vector with non-negative entries. Then the intersection of the rank conditions
		$$\{L\in \textrm{Pic}^d(C) \ \mid \ h^0(C, L-(mk+i)p) \geq h^0(\mathbb{P}^1, \OO(\sigma^i\Vec{e})(-m)) \ \text{ for } m\geq 0, \ \textrm{and} \ 0\leq i \leq k-1\}$$
		is reduced, with all components of dimension $g-u(\Vec{e})- \inv(\Vec{e})$.
	\end{prop}
	
	\begin{proof} We will prove this by induction on $g$, with the case of genus $1$ proven in \cite[Lemma 7.3]{Larson2Vogt}. We assume the statement is true up to genus $g-1$ and every non-negative vector, and prove it for genus $g$. 
		
		Since reduceness is an open condition, we will consider the $1$-dimensional family of curves in Setup~\ref{setup: smoothing} and prove it for the central fiber. Let
		\[
		\rho\colon
		\mathcal{C}\times_B\operatorname{Pic}^d(\mathcal{C}/B)
		\longrightarrow
		\operatorname{Pic}^d(\mathcal{C}/B)
		\]
		be the projection, and let $\mathcal{L}$ be a universal line bundle. Choose a relative effective Cartier divisor
		$D\subseteq\mathcal{C}$, supported in the smooth locus, such that, on the central fiber $\mathcal{C}_0=X\cup_q E$, both $D|_X$ and $D|_E$ have sufficiently large degree. Since only finitely many values of $m$ impose nontrivial conditions, we may choose $D$ so that
		\[
		R^1\rho_*
		\left(
		\mathcal{L}(-(mk+i)p)(D)
		\right)
		=0
		\]
		for every relevant pair $(i,m)$, using $D$ and $p$ to also denote the pullbacks of $D$ and $p$ from $\mathcal{C}$ to $\mathcal{C}\times_B\operatorname{Pic}^d(\mathcal{C}/B)$.

		The restriction sequence induces a morphism of vector bundles
		\[
		\varphi_{i,m}\colon
		\rho_*
		\left(
		\mathcal{L}(-(mk+i)p)(D)
		\right)
		\longrightarrow
		\rho_*
		\left(
		\left.
		\mathcal{L}(-(mk+i)p)(D)
		\right|_D
		\right).
		\]
		For a point $[L]\in\operatorname{Pic}^d(\mathcal{C}/B)$ lying over $b\in B$, we have
		\[
		\ker\left(
		\left.\varphi_{i,m}\right|_{[L]}
		\right)
		\cong
		H^0\left(
		C_b,
		L(-(mk+i)p_b)
		\right).
		\]
		Consequently, the condition
		\[
		h^0\left(
		C_b,
		L(-(mk+i)p_b)
		\right)
		\geq
		h^0\left(
		\PP^1,
		\OO(\sigma^i\Vec{e})(-m)
		\right)
		\]
		is equivalent to the rank condition
		\[
		\operatorname{rk}\left(
		\left.\varphi_{i,m}\right|_{[L]}
		\right)
		\leq
		\operatorname{rk}
		\rho_*
		\left(
		\mathcal{L}(-(mk+i)p)(D)
		\right)
		-
		h^0\left(
		\PP^1,
		\OO(\sigma^i\Vec{e})(-m)
		\right).
		\]
		We consider the intersection of all rank conditions above, for each relevant pair $(i,m)$. Our goal is to prove that on the central fiber these loci cut out exactly the following locus, with the reduced scheme structure: 
		
		\[\left(\mathcal{O}_X(q)\otimes \overline{W}^{\Vec{e}}(\pi_{|X},q)\right)\times \mathrm{Pic}^0(E) \cup  \bigcup_{t_i(\Vec{e})\neq \Vec{e}} \left(\OO_X(q) \otimes \overline{W}^{t_i(\Vec{e})}(\pi_{|X},q) \right) \times \{ \mathcal{O}_E(i\cdot p -i\cdot q)\} \]
		where the union is over $0\leq i \leq k-1$ satisfying $t_i(\Vec{e}) \neq \Vec{e}$.

		For simplicity of notation, we define
		
		\[A(d_E, x) \coloneqq \{L_X \in \textrm{Pic}^{g+r}(X) \ \mid \ h^0(X, L_X(-d_E\cdot q))\geq x\},\]
		
		\[B(d_E,y) \coloneqq \{L_E \in \textrm{Pic}^{0}(E)  \ \mid \ h^0(E, L_E(d_E\cdot q-(mk+i)\cdot p))\geq y\},\]
		
		and moreover we denote by $r_{i,m} \coloneqq h^0(\PP^1, \OO(\sigma^i\Vec{e})(-m))$.
		Using Lemma~\ref{lemma: nodal rank decomposition}, as well as~\cite[Lemma 7.1]{Larson2Vogt}, the rank condition for a relevant pair $(i,m)$ cuts out the locus 
		
		\[ \bigcup_{x+y =r_{i,m}}\{L_X \ \mid \ h^0(X, L_X(-q))\geq x\} \times \{L_E \ \mid \ h^0(E, L_E(-(mk+i)\cdot p - q))\geq y\}\]
		\[ \bigcup\]
		\[ \bigcup_{x+y = r_{i,m}+1}\{L_X \ \mid \ h^0(X, L_X)\geq x\} \times \{L_E \ \mid \ h^0(E, L_E(-(mk+i)\cdot p))\geq y\}\]
		in the central fiber $\Pic^{d_X}(X) \times \Pic^{d_E}(E)$, and we need to take the intersection of these loci over all possible multidegrees $d_X + d_E = g+r$. By twisting the line bundles in the central fiber, we can view all these conditions in $\Pic^{g+r}(X) \times \Pic^{0}(E)$, where the locus above becomes
		\[ \bigcup_{x+y =r_{i,m}} A(d_E+1, x) \times B(d_E-1, y) \cup \bigcup_{x+y = r_{i,m}+1}A(d_E,x) \times B(d_E,y).\]
		Next, we remark that, as we vary $d_E$ in $\mathbb{Z}$, many conditions in this union are redundant. 
		
		First, if the degree $d_E -1-mk-i$ is negative, then scheme-theoretically we have the equality
		\[ \bigcup_{x+y =r_{i,m}} A(d_E+1, x)\times B(d_E-1,y) = A(d_E+1,r_{i,m}) \times \Pic^0(E).\]
		Indeed, if $y > 0$ then $B(d_E-1, y) = \emptyset$, while if $y<0$, we have the scheme-theoretic inclusion 
		\[ A(d_E+1, x+1)\times \Pic^0(E) \subseteq A(d_E+1, x)\times \Pic^0(E). \]
		Similarly, when $d_E-1-mk-i \geq 1$, we always have the equality 
		\[ h^0\bigl(E, L_E((d_E-1)\cdot q-(mk+i)\cdot p)\bigr) =  d_E-1-mk-i.\]
		A similar argument as above gives in this case:
		\[ \bigcup_{x+y = r_{i,m}} A(d_E+1, x)\times B(d_E-1,y) = A(d_E+1, r_{i,m} + 1 +mk+i - d_E ) \times \Pic^0(E).\]
		
		Lastly, when $d_E-1-mk-i = 0$ a degree $0$ line bundle on $E$ can have $h^0(E,-)$ equal to either $0$ or $1$. In particular,it is sufficient to consider the cases $y=0$ and $y = 1$ in the union, i.e. 
		\[ \bigcup_{x+y = r_{i,m}} A(d_E+1, x)\times B(d_E-1,y) = A(d_E+1, r_{i,m})\times \Pic^0(E) \cup A(d_E+1, r_{i,m}-1) \times \{\OO_E(i\cdot p - i\cdot q)\}. \]
		
		A similar argument works for the union 
		\[ \bigcup_{x+y = r_{i,m}+1} A(d_E, x)\times B(d_E,y),\]
		which is scheme theoretically equal to 
		\begin{align*}
			& A(d_E, r_{i,m}+1 ) \times \Pic^0(E), \ \text{ if }d_E - mk-i < 0; \ \\
			& A(d_E, r_{i,m}+1+mk+i -d_E) \times \Pic^0(E),
			\text{ if } d_E - mk-i \geq 1; \text{ or } \\
			& A(d_E, r_{i,m}+1) \times \Pic^0(E) \cup A(d_E, r_{i,m}) \times \{\OO_E(i\cdot p - i\cdot q)\}, \text{ if } d_E - mk-i = 0.
		\end{align*}
		
		Indeed, we compute the union as $d_E$ varies in $\mathbb{Z}$.

		If $d_E < mk+i$, the union is scheme-theoretically just 
		\[ A(d_E+1, r_{i,m}) \times \Pic^0(E) \cup A(d_E, r_{i,m}+1) \times \Pic^0(E) = A(d_E+1, r_{i,m}) \times \Pic^0(E). \]
		
		If $d_E = mk+i$, the scheme-theoretical union is 
		\[ A(d_E+1, r_{i,m}) \times \Pic^0(E) \cup A(d_E, r_{i,m}+1) \times \Pic^0(E) \cup A(d_E,r_{i,m})\times \{\OO_E(i\cdot p - i\cdot q)\}.\]
		As before, the second term is contained scheme theoretically in the first term, and hence the union is
		\[ A(d_E+1, r_{i,m}) \times \Pic^0(E) \cup A(d_E,r_{i,m})\times \{\OO_E(i\cdot p - i\cdot q)\}.\]
		
		If $d_E = mk+i+1$, the union is 
		\[ A(d_E+1, r_{i,m})\times \Pic^0(E) \cup A(d_E+1, r_{i,m}-1) \times \{\OO_E(i\cdot p - i\cdot q)\} \cup A(d_E, r_{i,m}) \times \Pic^0(E). \]
		The first term is contained in the third and hence the union is 
		\[A(d_E+1, r_{i,m}-1) \times \{\OO_E(i\cdot p - i\cdot q)\} \cup A(d_E, r_{i,m}) \times \Pic^0(E). \]
		
		If $d_E > mk+i+1$, the union is
		\[A(d_E+1, r_{i,m} + 1 +mk+i - d_E ) \times \Pic^0(E) \cup A(d_E, r_{i,m}+1+mk+i -d_E) \times \Pic^0(E).\]
		The first term is contained in the second and hence this is just 
		\[A(d_E, r_{i,m}+1+mk+i -d_E) \times \Pic^0(E).\]
		
		We are now ready to consider the intersection of these rank degeneracy condition as the degree $d_E$ varies over $\mathbb{Z}$. We first make two observations, that 
		\[\text{for all $d_E< mk+1$, we have } A(mk+i, r_{i,m}) \subseteq A(d_E+1, r_{i,m}), \ \textrm{and} \] 
		\[\text{for all $d_E>mk+i+1$, we have } A(mk+i+2, r_{i,m}-1) \subseteq A(d_E, r_{i,m}+1+mk+i-d_E). \]
		In particular, the intersection of these rank conditions as $d_E$ varies over $\mathbb{Z}$ gives the same scheme as the intersection over $d_E$ between $mk+i-1$ and $mk+i+2$. 
		
		Furthermore, we have the scheme-theoretic inclusions
		\[
		A(mk+i+1,r_{i,m})
		\subseteq
		A(mk+i,r_{i,m})
		\]
		and
		\[
		A(mk+i+1,r_{i,m})
		\subseteq
		A(mk+i+2,r_{i,m}-1).
		\]
		Consequently, the first and fourth loci are redundant and it is enough to consider \[d_E \in \{mk+i,mk+i+1\}\] in the intersection.
		
		Using the general fact of scheme-theoretic distributivity of
		intersection across unions in Lemma~\ref{lemma:schematic-intersection}, we obtain
		\begin{gather*}
			\bigcap_{d_E\in\mathbb Z}
			\Bigg[
			\bigcup_{x+y=r_{i,m}}
			A(d_E+1,x)\times B(d_E-1,y)
			\qquad\cup
			\bigcup_{x+y=r_{i,m}+1}
			A(d_E,x)\times B(d_E,y)
			\Bigg]   \\
			=
			\Big(A(mk+i+1,r_{i,m})\times\Pic^0(E) \Big)
			\cup
			\Big(
			A(mk+i,r_{i,m})
			\cap
			A(mk+i+2,r_{i,m}-1)
			\Big)\times
			\{\OO_E(i(p-q))\}.
		\end{gather*}
		
		We now intersect over all relevant pairs $(i,m)$. Since $p-q$ has
		exact order $k$, the points
		\[
		\OO_E(i(p-q)),
		\qquad 0\leq i\leq k-1,
		\]
		are pairwise distinct. Therefore, a repeated application of Lemma~\ref{lemma:schematic-intersection} gives
		\begin{align*}
			&
			\bigcap_{\substack{m\geq0\\0\leq i\leq k-1}}
			\bigcap_{d_E\in\mathbb Z}
			\Bigg[
			\bigcup_{x+y=r_{i,m}}
			A(d_E+1,x)\times B(d_E-1,y) \quad
			\cup \quad
			\bigcup_{x+y=r_{i,m}+1}
			A(d_E,x)\times B(d_E,y)
			\Bigg]\\
			&=
			\left(
			\bigcap_{\substack{m\geq0\\0\leq i\leq k-1}}
			A(mk+i+1,r_{i,m})
			\right)
			\times\Pic^0(E)\quad\cup\\
			&
			\scalebox{0.85}{$\displaystyle
				\bigcup_{i=0}^{k-1}
				\Bigg[
				\left(
				\bigcap_{m\geq0}
				\big(
				A(mk+i,r_{i,m})
				\cap
				A(mk+i+2,r_{i,m}-1)
				\big)
				\right)\cap
				\Bigg(
				\bigcap_{\substack{m\geq0\\0\leq j\leq k-1\\j\neq i}}
				A(mk+j+1,r_{j,m})
				\Bigg)
				\Bigg]
				\times
				\{\OO_E(i(p-q))\}.
				$}
		\end{align*}
		If we replace $L_X$ by $L_X(-q)$, the first intersection is exactly
		the genus $g-1$ locus defined by the rank conditions associated to
		$\Vec e$. Indeed,
		\[
		h^0\bigl(X,L_X-(mk+i+1)q\bigr)
		=
		h^0\bigl(X,L_X(-q)-(mk+i)q\bigr),
		\]
		and by the induction hypothesis
		\[ \bigcap_{\substack{m\geq0\\0\leq i\leq k-1}}
		A(mk+i+1,r_{i,m}) = \OO_X(q) \otimes \overline{W}^{\Vec{e}}(\pi_X\colon X \rightarrow P_X, q), \]
		with the reduced scheme structure. 
		
		For each fixed $i$, we claim that the intersection appearing in
		square brackets is
		\[
		\OO_X(q)\otimes
		\overline{W}^{t_i(\Vec e)}(\pi_X,q).
		\]
		To identify its defining conditions, put $M=L_X(-q)$. The displayed
		intersection imposes
		\[
		h^0\bigl(X,M(-(mk+j)q)\bigr)\geq r_{j,m} \text{ for } j\neq i,
		\]
		together with
		\[
		h^0\bigl(X,M(-(mk+i-1)q)\bigr)\geq r_{i,m},
		\]
		and
		\[
		h^0\bigl(X,M(-(mk+i+1)q)\bigr)\geq r_{i,m}-1.
		\]
		We first explain which conditions at the omitted twists $mk+i$
		are implicit in these inequalities. We will use the conventions
		\[
		r_{-1,m}=r_{k-1,m-1},
		\qquad
		r_{k,m}=r_{0,m+1}.
		\]
		Fix $m$ with $mk+i>0$. At the preceding twist, the intersection
		imposes both the unchanged condition with bound $r_{i-1,m}$
		and the additional condition with bound $r_{i,m}$, i.e.
		\[ h^0\bigl(X,M(-(mk+i-1)q)\bigr)\geq r_{i,m}
		\ \textrm{ and } \ h^0\bigl(X,M(-(mk+i-1)q)\bigr)\geq r_{i-1,m}. \]
		At the following twist, it imposes the conditions
		\[ h^0\bigl(X,M(-(mk+i+1)q)\bigr)\geq r_{i,m}-1 \ \textrm{ and } \ h^0\bigl(X,M(-(mk+i+1)q)\bigr)\geq r_{i+1,m}.\]
		Since
		\[
		r_{i-1,m}\geq r_{i,m},
		\qquad
		r_{i+1,m}\geq r_{i,m}-1,
		\]
		the two additional conditions displayed above are redundant.
		
		Moreover, removing one point decreases the dimension of sections
		by at most one, while adding one point cannot decrease it.
		Consequently, the conditions at the preceding and following
		twists imply
		\[
		h^0\bigl(X,M(-(mk+i)q)\bigr)
		\geq
		\max\{r_{i-1,m}-1,\ r_{i+1,m}\}.
		\]
		Equivalently, in the original notation, the implicit conditions are
		\[
		h^0\bigl(X,L_X-(mk+i+1)q\bigr)
		\geq
		\max\{r_{i-1,m}-1,\ r_{i+1,m}\}.
		\]
		These implications hold scheme-theoretically for the determinantal
		loci. Indeed, the restriction matrices for consecutive twists
		are obtained from one another by adding or deleting one column.
		The corresponding inclusions of determinantal ideals follow from
		containment of minors and expansion along this column.
		Thus adjoining the implicit conditions does not change the
		scheme structure of the intersection.
		
		We will show that the intersection 
		\[ \Bigg[
		\left(
		\bigcap_{m\geq0}
		\big(
		A(mk+i,r_{i,m})
		\cap
		A(mk+i+2,r_{i,m}-1)
		\big)
		\right)\cap
		\Bigg(
		\bigcap_{\substack{m\geq0\\0\leq j\leq k-1\\j\neq i}}
		A(mk+j+1,r_{j,m})
		\Bigg)
		\Bigg]\]
		cuts out the locus $ \OO_X(q) \otimes \overline{W}^{t_i(\Vec{e})}(\pi_X,q)$. For this, we will show that the inequalities on the global sections of $M$ are exactly those imposed by the vanishing sequence $\textbf{a}(t_i(\Vec e))$. Notice that we have 
		\[
		r_{j,m}
		=
		\#\{a\in\mathbf a(\Vec e)\mid a\geq mk+j\},
		\]
		and consequently we obtain
		$$     h^0\left(\PP^1, \OO(\sigma^jt_i(\Vec{e})(-m)\right) = \#\{a\in\mathbf a(t_i(\Vec e))\mid a\geq mk+j\}= r_{j,m} \ \textrm{when} \ j \neq i.$$ 
		
		Moreover, we have 
		\[
		r_{i-1,m}-r_{i,m}
		=
		\begin{cases}
			1,& mk+i-1\in\mathbf a(\Vec e),\\
			0,& mk+i-1\notin\mathbf a(\Vec e),
		\end{cases}
		\]
		and
		\[
		r_{i,m}-r_{i+1,m}
		=
		\begin{cases}
			1,& mk+i\in\mathbf a(\Vec e),\\
			0,& mk+i\notin\mathbf a(\Vec e).
		\end{cases}
		\]
		It follows that
		\[
		\max\{r_{i-1,m}-1,\ r_{i+1,m}\}
		=
		\begin{cases}
			r_{i,m}-1,
			&
			\begin{array}{l}
				mk+i\in\mathbf a(\Vec e)\ \textrm{and} \
				mk+i-1\notin\mathbf a(\Vec e),
			\end{array}
			\\[2mm]
			r_{i,m},&\text{otherwise}.
		\end{cases}
		\]
		Thus the bound at $mk+i$ decreases by one precisely when
		$mk+i$ is an entry of $\mathbf a(\Vec e)$ and its predecessor
		is not. In other words, we have the equality 
		$$     h^0\left(\PP^1, \OO(\sigma^it_i(\Vec{e})(-m)\right) = \#\{a\in\mathbf a(t_i(\Vec e))\mid a\geq mk+i\}= \max\{r_{i-1,m}-1,\ r_{i+1,m}\}.$$ 
		
		We have therefore obtained, scheme-theoretically, exactly the
		conditions
		\[
		h^0\bigl(X,M(-(mk+j)q)\bigr)
		\geq
		h^0\bigl(
		\PP^1,\OO(\sigma^j t_i(\Vec e))(-m)
		\bigr),
		\qquad m\geq0,\quad 0\leq j\leq k-1.
		\]
		
		In particular, using the induction hypothesis, the locus in $\Pic^{g+r}(X) \times \Pic^0(E)$ respecting all rank conditions is scheme-theoretically 
		\[ \OO_X(q) \otimes \overline{W}^{\Vec{e}}(\pi_X,q) \times \Pic^0(E) \cup \bigcup_{i=0}^{k-1} \OO_X(q) \otimes \overline{W}^{t_i(\Vec{e})}(\pi_X,q) \times \{\OO_E(i(p-q))\}  \]
		with the reduced scheme structure. This is scheme-theoretically the same as 
		\[ \OO_X(q) \otimes \overline{W}^{\Vec{e}}(\pi_X,q) \times \Pic^0(E) \cup \bigcup_{t_i(\Vec{e})\neq \Vec{e}} \OO_X(q) \otimes \overline{W}^{t_i(\Vec{e})}(\pi_X,q) \times \{\OO_E(i(p-q))\}  \]
		with the reduced scheme structure. Thus the central fiber is reduced and all its components have dimension $g-u(\Vec e)-\inv(\Vec e)$.
		
		In particular, the central fiber of $\mathcal{W}^{\Vec e}(\pi,p)$ is contained scheme-theoretically in
		\[ \OO_X(q) \otimes \overline{W}^{\Vec{e}}(\pi_X,q) \times \Pic^0(E) \cup \bigcup_{i=0}^{k-1} \OO_X(q) \otimes \overline{W}^{t_i(\Vec{e})}(\pi_X,q) \times \{\OO_E(i(p-q))\}  \]
		endowed with the reduced structure.
		By Theorem~\ref{theorem: HBN regeneration}, this inclusion is
		an equality on supports. Since the containing scheme is reduced,
		it is therefore an equality of schemes.
		
		Finally, assume that the rank conditions cut out an additional generic component $\Delta$ and consider its closure over the central fiber. In particular, its central fiber $\Delta_0$ is contained in the reduced locus outlined above. It follows that both $\Delta$ and $\mathcal{W}^{\Vec e}(\pi,p)$ contain $\Delta_0$ in their degeneration, so the special fiber must have multiplicity at least two along $\Delta_0$.  This contradicts the reducedness
		of the central fiber.
		
		Therefore, the generic fiber is reduced
		and has pure dimension $g-u(\Vec e)-\inv(\Vec e)$. This proves the proposition.
	\end{proof}
	
	The proof of Proposition~\ref{prop: reduced-intersection-right-dim} also establishes the following scheme-theoretic regeneration statement, which we state separately for its independent interest.
	
	\begin{prop}[Scheme-theoretic regeneration]
		\label{prop: scheme-theoretic regeneration}
		Let $\mathcal C\xrightarrow{\mathfrak\pi}\mathcal P\longrightarrow B$ be a smoothing family as in Setup~\ref{setup: smoothing} and denote $\mathcal C^*\rightarrow \mathcal P^*\longrightarrow B^*$ the family defined away from the central fiber. Let
		$\Vec e\in\mathbb Z_{\geq0}^k$ and take
		$r=k+|\Vec e|-1$ and $d=g+r$.
		
		We consider the schematic closure of
		$\overline{\mathcal{W}}^{\Vec e}(\mathcal C^*\rightarrow \mathcal P^*)$, endowed with its
		determinantal scheme structure and consider $\overline{\mathcal{W}}^{\Vec e}(\mathfrak{\pi},p)$ its closure in the relative Picard space $\textrm{Pic}^d(\mathcal{C}/B)$. Then $\overline{\mathcal{W}}^{\Vec e}(\mathfrak{\pi},p)$ is flat and proper over $B$ and its special fiber in $\textrm{Pic}^d(X)\times \textrm{Pic}^0(E)$ is scheme-theoretically equal to
		\[
		\bigl(\OO_X(q)\otimes
		\overline W^{\Vec e}(\pi_X,q)\bigr)\times Pic^0(E)
		\quad\cup
		\bigcup_{\substack{0\leq i<k\\t_i(\Vec e)\neq\Vec e}}
		\bigl(\OO_X(q)\otimes
		\overline W^{t_i(\Vec e)}(\pi_X,q)\bigr)
		\times\{\OO_E(i(p-q))\},
		\]
		endowed with the reduced structure.
	\end{prop}
	
	As a consequence of Proposition~\ref{prop: reduced-intersection-right-dim}, we obtain the equality 
	\[\bigcap_{i=0}^{k-1} \overline{\{L \in \textrm{Pic}(C) \ \mid \ \pi_*(L-(i\cdot p))\cong \OO(\sigma^i\Vec{e}) \} } = \overline{W}^{\Vec{e}}(\pi,p). \]
	This gives a complete description of the closure and we obtain the following criterion.
	
	\begin{cor}
		Let $[\pi,p]$ a generic element of $\cHTot$ and let $\Vec{e}$, $\Vec{f}$ be two length $k$ vectors with nonnegative entries. Then we have the equivalence 
		\[ \textbf{a}(\Vec{e}) \leq \textbf{a}(\Vec{f}) \ \textrm{entrywise} \ \iff W^{\Vec{f}}(\pi,p) \subseteq  \overline{W}^{\Vec{e}}(\pi,p). \]
	\end{cor}
	
	We are now ready to show Cohen--Macaulayness by an induction on $g$. The argument follows immediately from Proposition~\ref{prop: reduced-intersection-right-dim}, Proposition~\ref{prop: scheme-theoretic regeneration} and induction, using the following criterion.
	
	\begin{lemma}[{\cite[Lemma 7.7]{Larson2Vogt} and \cite[Proposition 4.1]{Hoc16}}] Let $A, B$ be two Cohen--Macaulay schemes with $A\cap B$ having codimension $1$ in both of them. Then $A\cup B$ is Cohen--Macaulay if and only if $A\cap B$ is Cohen--Macaulay. 
	\end{lemma}
	
	\begin{theorem}
		\label{theorem: reduced and CM}
		Let $[\pi\colon C\longrightarrow\mathbb{P}^1,p]
		\in\mathcal{H}^{\mathrm{tot}}_{g,k}$ be general, and let $\Vec{e}\in\mathbb{Z}_{\geq 0}^k$. Then
		$\overline{W}^{\Vec{e}}(\pi,p)$ is Cohen--Macaulay.
	\end{theorem}
	
	\begin{proof}
		We argue by induction on $g$. In genus $1$, every nonempty closed
		locus is either the Picard variety or a reduced zero-dimensional
		scheme, and is therefore Cohen--Macaulay.
		
		Assume the assertion holds in genus $g-1$. We will use the same degeneration of the $k\colon 1$ cover as in Setup~\ref{setup: smoothing}. By the scheme-theoretic
		description of the special fiber, it is the union of
		\[
		\left(
		\OO_X(q)\otimes\overline{W}^{\Vec e}(\pi_X,q)
		\right)\times\Pic^0(E)
		\]
		and
		\[
		\left(
		\OO_X(q)\otimes\overline{W}^{t_i(\Vec e)}(\pi_X,q)
		\right)\times\{\OO_E(i(p-q))\},
		\qquad t_i(\Vec e)\neq\Vec e.
		\]
		All these schemes are Cohen--Macaulay by the induction hypothesis.
		Moreover, the scheme-theoretic inclusion
		\[
		\overline{W}^{\Vec e}(\pi_X,q)
		\subseteq
		\overline{W}^{t_i(\Vec e)}(\pi_X,q)
		\]
		implies that the intersection of the first scheme with the piece
		indexed by $i$ is
		\[
		\left(
		\OO_X(q)\otimes\overline{W}^{\Vec e}(\pi_X,q)
		\right)\times\{\OO_E(i(p-q))\}.
		\]
		This intersection is Cohen--Macaulay by induction. When nonempty,
		it has pure dimension
		\[
		g-1-u(\Vec e)-\inv(\Vec e),
		\]
		whereas both schemes being intersected have pure dimension
		\[
		g-u(\Vec e)-\inv(\Vec e).
		\]
		Here we use
		\[
		u(t_i(\Vec e))+\inv(t_i(\Vec e))
		=
		u(\Vec e)+\inv(\Vec e)-1
		\]
		for $t_i(\Vec e)\neq\Vec e$. Thus the intersection has
		codimension one in both schemes.
		
		Since $p-q$ has exact order $k$, the pieces indexed by distinct
		$i$ are disjoint. Adjoining them one at a time, their intersections
		with the preceding union are therefore exactly the intersections
		displayed above. Repeated application of
		\cite[Lemma~7.7]{Larson2Vogt}
		shows that the special fiber is Cohen--Macaulay. Since being Cohen--Macaulay is an open condition on the base $B$, it follows that the generic fiber is Cohen--Macaulay. Hence
		$\overline{W}^{\Vec e}(\pi,p)$ is Cohen--Macaulay. 
	\end{proof}
	
	\subsection{Irreducibility and Monodromy of Pointed Hurwitz--Brill--Noether loci}
	
	We note that these results have already appeared in~\cite{Agg26} in terms of permutations in the $k$-affine symmetric group. Using notation from~\cite{Agg26}, we have 
	\[ \overline{W}^{\Vec{e}}(\pi,p) = W^{\tau}(C,p),  \]
	where $\tau$ is the element of the $k$-affine symmetric group defined by
	\[ \tau(i) = -ke_i - i \ \text{ for } \ 0 \leq i \leq k-1. \]
	
	Now, irreducibility and monodromy are simply reformulations of \cite[Theorem 1.1]{Agg26}.
	\begin{prop}
		\label{prop: irreducible}
		Let $[\pi\colon C\longrightarrow\mathbb{P}^1,p]
		\in\mathcal{H}^{\mathrm{tot}}_{g,k}$ be general and $\Vec{e}\in\mathbb{Z}_{\geq 0}^k$ with $g-\inv(\Vec{e})-u(\Vec{e}) > 0$. Then
		$\overline{W}^{\Vec{e}}(\pi,p)$ is irreducible.   
	\end{prop}
	
	\begin{prop}
		\label{prop: monodromy}
		Let $\Vec{e}\in\mathbb{Z}_{\geq 0}^k$ such that $g-\inv(\Vec{e})-u(\Vec{e}) = 0$ and consider the universal pointed Hurwitz--Brill--Noether locus $\mathcal{W}_g^{\Vec{e}}$, parametrizing tuples .
		\[ [\pi\colon C\longrightarrow\mathbb{P}^1,p, L] \]
		with $[\pi\colon C\longrightarrow\mathbb{P}^1,p] \in \mathcal{H}^{\mathrm{tot}}_{g,k}$ and $[L]$ in $W^{\Vec{e}}(\pi,p)$. Then
		$\mathcal{W}^{\Vec{e}}_g$ has a unique irreducible component dominating $\mathcal{H}^{\mathrm{tot}}_{g,k}$.   
	\end{prop}
	
	\section{The algebraic Euler characteristic of pointed Hurwitz--Brill--Noether loci} \label{sec: Euler-char}
	
	We now use the degeneration constructed in the previous sections to study the
	algebraic Euler characteristic of pointed Hurwitz--Brill--Noether loci.
	Throughout this section, for a proper scheme $X$ we write
	\[
	\chi(X):=\chi(X,\mathcal O_X)
	=
	\sum_{i\geq 0}(-1)^i h^i(X,\mathcal O_X),
	\]
	its algebraic Euler characteristic. The main advantage of this invariant in our setting is that it is preserved
	under flat proper degenerations. More precisely, let
	\[
	\mathcal C\longrightarrow \mathcal P\longrightarrow B
	\]
	be the one parameter degeneration in Setup~\ref{setup: smoothing}, whose generic
	fiber is a smooth degree $k$ cover totally ramified at the marked point and whose special fiber has source $C_0=X\cup_q E$.
	
	Let $\overline{\mathcal W}^{\Vec e}(\mathfrak{\pi}, p)\longrightarrow B$
	denote the corresponding relative pointed Hurwitz--Brill--Noether locus, which is proper and flat. Therefore the algebraic Euler
	characteristic of its fibers is constant, and hence for a smooth fiber
	$[\pi:C\to\PP^1,p]$ and the special fiber
	$W^{\Vec e}_0(\pi,p)$, we have
	\[	\chi\bigl(W^{\Vec e}(\pi,p)\bigr) = \chi\bigl(W^{\Vec e}_0(\pi,p)\bigr).
	\]
	
	This is particularly useful in dimensions zero and one. If
	$W^{\Vec e}(\pi,p)$ is reduced and zero-dimensional, then
	\[
	\chi\bigl(W^{\Vec e}(\pi,p)\bigr)
	=
	\# W^{\Vec e}(\pi,p),
	\]
	so the cardinality of the pointed Hurwitz--Brill--Noether locus may be
	computed on the special fiber. If $W^{\Vec e}(\pi,p)$ is a proper
	one-dimensional scheme, then
	\[
	\chi\bigl(W^{\Vec e}(\pi,p)\bigr)
	=
	1-p_a\bigl(W^{\Vec e}(\pi,p)\bigr),
	\]
	where $p_a$ denotes the arithmetic genus. In particular, because the general
	fiber is smooth (see for example~\cite[Theorem 1.1]{Agg26}), its genus is determined by the Euler characteristic of the special fiber. 
	
	It is well-known that the algebraic Euler characteristic is well-behaved with respect to intersections and unions. If $A$ and $B$ are two proper schemes, both contained in the same ambient space $X$, then we have the short exact sequence 
	\[ 0 \rightarrow \OO_{A\cup B} \rightarrow \OO_A \oplus \OO_B \rightarrow \OO_{A\cap B} \rightarrow 0\]
	and consequently the relation 
	\[ \chi(A\cup B) = \chi(A) + \chi(B)- \chi(A\cap B). \]
	We apply this formula to the equality $\chi\bigl(W^{\Vec e}(\pi,p)\bigr) = \chi\bigl(W^{\Vec e}_0(\pi,p)\bigr).$ Using the description of the special fiber $W^{\Vec e}_0(\pi,p)$ in Theorem~\ref{theorem: HBN regeneration} and Proposition~\ref{prop: scheme-theoretic regeneration} we obtain the following.
	
	\begin{theorem} \label{thm: hol-Euler-char}
		Let $\Vec{e}$ be a length $k$ vector with non-negative entries. We denote $\chi_n(\Vec{e}) \coloneqq \chi\bigl(W^{\Vec e}(\pi,p)\bigr)$ for $[\pi,p]$ a generic element of $\cHTot$ where $g = n + u(\Vec{e})+ \inv(\Vec{e})$; and denote 
		\[ I(\Vec{e}) \coloneqq \{i \ \mid \ 0\leq i \leq k-1 \ \textrm{and} \ t_i(\Vec{e})\neq \Vec{e}\}. \]
		Then we have the recursion
		\[ \chi_n(\Vec e)
		=
		\sum_{i\in I(\Vec e)}
		\chi_n\bigl(t_i(\Vec e)\bigr)
		-
		|I(\Vec e)|\,\chi_{n-1}(\Vec e).\]
	\end{theorem}
	\begin{proof}
		This formula is an immediate consequence of the above Euler characteristic formula for union and intersection, of Proposition~\ref{prop: containments}, Proposition~\ref{prop: scheme-theoretic regeneration} and of the equalities 
		\[ \chi \left(W^{\Vec{e}}(\pi_X, q) \times \textrm{Pic}^0(E)\right) = \chi_{n-1}(\Vec{e}) \cdot \chi(\textrm{Pic}^0(E)) \ \textrm{and} \ \chi(\textrm{Pic}^0(E))=0.\qedhere
		\]
	\end{proof}
	\subsection{Ducklings: a simple rule for studying maximal Hurwitz--Brill--Noether loci} 
	
	We consider a maximal Hurwitz--Brill--Noether locus whose associated vector has the form 
	\[ \Vec{e} = (\underbrace{a+1, \ldots, a+1}_{k_1 \ \textrm{times}}, \underbrace{a,\ldots a}_{k_2 \ \textrm{times}}, \underbrace{1, \ldots, 1}_{k_3 \ \textrm{times}}, \underbrace{0,\ldots, 0}_{k_4 \ \textrm{times}}), \]
	where $a\geq 2$, $k_1 >0 $, the integers $k_i$ are nonnegative, and $k_1+k_2+k_3+k_4=k$. Write $m=k_1+k_2$, and assume $0<m<k$, so that both balanced blocks are nonempty.
	Our goal is reinterpret, when $g = u(\Vec{e})$, the recursion
	\[ \chi_0(\Vec e) = \sum_{t_i(\Vec e) \neq \Vec{e}}
	\chi_0\bigl(t_i(\Vec e)\bigr) \]
	as an elementary counting problem. 
	
	For the vector $\Vec{e}$, arrange the numbers in $\textbf{a}(\Vec{e})$ in a table as in Proposition~\ref{prop: ramification game} and put a duckling in all southernmost squares having non-zero displacement (see Remark~\ref{remark: southernomst description}, i.e. the ones corresponding to the first $k_1 + k_2$ columns). We will call \textit{the leading duckling} to be the one whose associated position has the smallest number. We denote by $\Delta$ the displacement of the leading duckling. A move of a duckling will represent a move according to the transformations in Definition~\ref{def: t_i} and we can give a precise description of the final vector (i.e. the vector $\Vec{f}$ such that $\textbf{a}(\Vec{f})$ consists of consecutive numbers). Getting to vector $\Vec{f}$ via the transformations is the same as moving the ducklings until they reach the positions numbered $|\Vec{e}|+k - k_1-k_2 +1, \ldots |\Vec{e}|$.
	
	\newcommand{\duckboardgrid}{
		% Grid
		\foreach \x in {0,...,5}
		\draw[gray] (\x,0) -- (\x,7);
		\foreach \y in {0,...,7}
		\draw[gray] (0,\y) -- (5,\y);
		
		% Final positions 15,16,17
		\fill[
		pattern=north east lines,
		pattern color=red
		] (0,3) rectangle (1,4);
		
		\fill[
		pattern=north east lines,
		pattern color=red
		] (1,3) rectangle (2,4);
		
		\fill[
		pattern=north east lines,
		pattern color=red
		] (2,3) rectangle (3,4);
	}
	
	% Put the number n in its square
	\newcommand{\duckboardnumber}[1]{
		\pgfmathtruncatemacro{\duckcol}{mod(#1,5)}
		\pgfmathtruncatemacro{\duckrow}{floor(#1/5)}
		\node at
		(\duckcol+0.5,6.5-\duckrow)
		{\small #1};
	}
	
	% Put a duck in square n
	\newcommand{\duckboardsduck}[1]{
		\pgfmathtruncatemacro{\duckcol}{mod(#1,5)}
		\pgfmathtruncatemacro{\duckrow}{floor(#1/5)}
		\node[inner sep=0pt] at
		(\duckcol+0.5,6.45-\duckrow)
		{\duckicon};
	}

	\begin{figure}[H]
		\centering
		
		\begin{tikzpicture}[scale=0.65]
			
			% ============================================================
			% FIRST CONFIGURATION: e=(5,4,4,0,0)
			% ============================================================
			\begin{scope}[shift={(0,0)}]
				
				\duckboardgrid
				
				% Boundary
				\draw[very thick] (0,1) -- (1,1);
				\draw[very thick] (1,1) -- (1,2);
				\draw[very thick] (1,2) -- (3,2);
				\draw[very thick] (3,2) -- (3,6);
				\draw[very thick] (3,6) -- (5,6);
				\draw[very thick] (1,2) -- (3,2);
				% Occupied squares, except for the three ducklings
				\foreach \n in {
					0,1,2,3,4,
					5,6,7,
					10,11,12,
					15,16,17,
					20
				}{
					\duckboardnumber{\n}
				}
				
				% Ducklings at 21,22,25
				\duckboardsduck{21}
				\duckboardsduck{22}
				\duckboardsduck{25}
				
				\node at (2.5,-0.65)
				{$\Vec e=(5,4,4,0,0)$};
				
			\end{scope}

			% ------------------------------------------------------------
			% t_0
			% ------------------------------------------------------------
			\draw[->,very thick]
			(5.4,3.5)
			--
			node[above] {move $t_0$}
			(7.7,3.5);

			% ============================================================
			% SECOND CONFIGURATION: t_0(e)=(1,4,4,0,4)
			% ============================================================
			\begin{scope}[shift={(8,0)}]
				
				\duckboardgrid
				
				% Boundary for (1,4,4,0,4)
				\draw[very thick] (0,5) -- (1,5);
				\draw[very thick] (1,5) -- (1,2);
				\draw[very thick] (1,2) -- (3,2);
				\draw[very thick] (3,2) -- (3,6);
				\draw[very thick] (3,6) -- (4,6);
				\draw[very thick] (4,6) -- (4,2);
				\draw[very thick] (4,2) -- (5,2);
				
				% Occupied squares, except ducklings
				\foreach \n in {
					0,1,2,3,4,
					5,6,7,9,
					11,12,14,
					16,17,19
				}{
					\duckboardnumber{\n}
				}
				
				% t_0 moves the duckling 25 -> 24
				\duckboardsduck{21}
				\duckboardsduck{22}
				\duckboardsduck{24}
				
				\node at (2.5,-0.65)
				{$t_0(\Vec e)=(1,4,4,0,4)$};
				
			\end{scope}

			% ------------------------------------------------------------
			% t_1
			% ------------------------------------------------------------
			\draw[->,very thick]
			(13.4,3.5)
			--
			node[above] {\large move $t_1$}
			(15.6,3.5);

			% ============================================================
			% THIRD CONFIGURATION: t_1t_0(e)=(4,1,4,0,4)
			% ============================================================
			\begin{scope}[shift={(16,0)}]
				
				\duckboardgrid
				
				% Boundary for (4,1,4,0,4)
				\draw[very thick] (0,2) -- (1,2);
				\draw[very thick] (1,2) -- (1,5);
				\draw[very thick] (1,5) -- (2,5);
				\draw[very thick] (2,5) -- (2,2);
				\draw[very thick] (2,2) -- (3,2);
				\draw[very thick] (3,2) -- (3,6);
				\draw[very thick] (3,6) -- (4,6);
				\draw[very thick] (4,6) -- (4,2);
				\draw[very thick] (4,2) -- (5,2);
				
				% Occupied squares, except ducklings
				\foreach \n in {
					0,1,2,3,4,
					5,6,7,9,
					10,12,14,
					15,17,19
				}{
					\duckboardnumber{\n}
				}
				
				% t_1 moves the duckling 21 -> 20
				\duckboardsduck{20}
				\duckboardsduck{22}
				\duckboardsduck{24}
				
				\node at (2.5,-0.65)
				{$t_1t_0(\Vec e)=(4,1,4,0,4)$};
				
			\end{scope}
			
		\end{tikzpicture}
		
		\caption{Example of two moves in the duckling game. Under $t_0$,the duckling in square $25$ moves to $24$, while under $t_1$ the duckling in square $21$ moves to $20$. The shaded squares $15,16,17$ indicate the prescribed final positions.}
		
	\end{figure}
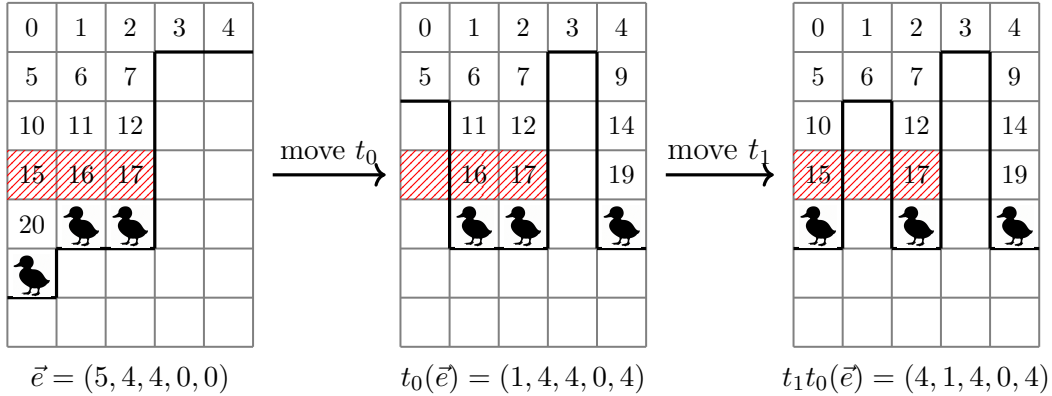
	The displacement  (cf. Remark~\ref{remark: southernomst description}) of the southernmost squares in $\textbf{a}(\Vec{f})$ is $0$. Therefore, a duckling stops when its corresponding displacement becomes zero. In this example, the leading duckling has displacement \(6\), so it moves from \(21\) to \(15\); since the order of the three ducklings is preserved, their final positions are \(15,16,17\). In fact, we can compute cardinalates of zero dimensional maximal Hurwitz--Brill--Noether loci in terms of a game of moving ducklings. The following is a consequence of Theorem~\ref{thm: hol-Euler-char} and Proposition~\ref{prop: ramification game}.
	
	\begin{lemma} \label{lemma: duckling}
		Let $\Vec{e}$ as above and consider the cardinality $m$ subset $S$ of $\{0, 1,\ldots, k-1\}$ defined as 
		\[
		S \coloneqq
		\begin{cases}
			\{0,\ldots,k_1-1\}\cup\{k-k_2,\ldots,k-1\},
			& \ \textrm{when} \ k_2>0,\\
			\{k-m,\ldots,k-1\},
			& \ \textrm{when} \ k_2=0.
		\end{cases}
		\]
		Moreover, we denote
		\[
		\Delta \coloneqq
		\begin{cases}
			(a-1)(k-m)-k_3, & \textrm{if } k_2>0,\\
			a(k-m)-k_3, & \textrm{if } k_2=0.
		\end{cases}
		\]
		We place ducklings on the line $\Z_{\geq 0}$ on the numbers in $S$, and move them to the $m$ consecutive positions 
		\[ \Delta + k - m, \ldots, \Delta + k - 1,\]
		allowing only moves according to the following rules: 
		\begin{enumerate}
			\item At each move, exactly one duckling advances by one position (moves from spot $i$ to $i+1$); 
			\item Two ducklings never occupy the same position;  
			\item The leading duckling (i.e. positioned in the highest number) never passes $\Delta + k-1$,; and 
			\item The distance between the leading and last duckling is always smaller than $k$. 
		\end{enumerate}
		
		Let $e(S, \Delta)$ be the number of ordered sequences of moves that reach the prescribed final configuration. Then we have the equality 
		\[ e(S, \Delta) = \chi_0(\Vec{e}).\]
	\end{lemma}
	\begin{proof}
		Let $\Vec{f}$ be the unique vector satisfying 
		\[ \textbf{a}(\Vec{f}) = (0, 1, \ldots, k-1 + \sum_{i=0}^{k-1}e_i).\]
		Using the formula 
		\[ \chi_0(\Vec e) = \sum_{t_i(\Vec e) \neq \Vec{e}}
		\chi_0\bigl(t_i(\Vec e)\bigr), \]
		together with $\chi_0\bigl(\Vec f\bigr) = 1$, we see inductively that $\chi_0(\Vec e)$ is the number of ordered sequences of moves $t_i$ that reach the prescribed final vector $\Vec{f}$. Thus we are left to identify the moves $t_i$ to the duckling moves. 
		
		We use the displacement description of Remark~\ref{remark: southernomst description}. We arrange the numbers in $\textbf{a}(\Vec{e})$ in a table as in Proposition~\ref{prop: ramification game} and look at the southernmost squares having non-zero displacement (see Remark~\ref{remark: southernomst description}, i.e. the ones corresponding to the first $m$ columns). We remark that the ducklings in the lemma have the same displacement numbers as the southernmost squares of $\Vec{e}$. We next compare the duckling moves to the moves induced by $t_i$ on southernmost squares. 
		
		In order for a transformation $t_i$ to move numbers $u, v$ to $u-1, v+1$ we need to have $u-v \geq k+1$. In particular, since the distance between the leading and last displaced southernmost square of $\Vec{e}$ is smaller than $k-1$, it follows that
		\begin{itemize}
			\item Every transformation $t_i$ is between a displaced southernmost square and a square with displacement $0$, and it reduces the displacement of the former by $1$.  
			\item The distance between first and last displaced southernmost squares can never be bigger than $k-1$; and consequently
			\item The order of the southernmost squares stays the same at each step and no two can be on the same position.  
		\end{itemize}
		
		Finally, the leading southernmost square cannot advance more than its displacement.
		
		In particular, the ducklings in the lemma and the displaced southernmost squares of $\Vec{e}$ have the same displacement and move according to the same rules, thus $e(S, \Delta) = \chi_0(\Vec{e})$, as claimed.
	\end{proof}
	
	\begin{figure}[H]
		\begin{tikzpicture}[x=1cm,y=1cm]
			
			% Squares
			\foreach \i in {0,...,10}{
				\draw (\i,0) rectangle (\i+1,1);
				\node[below] at (\i+0.5,0) {\(\i\)};
			}
			
			% Ducklings at 0,3,4
			\node[inner sep=0pt] at (0.5,0.52) {\duckiconTwo};
			\node[inner sep=0pt] at (3.5,0.52) {\duckiconTwo};
			\node[inner sep=0pt] at (4.5,0.52) {\duckiconTwo};
			
			% Displacement of the leading duckling
			\draw[<->,thick]
			(4.5,1.35) -- node[above] {\(\Delta=6\)} (10.5,1.35);
			
		\end{tikzpicture}
		\caption{Duckling game corresponding to Figure 4}
	\end{figure}
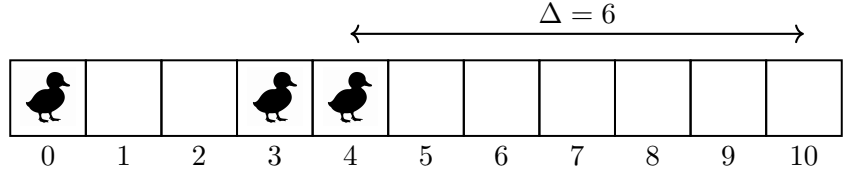
	Our next goal is to consider the case of gonality $k = 4$ and see how this duckling perspective can help us compute algebraic Euler characteristic. 
	
	\begin{lemma}
		Let $k = 4$. In the notation of Lemma~\ref{lemma: duckling} we have $e(\{2,3\}, n) = 2^{n-1}$. 
	\end{lemma}
	\begin{proof}
		We will compute this recursively. Since the two ducklings cannot sit on the same number, the first move is forced: the leading duckling has to move. This implies that $e(\{2,3\}, n)$ is the number of ways in which two ducklings sitting in positions $2$ and $4$ can reach reach $n+2, n+3$. By translating backwards, this is the number of ways in which two ducklings sitting in positions $1$ and $3$ can reach reach $n+1, n+2$. Hence we have 
		$$e(\{2,3\}, n) = e(\{1,3\}, n-1).$$
		
		When the ducklings are in positions $1$ and $3$, both of them can move, and we obtain, after translating backwards, the recursion formula 
		$$e(\{1,3\}, n) = e(\{2,3\}, n) + e(\{0,3\}, n-1).$$
		But if the ducklings are in position $0$ and $3$, the move is forced: the leading duck cannot advance, since the distance between the ducks is  always smaller than $k = 4$. This implies $$e(\{0,3\}, n-1) = e(\{1,3\}, n-1).$$
		
		In particular, we obtain $e(\{1,3\}, n) =2e(\{1,3\}, n-1)$ and because $e(\{1,3\}, 0) = 1$, we obtain $e(\{1,3\}, n) = 2^n$, hence the conclusion.
	\end{proof}
	
	This has implications for the computations of cardinalities of Hurwitz--Brill--Noether loci. 
	\begin{cor}
		Let $[\pi\colon C\to\PP^1,p]$ be a generic element of
		$\mathcal H^{\mathrm{tot}}_{g,4}$. For each residue class of $g$ modulo $4$, let $\Vec e$ be the vector displayed in the second column of the table below. Then $W^{\Vec e}(\pi,p)$ is finite, and its cardinality is given in the third column:
		{\large
			\[
			\begin{array}{c|c|c}
				g & \Vec e & \textrm{Cardinality of}\ \,W^{\Vec e}(\pi,p) \\ \hline 
				4h & (h+1,h+1,0,0) & 2^{2h-1}\\[4pt]
				4h+1 & (h+1,0,h+1,0) & 2^{2h}\\[4pt]
				4h-2 & (h+1,h+1,1,0) & 2^{2h-2}\\[4pt]
				4h-1 & (h+1,1,h+1,0) & 2^{2h-1}
			\end{array}
			\]
		}
	\end{cor}
	
	In fact, keeping track of the position of the displaced southernmost square is sufficient also when we want to compute the algebraic Euler characteristic for higher dimensional Hurwitz--Brill--Noether loci. 
	
	Let $\Vec{e}$ be a vector whose non-zero southernmost displacements are given by $\Delta + k-1-s$ for $s \in S$. We denote by $e_n(S, \Delta)$ the algebraic Euler characteristic $\chi_n(\Vec{e})$. Note that $e_n(S, \Delta)$ is well defined, because the vector $\Vec{e}$ is unique up to applications of $\sigma$ in Definition~\ref{def: sigma}. We have $\chi_n(\Vec{e}) = \chi_n(\sigma\Vec{e})$, since an application of $\sigma$ correspond with a translation by $\OO_C(p)$ of the pointed Hurwitz--Brill--Noether locus. 
	
	We will outline again how to compute algebraic Euler characteristics in gonality $4$. 
	
	\begin{lemma}
		Let $k = 4$ and $n\geq 1$. In the notation above, we have $e_1(\{2,3\}, n) = -(3n-1)2^{n-1}$.
	\end{lemma}
	\begin{proof}
		The formula in Theorem~\ref{thm: hol-Euler-char} implies 
		\[ e_1(\{2,3\}, n) = e_1(\{1,3\}, n-1) - e_0(\{2,3\}, n) = e_1(\{1,3\}, n-1) - 2^{n-1}. \]
		Furthermore, applying the same formula for $e_1(\{1,3\}, n)$, we obtain 
		\[ e_1(\{1,3\}, n) = e_1(\{2,3\}, n) + e_1(\{0,3\}, n-1) - 2  e_0(\{1,3\}, n) = e_1(\{2,3\}, n) + e_1(\{0,3\}, n-1) - 2^{n+1}. \]
		Similarly we obtain  
		\[ e_1(\{0,3\}, n) = e_1(\{1,3\}, n) - e_0(\{0,3\}, n) = e_1(\{1,3\}, n) - 2^n. \]
		
		In particular, out of the last two formulas, we obtain 
		\[ e_1(\{1,3\}, n) = e_1(\{2,3\}, n) + e_1(\{1,3\}, n-1) - 2^{n-1} - 2^{n+1}. \] 
		Therefore
		\[ e_1(\{1,3\}, n) = 2e_1(\{1,3\}, n-1) - 2^{n-1} - 2^{n-1} - 2^{n+1} = 2e_1(\{1,3\}, n-1) - 3 \cdot 2^n. \]
		Using the initial condition 
		\[ e_1(\{1,3\}, 0) = e_1(\{2,3\}, 0) - e_0(\{1,3\}, 0)  = 0 - 1=-1,\]
		we obtain 
		\[ e_1(\{1,3\}, n) = - (3n+1)2^n.\qedhere\]
	\end{proof}
	
	As a consequence of this result, we obtain the genus of maximal pointed Hurwitz--Brill--Noether loci, when their dimension is $1$ and the gonality is $4$. 
	
	\begin{cor}
		Let $[\pi\colon C\rightarrow \PP^1,p]$ be a generic element of
		$\mathcal H^{\mathrm{tot}}_{g,4}$. Depending on the residue of $g$
		modulo $4$, we have:
		\renewcommand{\arraystretch}{1.4}
		{\large
			\[
			\begin{array}{c|c|c}
				g & \Vec e & \text{genus of }W^{\Vec e}(\pi,p) \\ \hline
				4h+1
				&
				(h+1,h+1,0,0)
				&
				1+(6h-1)2^{2h-1}
				\\[4pt]
				
				4h+2
				&
				(h+1,0,h+1,0)
				&
				1+(6h+1)2^{2h}
				\\[4pt]
				
				4h-1
				&
				(h+1,h+1,1,0)
				&
				1+(6h-4)2^{2h-2}
				\\[4pt]
				
				4h
				&
				(h+1,1,h+1,0)
				&
				1+(6h-2)2^{2h-1}
			\end{array}
			\]
		}
	\end{cor}
	\section{Low codimension existence} \label{sec: expected codimension}
	For a nonnegative $k$-vector $\Vec e$, let
	\[
	\mathcal W^{\Vec e}_g
	\coloneqq
	\left\{
	[\pi:C\longrightarrow\mathbb P^1,p,L]
	\ \middle|\
	[\pi,p]\in\mathcal H^{\mathrm{tot}}_{g,k},
	\quad L\in W^{\Vec e}(\pi,p)
	\right\}.
	\]
	This is a locally closed subscheme of the relative Picard space
	of degree $g+k+|\Vec e|-1$ over
	$\mathcal H^{\mathrm{tot}}_{g,k}$. Its natural projection
	\[
	\mathcal W^{\Vec e}_g
	\longrightarrow
	\mathcal H^{\mathrm{tot}}_{g,k}
	\]
	forgets $L$ and has fiber $W^{\Vec e}(\pi,p)$ over $[\pi,p]$.
	We denote its image by $\mathcal H(\Vec e)$.
	
	We start with the case of a degree $k$ map $\pi\colon C \rightarrow \mathbb{P}^1$, with source curve of genus $g$, totally ramified at a point $p$. We assume that $k\geq g+1$ in this situation. Then, the Hurwitz space $\mathcal{H}_{g,k}^{\textrm{tot}}$ dominates $\mathcal{M}_{g,1}$ via the forgetful map. We are interested in finding pointed Hurwitz--Brill--Noether components $\mathcal{H}(\Vec{e})$ whose codimension in $\mathcal{H}_{g,k}^{\textrm{tot}}$ is the expected number $u(\Vec{e})+ \inv(\Vec{e})-g$. We can find these by induction, and we will use pullbacks of Brill--Noether loci in $\M_g$ as starting cases.
	
	\begin{prop}\label{prop: BN and HBN}
		Consider integers $g,r,d$ and $k$ satisfying the inequalities
		\[
		k\geq g+1,\qquad
		r\geq1,\qquad
		2r\leq d\leq g+r-2,
		\qquad
		3-g\leq\rho(g,r,d)\leq0,
		\]
		and set 
		\[
		\Vec e=
		\left(
		\underbrace{2,\ldots,2}_{r+1 \ \textrm{times}},
		\underbrace{1,\ldots,1}_{\ k-g-1+d-2r \ \textrm{times}},
		\underbrace{0,\ldots,0}_{\ g-d+r \ \textrm{times}}
		\right).
		\]
		Then $\mathcal W^{\Vec e}_g$ has an irreducible component of
		dimension $2g+k-3+g-u(\Vec{e})-\inv(\Vec{e})$, generically finite over its image in
		$\mathcal H^{\mathrm{tot}}_{g,k}$.
		This image lies in the pullback of $\mathcal M^r_{g,d}$ and has
		codimension $u(\Vec{e})+\inv(\Vec{e})-g$ in $\cHTot$.
		
		For $r=1$, the same conclusion holds under the weaker bound
		$2-g\leq\rho(g,1,d)\leq0$.
	\end{prop} 
	
	\begin{proof}  The entries of $\Vec e$ sum up to $|\Vec e|=k-g+d+1$. Consequently, its line bundles in $W^{\Vec e}_g$ have degree
		\[
		g+k+|\Vec e|-1=d+2k,
		\]
		and their prescribed vanishing sequence is
		\[\textbf{a}(\Vec e) = \bigl(0,1,\ldots,2k-g+d-r-1,\, 2k,2k+1,\ldots,2k+r \bigr). \]
		Moreover, since $\Vec{e}$ is decreasing, we have 
		\[ \inv(\Vec{e}) = 0 \ \textrm{and} \ u(\Vec{e}) = (r+1)(g-d+r), \]
		and hence $ g- \rho(g,r,d) = g - u(\Vec{e}) - \inv(\Vec{e})$. 
		
		In order to prove our result, we will use the Brill--Noether counterpart proved in \cite[Theorem A]{Pfl23} and pull it back to $\cHTot$. To do this, we will use the birational map 
		\[ W^{\Vec{e}}(\pi:C\to\mathbb P^1,p) \xrightarrow{\otimes \OO_C(-2kp)} W^{r,\circ}_d(C)\]
		where $W^{r,\circ}_d(C)$ parametrizes those line bundles $M$ satisfying $h^0(C,M) = r+1$. The line bundles in the image satisfy that $M$ and $\omega_C\otimes M^{-1}$ are both unramified at $p$, i.e. their vanishing sequences at $p$ are 
		\[
		(0,1,\ldots,r) \
		\text{and} \
		(0,1,\ldots,g-d+r-1), \ \textrm{respectively}.
		\]
		All these conditions ensure that the vanishing orders of $M(2kp)$ at $p$ are given by $\textbf{a}(\Vec{e})$.
		
		We consider the universal Brill--Noether space 
		\[ \mathcal W^{r, \circ}_{g,d} = \{[C, L] \ \mid \ C \in \M_g, \ L \in \textrm{Pic}^d(C), \ \textrm{and} \ h^0(C, L) = r+1\}.\]
		Thus twisting by $\mathcal O_C(2kp)$ identifies
		$\mathcal W^{\Vec e}_g$ with the open locus in
		\[
		\mathcal H^{\mathrm{tot}}_{g,k}
		\times_{\mathcal M_g}
		\mathcal W^{r,\circ}_{g,d}
		\]
		where $p$ is unramified
		for the line bundles in $W^{r,\circ}_{g,d}(C)$ and their duals.
		
		We now construct the required component. By \cite[Theorem A]{Pfl23},
		$\mathcal W^{r,\circ}_{g,d}$ has an irreducible component $Z$ of dimension $3g-3+\rho(g,r,d)$ generically finite over its image in $\mathcal M_g$. Next, we consider the Cartesian diagram 
		\[
		\begin{tikzcd}[column sep=large, row sep=large]
			\mathcal H^{\mathrm{tot, \circ}}_{g,k}
			\times_{\mathcal M_g} Z 
			\arrow[r]
			\arrow[d]
			\arrow[dr, phantom, "\ulcorner", very near start]
			&
			\mathcal M_{g,1} \times_{\mathcal M_g} Z 
			\arrow[r]
			\arrow[d]
			\arrow[dr, phantom, "\ulcorner", very near start]
			&
			Z
			\arrow[d]
			\\
			\mathcal H_{g,k}^{\mathrm{tot},\circ}
			\arrow[r]
			&
			\mathcal M_{g,1}^{\circ}
			\arrow[r]
			&
			\mathcal M_g
		\end{tikzcd}
		\]
		
		Here we denoted by $\M_{g,1}^{\circ}$ the open subspace of $\M_{g,1}$ parametrizing points $[C,p]$ where $p$ is not a Weierstrass point and is outside the ramification loci of
		$|M|$, $|\omega_C\otimes M^{-1}|$ for every $M$ in the fibre of $Z \rightarrow \M_g$ over $[C]$. Moreover, we denoted $\mathcal H_{g,k}^{\mathrm{tot},\circ}$ its preimage in the Hurwitz space.
		
		Let $p \in C$ outside the ramification locus of $\omega_C$ (i.e. not a Weierstrass point). Then, since $k\geq g+1$, the Riemann-Roch Theorem implies
		\[
		h^0(C,\mathcal O_C(kp))=k-g+1,
		\qquad
		h^0(C,\mathcal O_C((k-1)p))=k-g.
		\]
		The degree $k$ covering pencils totally ramified at $p$,
		up to automorphisms of the target, are parametrized by
		\[	\mathbb P\!\left(H^0(C,\mathcal O_C(kp))/\mathbb C \right) \setminus
		\mathbb P\!\left(H^0(C,\mathcal O_C((k-1)p))/\mathbb C\right).\]
		Indeed, such a pencil is spanned by the constant function
		and a function whose unique pole has order exactly $k$
		at $p$. This space is nonempty and has
		dimension $k-g-1$. Consequently, the forgetful morphism
		\[
		\mathcal H^{\mathrm{tot}, \circ}_{g,k}\longrightarrow\mathcal M_{g,1}^{\circ}
		\]
		has fibers of dimension $k-g-1$.
		
		In particular, we obtain 
		\[ \dim \mathcal H^{\mathrm{tot, \circ}}_{g,k}
		\times_{\mathcal M_g} Z  = 3g-2 + k - g- 1 + \rho(g,r,d) = 2g + k - 3 + g -u(\Vec{e}) - \inv(\Vec{e}). \]
		This is a component of $\mathcal{W}^{\Vec{e}}_g$ whose forgetful morphism to $\cHTot$ has generically finite fibers: for a generic point $[\pi\colon C \rightarrow \PP^1, p]$ in its image, the fiber consists of points 
		$[\pi\colon C \rightarrow \PP^1, p, L]\in \mathcal W^{\Vec e}_g$ for which we have $L(-2k\cdot p) \in W^{r,\circ}_d(C)$. By the choice of $Z$, this fiber consists of finitely many points. 
		
		Thus we obtained a component of $\mathcal{W}^{\Vec{e}}_g$ respecting the required conditions. 
		
		For the case $r=1$ the condition $2-g = \rho(g,1,d)$ implies $d=2$. It is well known that the hyperelliptic locus has codimension $g-2$ in $\M_g$, as expected. The remainder of the argument is unchanged.
	\end{proof}
	
	A similar result can be obtained for pointed Hurwitz--Brill--Noether subspaces of $\cHTot$ using \cite[Theorem 2.1]{Teixidor-BN} instead of \cite[Theorem A]{Pfl23}.
	
	\begin{prop}
		Consider integers $g,r,d$ and $k$ satisfying the inequalities
		\[
		k\geq g+1,\qquad
		r\geq1,\qquad
		2r\leq d\leq g+r-2,
		\qquad
		r+1-g\leq\rho(g,r,d)\leq0,
		\]
		and set 
		\[
		\Vec e=
		\left(
		\underbrace{2,\ldots,2}_{r+1 \ \textrm{times}},
		\underbrace{1,\ldots,1}_{\ k-g-1+d-2r \ \textrm{times}},
		\underbrace{0,\ldots,0}_{\ g-d+r \ \textrm{times}}
		\right).
		\]
		Then $\mathcal{H}_g(\Vec e)$ has an irreducible component of
		dimension $2g+k-3+g-u(\Vec{e})-\inv(\Vec{e})$, and for a generic element $[\pi,p]$, the locus $W^{\Vec{e}}(\pi,p)$ is $0$-dimensional. 
	\end{prop}
	
	\begin{proof}
		There exists $Z$ an irreducible component of $\M^r_{g,d}$ of expected dimension,  such that a generic element $[C] \in Z $ satisfies 
		\[ \dim W^r_d(C) = 0 \ \textrm{and}\ h^0(C,L) = r+1 \ \text{ for all } \ L \in  W^r_d(C). \]
		A generic element in the preimage of $Z$ to $\cHTot$ is contained in $\mathcal{H}_g(\Vec{e})$. Moreover, if the preimage is not a component, but a proper sublocus of a component, then the image of that component in $\M_g$ is an irreducible component of $\M^r_{g,d}$, properly containing $Z$, contradicting the choice of $Z$.
	\end{proof}
	
	Our main focus will be on the (pullback of the) hyperelliptic locus. We will use an inductive argument to reduce the existence of a pointed Hurwitz--Brill--Noether component of the expected dimension to the known fact that the hyperelliptic locus has expected dimension.
	
	\begin{prop}\label[proposition]{prop:low-codimension-existence} Let
		$\Vec e=(e_0,\ldots,e_{k-1})\in\mathbb Z_{\geq0}^k$ satisfy
		\[
		1\leq c \coloneqq u(\Vec e)+\operatorname{inv}(\Vec e)-g\leq k-3,
		\qquad
		|\Vec e|\equiv1-c\pmod{k}.
		\]
		Write
		\[
		b_0(\Vec{e})<b_1(\Vec{e})<\cdots<b_{k-1}(\Vec{e})
		\]
		for the increasing rearrangement of the integers $ke_j+j$, for $0\leq j \leq k-1$.
		Assume moreover that $b_{k-2}-b_c>k$. Then $\mathcal W_g^{\Vec e}$ has an irreducible component of
		dimension $2g+k-3-c$, generically finite over its image in
		$\mathcal H^{\mathrm{tot}}_{g,k}$.
	\end{prop}
	
	\begin{proof}
		Set $s=\frac{|\Vec e|+c-1-k}{k}$ and consider
		\[
		\Vec f=
		\bigl(s+2,s+2,
		\underbrace{s+1,\ldots,s+1}_{k-c-3 \ \textrm{times}},
		\underbrace{s,\ldots,s}_{c+1 \ \textrm{times}}\bigr).
		\]
		
		We will show that the locus $\mathcal{W}^{\vec{f}}_{c+2}$ has a component of the expected dimension generically finite over $\mathcal{H}^{\text{tot}}_{c+2,k}$. We then show that if $\mathcal{W}^{t_j(\vec{e})}_{g-1}$ has the required component, then so does $\mathcal{W}^{\vec{e}}_{g}$, where $t_j$ is an elementary shift (as defined in Definition~\ref{def: t_i}) such that $t_j(\vec{e})\neq\vec{e}$. And finally, we show that $\vec{e}$ can be reduced to $\vec{f}$ using $g-c-2$ elementary shifts.
		
		We begin with the base case, which is the pullback of the hyperelliptic locus.
		
		\begin{claim}\label{claim_basecase_pointed_HBN_comp_existence}
			The locus $\mathcal{W}^{\vec{f}}_{c+2}$ has a component of dimension $k+c+1$, as expected, and the forgetful map to $\mathcal H^{\mathrm{tot}}_{c+2,k}$ has generically finite fibers. 
		\end{claim}
		\begin{proof}[Proof of Claim~\ref{claim_basecase_pointed_HBN_comp_existence}]
			\renewcommand{\qedsymbol}{$\lozenge$}
			In the notations at the start of the proof, we consider the pointed Hurwitz--Brill--Noether space $\mathcal{W}^{\Vec{f}}_{c+2}$ in genus $c+2$. A direct computation gives
			\[
			|\Vec f|=|\Vec e|,
			\qquad
			u(\Vec f)=2(c+1),
			\qquad
			\operatorname{inv}(\Vec f)=0.
			\]
			Thus $\Vec f$ is the vector with which we shall start the
			induction, in genus $c+2$. Subtracting $s$ from every entry does not change either
			$u + \inv$ or the relevant exact incidence problem, up to twisting the
			line bundle by the pullback of a line bundle on $\mathbb P^1$.
			Using Proposition~\ref{prop: BN and HBN}, the locus $\mathcal W^{\Vec{f}}_{c+2}$ has a component of dimension $k+c+1$, as expected, and the forgetful map to $\mathcal H^{\mathrm{tot}}_{c+2,k}$ has generically finite fibers, proving the claim.
		\end{proof}
		
		We now explain the inductive strategy.
		
		\begin{claim}\label{claim_induct_pointed_HBN_comp_existence}
			Suppose that $t_j(\Vec e)\neq\Vec e$ and that, in genus $g-1$,
			the space $\mathcal W^{t_j(\Vec e)}_{g-1}$
			has an irreducible component $Z$ of dimension $2(g-1)+k-3 + \left(g-u(\Vec{e})-\inv(\Vec{e})\right)$ whose projection to $\mathcal H^{\mathrm{tot}}_{g-1,k}$ is
			generically finite onto its image. Then, the scheme $\mathcal{W}^{\vec{e}}_g$ has an irreducible component of dimension $3g-3-u(\vec{e})-\inv(\vec{e})$ whose projection to $\mathcal H^{\mathrm{tot}}_{g,k}$ is generically finite onto its image.
		\end{claim}
		\begin{proof}[Proof of Claim~\ref{claim_induct_pointed_HBN_comp_existence}]
			\renewcommand{\qedsymbol}{$\lozenge$}
			
			By Proposition~\ref{prop: ramification game}, we have
			\[
			u(t_j(\Vec e))+\inv(t_j(\Vec e))
			=
			u(\Vec e)+\inv(\Vec e)-1.
			\]
			Therefore
			\[
			u(t_j(\Vec e))+\inv(t_j(\Vec e))-(g-1)=c,
			\]
			so the value of the expected codimension is unchanged when we decrease the genus by
			one and apply a nontrivial elementary shift.
			
			We show that the conclusion of the theorem then holds for
			$\Vec e$ in genus $g$. Choose a general point $[\pi_X:X\longrightarrow\PP^1,q,L]\in Z$
			and attach an elliptic
			curve $E$ at $q$, with a marked point $p$ such that $p-q$ has order exactly $k$. In this way, we obtain an admissible cover
			\[
			\pi_0:X\cup_qE\longrightarrow \PP^1\cup\PP^1,
			\]
			totally ramified at point $p$, and we consider
			\[
			\mathcal C\longrightarrow\mathcal P\longrightarrow B,
			\]
			a versal deformation around a sufficiently small neighborhood $B$ of $[\pi_0,p]$ in
			$\overline{\mathcal H}^{\mathrm{tot}}_{g,k}$. Thus $B$ is smooth of dimension $2g+k-3$, and we denote by $\Delta\subset B$ the divisor where the curve of genus $g$ is a genus $g-1$ curve glued to an elliptic curve. 
			
			Theorem~\ref{theorem: nested regeneration} can be extended to arbitrary families of admissible $k:1$ covers with central fibers in the boundary divisor $\Delta$. In particular, we look at the loci of line bundles admitting $\textrm{ord}(\Vec{e})$-nested limit linear series with extra vanishing conditions given by the ordering of elements in $\Vec{e}$ (cf. Proposition~\ref{prop: lowerbound}, first proof)
			\[\left(\mathcal{O}_{\mathcal{X}}(q)\otimes \mathcal{W}_{g-1}^{\Vec{e}}\right)\times \mathrm{Pic}^0(\mathcal{E}\rightarrow \Delta), \] and
			\[ \left(\OO_\mathcal{X}(q) \otimes \mathcal{W}_{g-1}^{t_i(\Vec{e})} \right) \times \{ \mathcal{O}_{\mathcal{E}}(i\cdot p -i\cdot q)\} \text{ for $0\leq i \leq k-1$ satisfying $t_i(\Vec{e}) \neq \Vec{e}$.} \]
			
			By Lemma~\ref{lemma: nested description} and Lemma~\ref{lem: nested description}, all elements in these loci contain a unique $\textrm{ord}(\Vec{e})$-nested linear series. Since 
			\[\left(\OO_\mathcal{X}(q) \otimes Z \right) \times \{ \mathcal{O}_{\mathcal{E}}(j\cdot p -j\cdot q)\} \]
			has dimension 
			$$2(g-1) + k -3 + g -u(\Vec{e})-\inv(\Vec{e}) + 1 = 3g-4 + k -u(\Vec{e})-\inv(\Vec{e}),$$
			a similar Regeneration Theorem to Theorem~\ref{theorem: nested regeneration} implies it is contained in the closure of $\mathcal{W}^{\Vec{e}}_g$. We denote $Y$ the irreducible component of $\mathcal{W}^{\Vec{e}}_g$ containing this locus in its closure. We consider the forgetful map $Y \rightarrow \overline{\mathcal{H}}^{\textrm{tot}}_{g,k}$ and notice that the fiber over $[\pi_0,p]$ has a $0$-dimensional component (because of the choice of $Z$). It follows that the generic fiber has dimension $0$. On the other hand 
			\[ \left(\OO_\mathcal{X}(q) \otimes Z \right) \times \{ \mathcal{O}_{\mathcal{E}}(j\cdot p -j\cdot q)\} \]
			is cut out from $Y$ by the divisorial condition $t_1=0$, implying 
			\[ \dim Y = 1 + \dim \left(\OO_\mathcal{X}(q) \otimes Z \right) \times \{ \mathcal{O}_{\mathcal{E}}(j\cdot p -j\cdot q)\} = \dim Z + 2,\]
			as required. Consequently, $\mathcal{W}^{\Vec{e}}_g$ has an irreducible component of dimension $3g-3+k -u(\Vec{e})-\inv(\Vec{e})$ whose projection to $\cHTot$ is generically finite onto its image.  This proves the claim.
		\end{proof}

		To apply the inductive argument, it remains only to prove that $\Vec e$ can be transformed into $\Vec f$ by nontrivial elementary shifts as in Definition~\ref{def: t_i}. 
		
		Given a vector $\vec{d}=(d_0,\dots,d_{k-1})\in\Z_{\geq 0}^k$, we will denote by $b_0(\vec{d})<\cdots < b_{k-1}(\vec{d})$ the increasing rearrangement of the integers $kd_j+j$ for $0\leq j\leq k-1$, which are the bottom entries of the boxes given by $\vec{d}$, as in Figure~\ref{fig:54010_vanishing}.
		
		\begin{claim}\label{claim_transforms_pointed_HBN_comp_existence} There is a finite sequence of elementary shifts $t_{j_1}^{(1)},\dots,t_{j_{g-c-2}}^{(g-c-2)}$ such that \[\vec{f}= t_{j_{g-c-2}}^{(g-c-2)}\circ\cdots\circ t_{j_1}^{(1)} (\vec{e}),\] and at each intermediate step (for every $0\leq \ell \leq g-c-2$) denoting $\vec{d}=t_{j_\ell}^{(\ell)}\circ \cdots \circ t_{j_1}^{(1)} (\vec{e})$ and $t_{j_0}^{(0)}=\operatorname{id}$, we have $t_{j_{\ell+1}}^{(\ell+1)}(\vec{d})\neq \vec{d}$ and the conditions 
			\begin{enumerate}[label={\normalfont (\roman*)}]
				\item $b_{k-2}(\vec{d})-b_c(\vec{d})>k$,
				\item $1\leq c=u(\vec{d})+\inv(\vec{d})-(g-\ell)\leq k-3$, and
				\item $|\vec{d}|\equiv 1-c \pmod k$
			\end{enumerate} are satisfied.
		\end{claim}
		\begin{proof}[Proof of Claim~\ref{claim_transforms_pointed_HBN_comp_existence}]
			\renewcommand{\qedsymbol}{$\lozenge$}

			Starting from $\vec{d}$, applying any non-trivial $t_j$ decreases $u(\vec{d})+\operatorname{inv}(\vec{d})$ by $1$ as observed above. Thus, we see that conditions
			\[
			1\leq c \coloneqq u(\Vec d)+\operatorname{inv}(\Vec d)-(g-\ell)\leq k-3, \
			\textrm{and} \
			|\Vec d|\equiv1-c\pmod{k}
			\]
			are automatically satisfied, as the sum of the entries is unchanged after applying $t_j$ and since we decrease $g$ by one, we have \[u(\Vec d)+\operatorname{inv}(\Vec d)-(g-\ell)=u(t_j(\vec{d})+\inv(t_j(\vec{d}))-(g-\ell-1).\]
			
			We turn our focus to (i), namely that we can choose some $t_{j}$ to preserve the inequality
			\begin{equation}\label{eq:72-gap}
				b_{k-2}(\vec{d})-b_c(\vec{d})>k.
			\end{equation}

			For a vector $(d_0,\ldots,d_{k-1})$, the endpoint in column $j$ is numbered $kd_j+j$. If $1\leq j\leq k-1$ and $d_j>d_{j-1}$, the transformation $t_j$
			replaces the two endpoints numbered
			\[
			kd_{j-1}+j-1,\qquad kd_j+j
			\]
			with
			\[
			kd_{j-1}+j,\qquad kd_j+j-1.
			\]
			Thus the larger of these two numbers moves one square to the left
			and the smaller moves one square to the right.
			Similarly, when $t_0$ is nontrivial
			(i.e.\ $d_0>d_{k-1}+1$), the two numbers
			\[
			kd_{k-1}+k-1,\qquad kd_0
			\]
			are replaced by
			\[
			kd_{k-1}+k,\qquad kd_0-1.
			\]
			Also in this case, the larger of these two numbers moves one square to the left and the smaller moves one square to the right, cyclically.
			
			Note that at each step, at most one nontrivial transformation can fail to preserve
			\eqref{eq:72-gap}. Indeed, the difference in \eqref{eq:72-gap} changes by at most two after each application of a $t_j$ and can never
			equal $k$, because distinct endpoints have distinct residues modulo $k$.
			Hence \eqref{eq:72-gap} can fail only when
			\[
			b_{k-2}(\vec{d})-b_c(\vec{d})=k+1.
			\]
			and the transform replaces these two endpoints by
			$b_{k-2}(\vec{d})-1$ and $b_c(\vec{d})+1$. Their residues uniquely determine
			that transformation.
			
			We now show that the required $t_{j_\ell}$ exist.
			
			Starting from $\Vec e$, perform any nontrivial transformation
			preserving \eqref{eq:72-gap}, as long as one exists.
			Every move decreases $u+\operatorname{inv}$ by one, thus the
			procedure terminates eventually. At a stopping vector there is still a nontrivial transformation that can be applied. Indeed, if there are no non-trivial $t_j$, then the vector would be decreasing with its largest and
			smallest entries differing by at most one, and its endpoints
			would be $k$ consecutive integers, contradicting
			\eqref{eq:72-gap}. Thus at most one transformation can violate
			\eqref{eq:72-gap}, and therefore the stopping vector has exactly one
			nontrivial transformation. It is the transformation changing $b_{k-2}$, and $b_c$ into $b_{k-2}-1$, and $b_c+1$, respectively, as described above. 
			
			It remains to show that the number of $t_j$ is exactly $g-c-2$, and that the stopping vector is $\vec{f}$.
			
			Now let $\vec{d}$ be the stopping vector. Suppose the unique transform not preserving \eqref{eq:72-gap} is $t_j$, with $0\leq j<k$. Temporarily apply $\sigma^j$ (cf. Definition~\ref{def: sigma}) to the stopping vector. This translates all endpoints by $-j$, preserves their ordering,
			and makes its unique nontrivial transform $t_0$.
			This is only a normalization for the following numerical argument; no nonnegativity assumption is needed for the normalized entries.
			
			Since all transformations $t_i$ for $1\leq i \leq k-1$ are trivial, the normalized vector $\vec{n}=\sigma^j\vec{d}$ is decreasing, i.e. $n_0\geq n_1\geq\cdots\geq n_{k-1}$. For a decreasing vector, $t_0$ acts on the endpoint in the
			first column and that in the last column. Here they are the
			endpoints of ranks $k-2$ and $c$, respectively, and their
			difference is $k+1$, i.e. $b_c(\vec{n})=kn_{k-1}+k-1$, $b_{k-2}(\vec{n})=kn_{0}$, and $kn_{0}-(kn_{k-1}+k-1)=k+1$. Immediately we have \[n_0-n_{k-1}=2,\] whereby $n_0-n_i\leq 2$ for all $0\leq i \leq k-1$. We observe that $b_{k-2}(\vec{n})=kn_0$
			means that only $b_{k-1}(\vec{n})\geq b_{k-2}(\vec{n})$, hence there is exactly one $0\leq i \leq k-1$ such that $b_{k-1}(\vec{n})=kn_i+i>kn_0=b_{k-2}(\vec{n})$. If $i\geq 2$, then $k(n_0-n_i)<i\leq k-1$, so $n_0=n_i$. Since $\vec{n}$ is decreasing and there is only one such $i$, $i=1$ and we have $n_0=n_1>n_i\geq n_0-2$ for $2\leq i\leq k-1$.

			Since $b_c(\vec{n})=kn_{k-1}+k-1$, there are $c$ bottom entries $b_0(\vec{n})=kn_{i_{1}}+i_{1}<\cdots<b_{c-1}(\vec{n})=kn_{i_c}+i_c$ less than $b_c(\vec{n})=kn_{k-1}+k-1$. Suppose that $n_{i_j}\neq n_{k-1}$, then by the above, $n_{i_j}=n_{k-1}+1$, so \[kn_{i_j}+i_j=k(n_{k-1}+1)+i_j=kn_{k-1}+k+i_j<kn_{k-1}+k-1,\] which gives $i_j<0,$ a contradiction. Thus $n_{i_j}=n_{k-1}$, and we have \[\vec{n} = (n_0,n_0,\underbrace{n_0-1,\dots,n_0-1}_{k-c-3 \text{ times}},\underbrace{n_0-2,\dots,n_0-2}_{c+1 \text{ times}}).\] Directly computing, we see that $|\vec{n}|\equiv 1-c\pmod k$, and we also have $\vec{n}=\sigma^j\vec{d}$, and \[|\vec{n}|=|\sigma^j\vec{d}|=|\vec{d}|-j=|\vec{e}|-j\equiv 1-c-j\pmod k.\] 
			Consequently
			\[
			1-c-j\equiv1-c\pmod{k}.
			\]
			Since $0\leq j<k$, we have $j=0$.
			Thus the stopping vector is $\vec{n}$ itself. As \[k(n_0-2)+4+k-c-3=|\vec{n}|=|\vec{e}|=ks+k+1-c,\] we have $s=n_0-2$ and $\vec{n}=\vec{f}$, as desired.
			
			The number of nontrivial transforms in this reduction is
			therefore
			\[
			u(\Vec e)+\operatorname{inv}(\Vec e)
			-u(\Vec f)-\operatorname{inv}(\Vec f)
			=g+c-2(c+1)=g-c-2.\qedhere
			\]
		\end{proof}
		
		Finally, we are in a position to use the induction to conclude. Starting with the component constructed in Claim~\ref{claim_basecase_pointed_HBN_comp_existence} and using the induction of Claim~\ref{claim_induct_pointed_HBN_comp_existence} for each of the $g-c-2$ transforms of Claim~\ref{claim_transforms_pointed_HBN_comp_existence}, we obtain a component of
		$\mathcal W_g^{\Vec e}$ of dimension
		\[
		(k+c+1)+2(g-c-2)=2g+k-3-c,
		\]
		generically finite over its image in $\cHTot$, as desired.
	\end{proof}
	
	Notice that condition $|\Vec e|\equiv1-c\pmod{k}$ is redundant. Indeed, if we take as in Definition~\ref{def: sigma}
	\[ \sigma^{-1}\Vec{e} = (e_{k-1}+1, e_0, e_1,\ldots, e_{k-2}),\]
	we have $W^{\sigma^{-1}\Vec{e}}(\pi,p) = \OO_C(p) \otimes W^{\Vec{e}}(\pi,p)$. In other words, for any $[\pi,p] \in \cHTot$, these loci are the same up to twisting with $\OO_C(p)$. In particular, in Proposition~\ref{prop:low-codimension-existence} we can take $ c$ to be the expected codimension $u(\Vec{e})+ \inv(\Vec{e}) - g$ and then apply $\sigma^{-1}$ enough times until the new vector satisfies the congruence in the hypothesis of Proposition~\ref{prop:low-codimension-existence}. We thus immediately obtain the following.
	
	\begin{theorem}\label{thrm:low-codimension-existence}
		
		Let $g \geq 1$ and let $\Vec e=(e_0,\ldots,e_{k-1})\in\mathbb Z_{\geq0}^k$ satisfy
		\[
		1\leq c \coloneqq u(\Vec e)+\operatorname{inv}(\Vec e)-g\leq k-3.
		\]
		Write $b_0(\Vec{e})<b_1(\Vec{e})<\cdots<b_{k-1}(\Vec{e})$ for the increasing rearrangement of the integers $ke_j+j$, for $0\leq j \leq k-1$.
		Assume moreover that $b_{k-2}(\Vec{e})-b_c(\Vec{e})>k$. Then $\mathcal W_g^{\Vec e}$ has an irreducible component of
		dimension $2g+k-3 +(g-u(\Vec{e})-\inv(\Vec{e}))$, generically finite over its image in
		$\mathcal H^{\mathrm{tot}}_{g,k}$.
	\end{theorem}
	
	The expected codimension $k-3$ is optimal. We will next construct a locus of expected codimension $k-2$, that only has an irreducible component of codimension $k-3$. For this, we will consider the case $g = k-1$ and consider the dominant morphism $\cHTot \rightarrow \M_{g,1}$, with generically finite fibers. Inside $\M_{g,1}$ we consider the locus $\mathcal{H}_g(2g-2)$ parametrizing pointed curves $[C,p]$ with the point being subcanonical, i.e. $\omega_C \cong \OO_C\left((2g-2) p\right)$.
	
	This strata of differentials $\mathcal{H}_g(2g-2)$ has exactly three irreducible components: $\mathcal{H}^{\textrm{hyp}}_g(2g-2)$, $\mathcal{H}^{\textrm{odd}}_g(2g-2)$ and $\mathcal{H}^{\textrm{even}}_g(2g-2)$, see~\cite{Kon-Zor}. For a generic element $[C,p]$ in each of these components, the ramification orders of $\omega_C$ at $p$ are explicitly known, see~\cite{Bul13}, and are given as 
	\begin{itemize}
		\item $(0,2,4, \ldots, 2g-2)$ for the component $\mathcal{H}^{\textrm{hyp}}_g(2g-2)$; 
		\item $(0,1,\ldots, g-2, 2g-2)$ for the component $\mathcal{H}^{\textrm{odd}}_g(2g-2)$, and  
		\item $(0,1,\ldots, g-3, g-1, 2g-2)$ for the component $\mathcal{H}^{\textrm{even}}_g(2g-2)$. 
	\end{itemize}
	
	But knowing all the vanishing orders uniquely determines the vector of the corresponding Hurwitz--Brill--Noether locus, giving us the following results. 
	
	\begin{lemma}
		Let $[C,p]$ be a generic element in a component of $\mathcal{H}_g(2g-2)$ and let $\pi\colon C \rightarrow \PP^1$ be a degree $(g+1)$ cover, totally ramified at the point $p$. Then we have 
		\begin{enumerate}
			\item If $[C,p] \in \mathcal{H}^{\textrm{odd}}_g(2g-2)$, then $\omega_C(2(g+1)p) \in W^{\Vec{e}_{\textrm{odd}}}(\pi,p)$ where 
			\[ \Vec{e}_{\textrm{odd}} = (\underbrace{0,\ldots, 0}_{(g-3) \ \textrm{times}},1,0, -1, -2)+(2,2\dots,2)\]
			\item If $[C,p] \in \mathcal{H}^{\textrm{even}}_g(2g-2)$, then $\omega_C(2(g+1)p) \in W^{\Vec{e}_{\textrm{even}}}(\pi,p)$ where 
			\[ \Vec{e}_{\textrm{even}} = (\underbrace{0,\ldots, 0}_{(g-3) \ \textrm{times}},1,-1, 0, -2)+(2,2\dots,2)\]
		\end{enumerate}
	\end{lemma}

	\begin{cor}
		Let $g\geq 4$ and $k = g+1$. Then the locus $\mathcal{W}^{\Vec{e}_{\textrm{even}}}_g$ is irreducible of dimension $2g$. This is one more than the expected dimension.
	\end{cor}
	
	\begin{proof}
		A simple calculation gives 
		\[ u(\Vec{e}_{\textrm{even}}) = g+1, \ \textrm{and} \ \inv(\Vec{e}_{\textrm{even}}) = g-2. \] 
		
		In particular, the expected codimension of $\mathcal{W}^{\Vec{e}_{\textrm{even}}}_g$ in $\cHTot$ is $g-1$. But the fiber of the morphism $\cHTot \rightarrow \M_{g,1}$ above a point $[C,p]$ can be identified with  \[	\mathbb P\!\left(H^0(C,\mathcal O_C((g+1)p))/\mathbb C \right) \setminus
		\mathbb P\!\left(H^0(C,\mathcal O_C(gp))/\mathbb C\right)\]
		and this is one dimensional for $[C,p]$ a generic element of $\mathcal{H}^{\textrm{even}}_g(2g-2)$. So the generic fiber of $\cHTot \rightarrow \M_{g,1}$ over $\mathcal{H}^{\textrm{even}}_g(2g-2)$ is one dimensional. This implies 
		\[   \dim \mathcal{W}^{\Vec{e}_{\textrm{even}}}_g = \dim \mathcal{H}^{\textrm{even}}_g(2g-2) + 1 = 2g. \]
		But, when $k = g+1$ we have that $\dim \cHTot = 3g-2$, so the pointed Hurwitz-Brill-Noether space has dimension one bigger than expected. 
	\end{proof}
	\bibliographystyle{alpha}
	\bibliography{biblico}

@article{AbramovichCV,
  author    = {Abramovich, D. and Corti, A. and Vistoli, A.},
  title     = {Twisted bundles and admissible covers},
  journal   = {Comm. Algebra},
  volume    = {31},
  number    = {8},
  pages     = {3547--3618},
  year      = {2003},
  doi       = {10.1081/AGB-120022434},
}

@article{Agg26,
  author    = {Aggarwal, D.},
  title     = {{Brill--Noether} theory for totally ramified covers of the projective line},
  journal   = {Preprint, arXiv:2604.26193},
  year      = {2026},
}

@article{Bul13,
  author    = {Bullock, E. M.},
  title     = {Subcanonical points on algebraic curves},
  journal   = {Trans. Amer. Math. Soc.},
  volume    = {365},
  number    = {1},
  pages     = {99--122},
  year      = {2013},
  doi       = {10.1090/S0002-9947-2012-05506-1},
}

@article{Larson2Vogt,
  author    = {Larson, E. and Larson, H. and Vogt, I.},
  title     = {Global {Brill--Noether} theory over the {Hurwitz} space},
  journal   = {Geom. Topol.},
  volume    = {29},
  number    = {1},
  pages     = {193--257},
  year      = {2025},
  doi       = {10.2140/gt.2025.29.193},
}

@article{basiclimitlinear,
  author    = {Eisenbud, D. and Harris, J.},
  title     = {Limit linear series: Basic theory},
  journal   = {Invent. Math.},
  volume    = {85},
  number    = {2},
  pages     = {337--371},
  year      = {1986},
  doi       = {10.1007/BF01389094},
}

@incollection{Hoc16,
  author    = {Hochster, M.},
  title     = {{Cohen--Macaulay} varieties, geometric complexes, and combinatorics},
  booktitle = {The mathematical legacy of {Richard P. Stanley}},
  pages     = {219--229},
  publisher = {Amer. Math. Soc.},
  address   = {Providence, RI},
  year      = {2016},
  doi       = {10.1090/mbk/100/13},
}

@article{Kon-Zor,
  author    = {Kontsevich, M. and Zorich, A.},
  title     = {Connected components of the moduli spaces of {Abelian} differentials with prescribed singularities},
  journal   = {Invent. Math.},
  volume    = {153},
  number    = {3},
  pages     = {631--678},
  year      = {2003},
  doi       = {10.1007/s00222-003-0303-x},
}

@article{Pfl23,
  author    = {Pflueger, N.},
  title     = {Linear series with {$\rho<0$} via thrifty {Lego} building},
  journal   = {J. Reine Angew. Math.},
  volume    = {797},
  pages     = {193--228},
  year      = {2023},
  doi       = {10.1515/crelle-2023-0001},
}

@article{harris_mumford,
  author    = {Harris, J. and Mumford, D.},
  title     = {On the {K}odaira dimension of the moduli space of curves},
  journal   = {Invent. Math.},
  volume    = {67},
  number    = {1},
  pages     = {23--86},
  year      = {1982},
  doi       = {10.1007/BF01393371},
}

@article{cook-powell_jensen,
  author    = {Cook-Powell, K. and Jensen, D.},
  title     = {Components of {Brill--Noether} loci for curves with fixed gonality},
  journal   = {Michigan Math. J.},
  volume    = {71},
  number    = {1},
  pages     = {19--45},
  year      = {2022},
  doi       = {10.1307/mmj/1600329610},
}

@article{Teixidor-BN,
  author    = {Teixidor i Bigas, M.},
  title     = {{Brill--Noether} loci},
  journal   = {Manuscripta Math.},
  volume    = {176},
  number    = {1},
  pages     = {Paper No. 14, 17 pp.},
  year      = {2025},
  doi       = {10.1007/s00229-025-01616-z},
}

@article{Osserman1,
  author    = {Osserman, B.},
  title     = {Dimension counts for limit linear series on curves not of compact type},
  journal   = {Math. Z.},
  volume    = {284},
  number    = {1--2},
  pages     = {69--93},
  year      = {2016},
  doi       = {10.1007/s00209-016-1646-5},
}

@article{Osserman2,
  author    = {Osserman, B.},
  title     = {Limit linear series for curves not of compact type},
  journal   = {J. Reine Angew. Math.},
  volume    = {753},
  pages     = {57--88},
  year      = {2019},
  doi       = {10.1515/crelle-2017-0003},
}

@article{Eisenbud_Schreyer_2008,
  author    = {Eisenbud, D. and Schreyer, F.-O.},
  title     = {Relative {Beilinson} monad and direct image for families of coherent sheaves},
  journal   = {Trans. Amer. Math. Soc.},
  volume    = {360},
  number    = {10},
  pages     = {5367--5396},
  year      = {2008},
  doi       = {10.1090/S0002-9947-08-04454-1},
}

@article{GriffithsHarris1980,
  author    = {Griffiths, P. A. and Harris, J.},
  title     = {On the variety of special linear systems on a general algebraic curve},
  journal   = {Duke Math. J.},
  volume    = {47},
  number    = {1},
  pages     = {233--272},
  year      = {1980},
  doi       = {10.1215/S0012-7094-80-04717-1},
}

@article{Gieseker1982,
  author    = {Gieseker, D.},
  title     = {Stable curves and special divisors: {P}etri's conjecture},
  journal   = {Invent. Math.},
  volume    = {66},
  number    = {2},
  pages     = {251--275},
  year      = {1982},
  doi       = {10.1007/BF01389394},
}

@article{FultonLazarsfeld1981,
  author    = {Fulton, W. and Lazarsfeld, R.},
  title     = {On the connectedness of degeneracy loci and special divisors},
  journal   = {Acta Math.},
  volume    = {146},
  pages     = {271--283},
  year      = {1981},
  doi       = {10.1007/BF02392466},
}

@article{Larson2021,
  author    = {Larson, H. K.},
  title     = {A refined {Brill--Noether} theory over {Hurwitz} spaces},
  journal   = {Invent. Math.},
  volume    = {224},
  number    = {3},
  pages     = {767--790},
  year      = {2021},
  doi       = {10.1007/s00222-020-01023-z},
}

@article{CookPowellJensen2022Tropical,
  author    = {Cook-Powell, K. and Jensen, D.},
  title     = {Tropical methods in {Hurwitz--Brill--Noether} theory},
  journal   = {Adv. Math.},
  volume    = {398},
  pages     = {Paper No. 108199, 42 pp.},
  year      = {2022},
  doi       = {10.1016/j.aim.2022.108199},
}

@article{Pflueger2026Transmission,
  author    = {Pflueger, N.},
  title     = {Transmission permutations and {Demazure} products in {Hurwitz--Brill--Noether} theory},
  journal   = {Preprint, arXiv:2604.03895},
  year      = {2026},
}
	
\end{document}